\documentclass[a4paper,11pt,oneside]{article}
\usepackage{graphicx}
\usepackage[english]{babel}
\usepackage{amsthm,amssymb}
\usepackage{mathrsfs,amscd}
\usepackage[colorinlistoftodos]{todonotes}
\usepackage{amsmath}
\usepackage{amsfonts}
\usepackage{geometry}
\usepackage{makeidx}
\usepackage[colorlinks=false, linkcolor=blue]{hyperref}
\usepackage{float}
\usepackage{setspace}
\usepackage{fancyhdr}

\DeclareFontFamily{OT1}{ptm}{}
\DeclareFontShape{OT1}{ptm}{m}{n} { <-> ptmr}{}
\DeclareFontShape{OT1}{ptm}{m}{it}{ <-> ptmri}{}
\DeclareFontShape{OT1}{ptm}{m}{sl}{ <->ptmro}{}
\DeclareFontShape{OT1}{ptm}{m}{sc}{ <-> ptmrc}{}
\DeclareFontShape{OT1}{ptm}{b}{n} { <-> ptmb}{}
\DeclareFontShape{OT1}{ptm}{b}{it}{ <-> ptmbi}{}
\DeclareFontShape{OT1}{ptm}{bx}{n} {<->ssub * ptm/b/n}{}
\DeclareFontShape{OT1}{ptm}{bx}{it}{<->ssub * ptm/b/it}{}

\DeclareSymbolFont{bold}{OT1}{ptm}{b}{n}
\DeclareMathAlphabet{\mathbf}{OT1}{ptm}{b}{n}
\DeclareMathAlphabet{\mathrm}{OT1}{ptm}{m}{n}

\DeclareFontFamily{OT1}{psy}{}
\DeclareFontShape{OT1}{psy}{m}{n}{ <-> s * [0.9] psyr}{}
\DeclareFontFamily{OMS}{ptm}{}
\DeclareFontShape{OMS}{ptm}{m}{n}{ <8> <9> <10> gen * cmsy }{}
\DeclareFontFamily{OMS}{cmtt}{}
\DeclareFontShape{OMS}{cmtt}{m}{n}{ <8> <9> <10> gen * cmsy }{}

\SetSymbolFont{operators}{normal}{OT1}{ptm}{m}{n}
\SetSymbolFont{operators}{bold}{OT1}{ptm}{b}{n}
\DeclareSymbolFont{emsy}{OT1}{ptm}{m}{it}
\DeclareSymbolFont{emsr}{OT1}{ptm}{m}{n}
\DeclareSymbolFont{emcmr}{OT1}{cmr}{m}{n}
\DeclareSymbolFont{emsymb}{OT1}{psy}{m}{n}
\DeclareMathSymbol a{\mathalpha}{emsy}{"61}
\DeclareMathSymbol b{\mathalpha}{emsy}{"62}
\DeclareMathSymbol c{\mathalpha}{emsy}{"63}
\DeclareMathSymbol d{\mathalpha}{emsy}{"64}
\DeclareMathSymbol e{\mathalpha}{emsy}{"65}
\DeclareMathSymbol f{\mathalpha}{emsy}{"66}
\DeclareMathSymbol g{\mathalpha}{emsy}{"67}
\DeclareMathSymbol h{\mathalpha}{emsy}{"68}
\DeclareMathSymbol i{\mathalpha}{emsy}{"69}
\DeclareMathSymbol j{\mathalpha}{emsy}{"6A}
\DeclareMathSymbol k{\mathalpha}{emsy}{"6B}
\DeclareMathSymbol l{\mathalpha}{emsy}{"6C}
\DeclareMathSymbol m{\mathalpha}{emsy}{"6D}
\DeclareMathSymbol n{\mathalpha}{emsy}{"6E}
\DeclareMathSymbol o{\mathalpha}{emsy}{"6F}
\DeclareMathSymbol p{\mathalpha}{emsy}{"70}
\DeclareMathSymbol q{\mathalpha}{emsy}{"71}
\DeclareMathSymbol r{\mathalpha}{emsy}{"72}
\DeclareMathSymbol s{\mathalpha}{emsy}{"73}
\DeclareMathSymbol t{\mathalpha}{emsy}{"74}
\DeclareMathSymbol u{\mathalpha}{emsy}{"75}
\DeclareMathSymbol v{\mathalpha}{emsy}{"76}
\DeclareMathSymbol w{\mathalpha}{emsy}{"77}
\DeclareMathSymbol x{\mathalpha}{emsy}{"78}
\DeclareMathSymbol y{\mathalpha}{emsy}{"79}
\DeclareMathSymbol z{\mathalpha}{emsy}{"7A}
\DeclareMathSymbol A{\mathalpha}{emsy}{"41}
\DeclareMathSymbol B{\mathalpha}{emsy}{"42}
\DeclareMathSymbol C{\mathalpha}{emsy}{"43}
\DeclareMathSymbol D{\mathalpha}{emsy}{"44}
\DeclareMathSymbol E{\mathalpha}{emsy}{"45}
\DeclareMathSymbol F{\mathalpha}{emsy}{"46}
\DeclareMathSymbol G{\mathalpha}{emsy}{"47}
\DeclareMathSymbol H{\mathalpha}{emsy}{"48}
\DeclareMathSymbol I{\mathalpha}{emsy}{"49}
\DeclareMathSymbol J{\mathalpha}{emsy}{"4A}
\DeclareMathSymbol K{\mathalpha}{emsy}{"4B}
\DeclareMathSymbol L{\mathalpha}{emsy}{"4C}
\DeclareMathSymbol M{\mathalpha}{emsy}{"4D}
\DeclareMathSymbol N{\mathalpha}{emsy}{"4E}
\DeclareMathSymbol O{\mathalpha}{emsy}{"4F}
\DeclareMathSymbol P{\mathalpha}{emsy}{"50}
\DeclareMathSymbol Q{\mathalpha}{emsy}{"51}
\DeclareMathSymbol R{\mathalpha}{emsy}{"52}
\DeclareMathSymbol S{\mathalpha}{emsy}{"53}
\DeclareMathSymbol T{\mathalpha}{emsy}{"54}
\DeclareMathSymbol U{\mathalpha}{emsy}{"55}
\DeclareMathSymbol V{\mathalpha}{emsy}{"56}
\DeclareMathSymbol W{\mathalpha}{emsy}{"57}
\DeclareMathSymbol X{\mathalpha}{emsy}{"58}
\DeclareMathSymbol Y{\mathalpha}{emsy}{"59}
\DeclareMathSymbol Z{\mathalpha}{emsy}{"5A}
\DeclareMathSymbol{\bullet}{\mathalpha}{emsymb}{"B7}
\DeclareMathSymbol{\regis}{\mathalpha}{emsymb}{"D2}
\def\Bullet{\leavevmode\unkern{$\m@th\bullet$}\kern.32em\ignorespaces}
\def\Regis{\leavevmode\raise.5ex\hbox{$\m@th\regis$}}
\DeclareMathSymbol +{\mathbin}{emcmr}{`+}
\DeclareMathSymbol ={\mathrel}{emcmr}{`=}
\DeclareMathSymbol{\Gamma}{\mathalpha}{emcmr}{"00}
\DeclareMathSymbol{\Delta}{\mathalpha}{emcmr}{"01}
\DeclareMathSymbol{\Theta}{\mathalpha}{emcmr}{"02}
\DeclareMathSymbol{\Lambda}{\mathalpha}{emcmr}{"03}
\DeclareMathSymbol{\Xi}{\mathalpha}{emcmr}{"04}
\DeclareMathSymbol{\Pi}{\mathalpha}{emcmr}{"05}
\DeclareMathSymbol{\Sigma}{\mathalpha}{emcmr}{"06}
\DeclareMathSymbol{\Upsilon}{\mathalpha}{emcmr}{"07}
\DeclareMathSymbol{\Phi}{\mathalpha}{emcmr}{"08}
\DeclareMathSymbol{\Psi}{\mathalpha}{emcmr}{"09}
\DeclareMathSymbol{\Omega}{\mathalpha}{emcmr}{"0A}
\DeclareMathSizes{7.6}{8}{6}{5}
\def\`#1{{\accent"12 #1}}            
\chardef\J="11
\chardef\AA="C8                      
\chardef\gbp="A3                     
\chardef\TIL="81                     
\chardef\endash="B1
\chardef\emdash="D0
\chardef\pourmille="BD               
\chardef\aoben="E3                   
\chardef\ooben="EB                   
\def\S{\leavevmode\unkern{\char"A7}\kern.1em\ignorespaces}
\DeclareMathAccent{\dot}{\mathalpha}{operators}{"C7} 

\newtheorem{de}{Definition}[section]
\newtheorem{te}[de]{Theorem}
\newtheorem{Pro}[de]{Proposition}
\newtheorem{co}[de]{Corollary}
\newtheorem{lem}[de]{Lemma}
\newtheorem{re}[de]{Remark}
\newtheorem{ex}[de]{Example}
\newtheorem{q}[de]{Question}

\title{\textbf{Rational homology cobordisms of plumbed 3-manifolds}}
\author{Paolo Aceto}
\date{}

\begin{document}
\maketitle
\begin{abstract}
We investigate rational homology cobordisms of 3-manifolds with non-zero first Betti number. 
This is motivated by the natural generalization of the \emph{slice-ribbon conjecture} to multicomponent links.
In particular we consider the problem
of which rational homology $S^1\times S^2$'s bound rational homology $S^1\times D^3$'s. We give a simple procedure to construct
rational homology cobordisms between plumbed 3-manifold. We introduce a family $\mathcal{F}$ of plumbed 3-manifolds
with $b_1=1$. By adapting an obstruction based on Donaldson's diagonalization theorem
we characterize all manifolds in $\mathcal{F}$ that bound rational homology $S^1\times D^3$'s. For all these manifolds a rational
homology cobordism to $S^1\times S^2$ can be constructed via our procedure.
The family $\mathcal{F}$ is large enough to include all Seifert fibered spaces over the 2-sphere with vanishing Euler invariant.
In a subsequent paper we describe applications to arborescent link concordance.
\end{abstract}
\newpage

\tableofcontents
\begin{section}{Introduction}
 The study of concordance properties of classical knots and links in the 3-sphere is a highly active field 
of research in low dimensional topology. Problems in this area involve a wide range of techniques, from the use
of sophisticated combinatorial invariants derived from knot homology theories to the interplay with 3 and 4-manifold
topology. \bigskip

One of the most famous unsolved problems in this field is the so called \emph{slice-ribbon conjecture}.
A knot $K\subset S^3$ is \emph{smoothly slice} if it bounds a properly embedded smooth disk in the 4-ball. A smoothly slice knot
is \emph{ribbon} if the spanning disk $D^2\subset D^4$ can be choosen so that there are no local maxima of the radial function
$\rho:D^4\rightarrow [0,1]$ restricted to the image of $D^2$. The slice ribbon cojecture states that every slice knot is ribbon.
Since it was first formulated by Fox in 1962 (as a question rather than a conjecture) there have been many efforts
towards understanding slice and ribbon knots. One stimulating aspect of this topic is that it naturally leads
to several related questions on 3-manifold topology.

In \cite{Lisca:1} Lisca proved that the slice ribbon conjecture holds true for 2-bridge knots. He used an obstruction
based on Donaldson's diagonalization theorem to determine which lens spaces bound rational homology balls. This technique has been
used by Lecuona in \cite{Lecuona:1} to prove that the slice ribbon conjecture holds true for an infinite family of Montesinos knots. 
In \cite{Donald:1} Donald refined the obstruction used by Lisca to determine which connected sums of lens spaces embed smoothly in $S^4$.
The starting point of this work is an adaption of these ideas to the study of slice links with more than one component.\bigskip

The basic idea of \cite{Lisca:1} can be described as follows. If a knot $K$ is slice its branched double cover $\Sigma(K)$ is a rational homology sphere
that bounds a rational homology ball $W$. If $K$ is a 2-bridge knot then $\Sigma(K)$ is a lens space, say $L(p,q)$. Each lens space is the boundary 
of a canonical plumbed 4-manifold $X(p,q)$ with negative definite intersection form. 
By taking the union $X'=X(p,q)\cup -W$ we obtain a smooth closed oriented 4-manifold
with unimodular, negative definite intersection form, and by Donaldson's diagonalization theorem this intersection form is diagonalizable over the integers. The inclusion
$X(p,q)\hookrightarrow X'$ induces an embedding of intersection lattices $(H_2(X(p,q);\mathbb{Z}),Q_{X(p,q)})\hookrightarrow (\mathbb{Z}^N,-I_N)$. This
fact turns out to be a powerful obstruction which eventually leads to a complete list of lens spaces that bound rational homology balls.

A link $L\subset S^3$ is (smoothly) slice if it bounds a disjoint union of properly embedded disks in the 4-ball, one for each component of $L$.
Let $L$ be a slice link with n components (n>1). The first observation is that $\Sigma(L)$ is a 3-manifold with $b_1=n-1$ which
bounds a smooth 4 manifold $W$ with the rational homology of a boundary connected sum of $n-1$ copies of $S^1\times D^3$ (see Proposition \ref{motivation}). 
Motivated by this fact 
and focusing on the case $n=2$ we are led to the following general problem:
\begin{q}\label{q1}
 Which rational homology $S^1\times S^2$'s bound rational homology $S^1\times D^3$'s?
\end{q}
In Section \ref{2-2} we introduce a general procedure which allows one to construct rational homology cobordisms between plumbed 3-manifolds. For any
plumbed 3-manifold $Y$ our procedure gives infinitely many plumbed 3-manifolds which are rational homology cobordant to $Y$. We then introduce
a family $\mathcal{F}$ of plumbed 3-manifolds with $b_1=1$. This family includes, up to orientation reversal, all Seifert fibered spaces
over the 2-sphere with vanishing Euler invariant. We prove that if a given $Y\in\mathcal{F}$ bounds a rational homology $S^1\times D^3$ then $Y$
can be constructed with our procedure (see Theorem \ref{main}). 
This gives us a complete list of the 3-manifolds in $\mathcal{F}$ that bound a rational $S^1\times D^3$.
By specializing Theorem \ref{main} to star-shaped plumbing graphs, we obtain the following characterization for the Seifert fibered 
spaces over the 2-sphere which bound rational homology $S^1\times D^3$'s.

\begin{te}
 A Seifert fibered manifold 
 $Y=(0;b;(\alpha_1,\beta_1),\dots,(\alpha_h,\beta_h))$ 
 bounds a $\mathbb{Q}H-S^1\times D^3$ if and only if the Seifert invariants occur in complementary pairs
 and $e(Y)=0$.
\end{te}
Two pairs of  Seifert invariants $(\alpha_i,\beta_i)$ and $(\alpha_j,\beta_j)$ are \emph{complementary} if
they can be chosen so that $\frac{\beta_i}{\alpha_i}+\frac{\beta_j}{\alpha_j}=-1$ (see Section \ref{lensseifert} for precise definitions).

This result (as well as Theorem \ref{main}) is obtained by using an obstruction based on Donaldson's theorem. Roughly speaking we proceed as follows. 
Each $Y$ in $\mathcal{F}$ bounds a negative \emph{semidefinite} plumbed 4-manifold $X$. If $Y$ bounds a rational homology $S^1\times D^3$, say $W$,
we can form the closed 4-manifold $X'=X\cup -W$. The intersection form $Q_{X'}$ will again be negative definite and this fact provides the costraints
we need for our analysis. 

In a subsequent paper \cite{Aceto} we will describe the applications of our work on arborescent link concordance. 
To each $Y\in \mathcal{F}$ we can associate the family $L(Y)$ of arborescent links whose branched double cover
is $Y$. In general, the family $L(Y)$ contains many non isotopic links. However, these links are all related to each other by Conway mutation.
In \cite{Aceto} we will prove the following 
\begin{te}
Let $L$ be a link in $L(Y)$ for some $Y\in\mathcal{F}$ (e.g. any Montesinos link). 
The following conditions are equivalent:
\begin{itemize}
 \item $Y$ bounds a rational homology $S^1\times D^3$;
 \item there exists $L'\in L(Y)$ that bounds a properly embedded smooth surface $S$ in $D^4$ with $\chi(S)=2$ without local maxima.
\end{itemize}
In particular every 2-component slice link $L\in L(\mathcal{F})$ has a ribbon mutant. 
\end{te}
This  paper is organized as follows. In Section \ref{plumbedmanifolds} we provide an introduction to plumbed manifolds
following \cite{Neumann:1}, \cite{Neumann:2} and \cite{Neumann:3}. We also introduce some new terminology that will be useful 
later on. In Section \ref{motiv} we give some motivation for our work relating rational homology cobordism of 3-manifolds and link concordance. 
We also state our lattice theoretical obstruction. In Section \ref{2-2} we introduce a method that allows one to construct rational 
homology cobordisms between plumbed 3-manifolds. In Section \ref{mainchapter} we state our main theorem (Theorem \ref{main})
and give a proof modulo a technical result (Theorem \ref{technical}). Sections \ref{3basic}- \ref{3conclusion} are dedicated
to the technical analysis needed to prove Theorem \ref{technical}.

\section*{Acknowledgements}
I would like to thank my supervisor Paolo Lisca for his support and for suggesting this topic, Giulia Cervia for her constant
encouragement and her help in drawing pictures.
\end{section}

\begin{section}{Plumbed manifolds}\label{plumbedmanifolds}
In this section, following \cite{Neumann:1},\cite{Neumann:2} and \cite{Neumann:3}, we review the basic definitions and properties of plumbed 
3-manifolds. We recall Neumann's normal form of a plumbing graph, and the
generalized continued fraction associated to a plumbing graph. 
We show how these data behave with respect to orientation reversal.
We briefly recall the definitions of lens spaces and Seifert manifolds
viewed as special plumbed 3-manifolds. 

Almost everything in this section is well known. 
The main purpose here is to fix notations and conventions as well as
putting our main result, Theorem \ref{main}, into the right context.  

\begin{de}\label{plumbinggraph}
 A \emph{plumbing graph} $\Gamma$ is a finite tree where every vertex has an integral 
 weight assigned to it.
\end{de}
To every plumbing graph $\Gamma$ we can associate a smooth oriented 4-manifold 
$P\Gamma$ with boundary $\partial P\Gamma$ in the following way. 
For each vertex take a disc
bundle over the 2-sphere with Euler number prescribed by the weight of 
the vertex. Whenever two vertices are connected by an edge we identify 
the trivial bundles over two small discs (one in each sphere) by exchanging 
the role of the fiber and the base coordinates. We call $P\Gamma$ 
(resp. $\partial P\Gamma$) a \emph{plumbed 4-manifold} 
(resp. \emph{plumbed 3-manifold}).

This definition can be extended to reducible 3-manifolds; 
if the graph is a finite forest (i.e. a disjoint union of 
 trees) we take the boundary 
connected sum of the plumbed 4-manifolds associated to each connected 
component of $\Gamma$. Unless otherwise stated, by a plumbing graph we will always mean
a connected one, as in Definition \ref{plumbinggraph}.

Every plumbed 4-manifold has a nice surgery description which can be
obtained directly from the plumbing graph. To every vertex we associate
an unknotted circle framed according to the weight of the vertex. Whenever
two vertices are connected by an edge the corresponding circles are linked
in the simplest possible way, i.e. like the Hopf link. The framed link 
obtained in this way also gives an integral surgery presentation for the corresponding 
plumbed 3-manifold. The group $H_2(P(\Gamma);\mathbb{Z})$ is a free 
abelian group generated by the zero sections of the sphere bundles (i.e. by
vertices of the graph). Moreover,with respect to this basis,
the intersection form of $P(\Gamma)$,
which we indicate by $Q_{\Gamma}$, is described by the matrix 
$M_{\Gamma}$ whose entries $(a_{ij})$ are defined as follows:
\begin{itemize}
 \item $a_{i,i}$ equals the Euler number of the corresponding disc bundle
 \item $a_{i,j}=1$ if the corresponding vertices are connected
 \item $a_{i,j}=0$ otherwise.
\end{itemize}
Finally note that $M_{\Gamma}$ is also a presentation matrix for the group
$H_1(\partial P\Gamma;\mathbb{Z})$.
 
\subsection{The normal form of a plumbing graph}

We will be mainly interested in plumbed 3-manifolds. There are some elementary
operations on the plumbing graph which alter the 4-manifold but not its boundary.
Following \cite{Neumann:1} we will state a theorem which establishes the existence
of a unique \emph{normal form} for the graph of a plumbed 3-manifold. In \cite{Neumann:1}
these results are stated in a more general context, here we extrapolate only what we 
need in order to deal with plumbed manifolds.

First consider the \emph{blow-down} operation. It can be performed in any of 
three situations depicted below.\\
\begin{enumerate}
 \item We can add or remove an isolated vertex with weight 
$\varepsilon\in\{\pm 1\}$ from any plumbing graph.
    \[
  \begin{tikzpicture}[xscale=0.7,yscale=-0.5]
    \node (A0_2) at (2, 0) {$\varepsilon$};
    \node (A1_0) at (0, 1) {$\Gamma$};
    \node (A1_1) at (1, 1) {$\sqcup$};
    \node (A1_2) at (2, 1) {$\bullet$};
    \node (A1_4) at (4, 1) {$\longleftrightarrow$};
    \node (A1_6) at (6, 1) {$\Gamma$};
  \end{tikzpicture}
  \]
  \item A vertex with weight $\varepsilon\in\{\pm 1\}$ linked to a single vertex of a 
plumbing graph can be removed as shown below. From now on we use three edges 
coming out of a vertex to indicate that any number of edges may be linked to that vertex.
      \[
  \begin{tikzpicture}[xscale=1.5,yscale=-0.5]
    \node (A0_0) at (0, 0) {};
    \node (A0_5) at (5, 0) {};
    \node (A1_1) at (1, 1) {$a$};
    \node (A1_2) at (2, 1) {$\varepsilon$};
    \node (A1_6) at (6, 1) {$a-\varepsilon$};
    \node (A2_0) at (0, 2) {};
    \node (A2_1) at (1, 2) {$\bullet$};
    \node (A2_2) at (2, 2) {$\bullet$};
    \node (A2_4) at (4, 2) {$\longleftrightarrow$};
    \node (A2_5) at (5, 2) {};
    \node (A2_6) at (6, 2) {$\bullet$};
    \node (A4_0) at (0, 4) {};
    \node (A4_5) at (5, 4) {};
    \path (A4_5) edge [-] node [auto] {$\scriptstyle{}$} (A2_6);
    \path (A0_5) edge [-] node [auto] {$\scriptstyle{}$} (A2_6);
    \path (A2_0) edge [-] node [auto] {$\scriptstyle{}$} (A2_1);
    \path (A4_0) edge [-] node [auto] {$\scriptstyle{}$} (A2_1);
    \path (A2_5) edge [-] node [auto] {$\scriptstyle{}$} (A2_6);
    \path (A2_1) edge [-] node [auto] {$\scriptstyle{}$} (A2_2);
    \path (A0_0) edge [-] node [auto] {$\scriptstyle{}$} (A2_1);
  \end{tikzpicture}
  \]
  \item Finally, if a $\pm 1$-weighted vertex is linked to exactly two vertices it can be removed as shown below.
    \[
  \begin{tikzpicture}[xscale=1.5,yscale=-0.5]
    \node (A0_0) at (0, 0) {};
    \node (A0_4) at (4, 0) {};
    \node (A0_6) at (6, 0) {};
    \node (A0_9) at (9, 0) {};
    \node (A1_1) at (1, 1) {$a$};
    \node (A1_2) at (2, 1) {$\varepsilon$};
    \node (A1_3) at (3, 1) {$b$};
    \node (A1_7) at (7, 1) {$a-\varepsilon$};
    \node (A1_8) at (8, 1) {$b-\varepsilon$};
    \node (A2_0) at (0, 2) {};
    \node (A2_1) at (1, 2) {$\bullet$};
    \node (A2_2) at (2, 2) {$\bullet$};
    \node (A2_3) at (3, 2) {$\bullet$};
    \node (A2_4) at (4, 2) {};
    \node (A2_5) at (5, 2) {$\longleftrightarrow$};
    \node (A2_6) at (6, 2) {};
    \node (A2_7) at (7, 2) {$\bullet$};
    \node (A2_8) at (8, 2) {$\bullet$};
    \node (A2_9) at (9, 2) {};
    \node (A4_0) at (0, 4) {};
    \node (A4_4) at (4, 4) {};
    \node (A4_6) at (6, 4) {};
    \node (A4_9) at (9, 4) {};
    \path (A0_9) edge [-] node [auto] {$\scriptstyle{}$} (A2_8);
    \path (A0_4) edge [-] node [auto] {$\scriptstyle{}$} (A2_3);
    \path (A2_9) edge [-] node [auto] {$\scriptstyle{}$} (A2_8);
    \path (A4_9) edge [-] node [auto] {$\scriptstyle{}$} (A2_8);
    \path (A2_6) edge [-] node [auto] {$\scriptstyle{}$} (A2_7);
    \path (A2_1) edge [-] node [auto] {$\scriptstyle{}$} (A2_2);
    \path (A4_0) edge [-] node [auto] {$\scriptstyle{}$} (A2_1);
    \path (A2_2) edge [-] node [auto] {$\scriptstyle{}$} (A2_3);
    \path (A4_4) edge [-] node [auto] {$\scriptstyle{}$} (A2_3);
    \path (A4_6) edge [-] node [auto] {$\scriptstyle{}$} (A2_7);
    \path (A0_6) edge [-] node [auto] {$\scriptstyle{}$} (A2_7);
    \path (A2_7) edge [-] node [auto] {$\scriptstyle{}$} (A2_8);
    \path (A2_0) edge [-] node [auto] {$\scriptstyle{}$} (A2_1);
    \path (A2_4) edge [-] node [auto] {$\scriptstyle{}$} (A2_3);
    \path (A0_0) edge [-] node [auto] {$\scriptstyle{}$} (A2_1);
  \end{tikzpicture}
  \] 
\end{enumerate}
Next we have the \emph{0-chain absorption} move. A $0$-weighted vertex linked to two vertices 
can be removed and the plumbing graph changes as shown.
   \[
  \begin{tikzpicture}[xscale=1.5,yscale=-0.5]
    \node (A0_0) at (0, 0) {};
    \node (A0_4) at (4, 0) {};
    \node (A0_6) at (6, 0) {};
    \node (A0_8) at (8, 0) {};
    \node (A1_1) at (1, 1) {$a$};
    \node (A1_2) at (2, 1) {$0$};
    \node (A1_3) at (3, 1) {$b$};
    \node (A1_7) at (7, 1) {$a+b$};
    \node (A2_0) at (0, 2) {};
    \node (A2_1) at (1, 2) {$\bullet$};
    \node (A2_2) at (2, 2) {$\bullet$};
    \node (A2_3) at (3, 2) {$\bullet$};
    \node (A2_4) at (4, 2) {};
    \node (A2_5) at (5, 2) {$\longleftrightarrow$};
    \node (A2_6) at (6, 2) {};
    \node (A2_7) at (7, 2) {$\bullet$};
    \node (A2_8) at (8, 2) {};
    \node (A4_0) at (0, 4) {};
    \node (A4_4) at (4, 4) {};
    \node (A4_6) at (6, 4) {};
    \node (A4_8) at (8, 4) {};
    \path (A0_8) edge [-] node [auto] {$\scriptstyle{}$} (A2_7);
    \path (A2_2) edge [-] node [auto] {$\scriptstyle{}$} (A2_3);
    \path (A2_8) edge [-] node [auto] {$\scriptstyle{}$} (A2_7);
    \path (A4_8) edge [-] node [auto] {$\scriptstyle{}$} (A2_7);
    \path (A2_6) edge [-] node [auto] {$\scriptstyle{}$} (A2_7);
    \path (A2_0) edge [-] node [auto] {$\scriptstyle{}$} (A2_1);
    \path (A4_0) edge [-] node [auto] {$\scriptstyle{}$} (A2_1);
    \path (A0_4) edge [-] node [auto] {$\scriptstyle{}$} (A2_3);
    \path (A4_4) edge [-] node [auto] {$\scriptstyle{}$} (A2_3);
    \path (A4_6) edge [-] node [auto] {$\scriptstyle{}$} (A2_7);
    \path (A0_6) edge [-] node [auto] {$\scriptstyle{}$} (A2_7);
    \path (A2_1) edge [-] node [auto] {$\scriptstyle{}$} (A2_2);
    \path (A2_4) edge [-] node [auto] {$\scriptstyle{}$} (A2_3);
    \path (A0_0) edge [-] node [auto] {$\scriptstyle{}$} (A2_1);
  \end{tikzpicture}
  \]
The \emph{splitting} move can be applied in the following situation. Given a plumbing graph with a 
$0$-weighted vertex which is linked to a single vertex $v$, we may remove both vertices
(and all the corresponding edges) obtaining a disjoint union of plumbing trees. We may
depict this move as follows
  \[
  \begin{tikzpicture}[xscale=1.2,yscale=-0.5]
    \node (A0_2) at (2, 0) {$\Gamma_1$};
    \node (A1_2) at (2, 1) {$\cdot$};
    \node (A2_0) at (0, 2) {$0$};
    \node (A2_1) at (1, 2) {$a$};
    \node (A2_2) at (2, 2) {$\cdot$};
    \node (A3_0) at (0, 3) {$\bullet$};
    \node (A3_1) at (1, 3) {$\bullet$};
    \node (A3_2) at (2, 3) {$\cdot$};
    \node (A3_3) at (3, 3) {$\longleftrightarrow$};
    \node (A3_5) at (5, 3) {$\Gamma_1\sqcup\dots\sqcup\Gamma_k$};
    \node (A4_2) at (2, 4) {$\cdot$};
    \node (A5_2) at (2, 5) {$\cdot$};
    \node (A6_2) at (2, 6) {$\Gamma_k$};
    \path (A3_0) edge [-] node [auto] {$\scriptstyle{}$} (A3_1);
    \path (A3_1) edge [-] node [auto] {$\scriptstyle{}$} (A6_2);
    \path (A3_1) edge [-] node [auto] {$\scriptstyle{}$} (A0_2);
  \end{tikzpicture}
  \]
\begin{Pro}{\rm{\cite{Neumann:1}}}
Applying any of the above operations to a plumbing graph does not
change the oriented diffeomorphism type of the corresponding plumbed 3-manifold.
\end{Pro}

Before discussing the normal form of a plumbing graph we need some terminology. 
A \emph{linear chain} of a plumbing graph is a portion of the graph consisting 
of some vertices $v_1,\dots,v_k$ ($k\geq 1$) such that:
\begin{itemize}
 \item each $v_i$ with $1<i<k$ is linked only to $v_{i-1}$ and $v_{i+1}$ 
 \item $v_1$ and $v_k$ are linked to at most two vertices.
\end{itemize}
A linear chain is \emph{maximal} if it is not contained in any larger linear chain.
A vertex of a plumbing graph is said to be:
\begin{enumerate}
 \item \emph{isolated} if it is not linked to any other vertex
 \item \emph{final} if it is linked exactly to one vertex
 \item \emph{internal} otherwise.
\end{enumerate}
Note that isolated and final vertices always belong to some linear chain,
while an internal vertex belongs to some linear chain if and only if it is 
linked to exactly two vertices.
\begin{de}
 A plumbing graph $\Gamma$ is said to be in \emph{normal form} if one of the following holds
 \begin{enumerate}
  \item  
    $\begin{tikzpicture}[xscale=0.7,yscale=-0.5]
    \node (A0_4) at (4, 0) {$0$};
    \node (A1_0) at (0, 1) {$\Gamma=$};
    \node (A1_1) at (1, 1) {$\varnothing$};
    \node (A1_2) at (2, 1) {$or$};
    \node (A1_3) at (3, 1) {$\Gamma=$};
    \node (A1_4) at (4, 1) {$\bullet$};
  \end{tikzpicture}$
 \item every vertex of a linear chain has weight less than or equal to -2.  
 \end{enumerate}
\end{de}

\begin{te}{\rm{\cite{Neumann:1}}}\label{exnormalform} Every plumbing graph can be reduced to a unique normal form
via a sequence of blow-downs, 0-chain absorptions and splittings. Moreover two oriented plumbed
3-manifolds are diffeomorphic (preserving the orientation) if and only if their plumbing graphs
have the same normal form. 
\end{te}
\begin{re}
 We point out that using this theorem one can specify a certain class of plumbed 3-manifolds
 simply by describing the shape of the plumbing graph in its normal form. In particular we will
 see at the end of this section that lens spaces and some Seifert manifolds admit such a description.
\end{re}

\subsection{The continued fraction of a plumbing graph}
In this section, following \cite{Neumann:2} we introduce some additional data associated to a
plumbing graph. As we have seen to any plumbing graph $\Gamma$ we can associate an integral
symmetric bilinear form $Q_{\Gamma}$. All the usual invariants of $Q_{\Gamma}$ will 
be denoted referring only to the graph. In particular rank, signature and determinant
will be denoted respectively by $\textrm{rk}\Gamma$, 
$(b_+\Gamma,b_-\Gamma,b_0\Gamma)$ and $det\Gamma$.

Let $(\Gamma,v)$ be a connected \emph{rooted plumbing graph}, i.e. 
a plumbing graph together with the choice of a particular vertex.
If we remove from $\Gamma$ the vertex $v$ and all the corresponding edges we obtain a plumbing 
graph $\Gamma_v$ which is the disjoint union of some trees $\Gamma_1,\dots,\Gamma_k$ 
($k$ is the valency of $v$). Every such tree has a distinguished vertex $v_j$ which is the one 
adjacent to $v$.
\begin{de}
 With the notation above we define the \emph{continued fraction} of $\Gamma$ as
 $$
 cf(\Gamma):=\frac{det\Gamma}{det\Gamma_v}\in\mathbb{Q}\cup\{\infty\}
 $$
\end{de}
We put $\alpha/0=\infty$ for each $\alpha\in\mathbb{Q}$.
\begin{re}
 Note that $cf(\Gamma)$ depends on the \emph{rooted} plumbing graph $(\Gamma,v)$. By abusing notation we do
 not indicate this dependence explicitely. In the sequel, it will always be clear from the context which vertex has been chosen.
\end{re}
\begin{Pro}\label{cf}\cite{Neumann:2}
 If the weight of the distinguished vertex is $b\in\mathbb{Z}$ then
 $$
 det\Gamma=b\cdot det\Gamma_v-\sum_{i=1}^{k} \left(det\Gamma_{v_i}\prod_{j\neq i}det\Gamma_j\right)
 $$
 and
 $$
 cf (\Gamma)=b-\sum_{i=1}^{k}\frac{1}{cf(\Gamma_i)}.
 $$
\end{Pro}
\subsection{Reversing the orientation}
Let $\Gamma$ be a plumbing graph in normal form. In this section, following \cite{Neumann:1},
we explain how to compute the normal form for the plumbed 3-manifold 
$-\partial P\Gamma$, i.e. $\partial P\Gamma$ with reversed orientation. We call this plumbing
graph the \emph{dual graph} of $\Gamma$ and we denote it with $\Gamma^*$.

 For a vertex $v$ of a plumbing graph which is not on a linear chain we define the quantity
$c(v)$ to be the number of linear chains adjacent to $v$, i.e. the number of vertices 
belonging to a linear chain that are linked to $v$. For instance in the graph

  \[
  \begin{tikzpicture}[xscale=1.0,yscale=-0.7]
    \node (A0_0) at (0, 0) {$\bullet$};
    \node (A0_3) at (3, 0) {$\bullet$};
    \node (A1_1) at (1, 1) {$\bullet$};
    \node (A1_2) at (2, 1) {$\bullet$};
    \node (A2_0) at (0, 2) {$\bullet$};
    \node (A2_3) at (3, 2) {$\bullet$};
    \path (A2_0) edge [-] node [auto] {$\scriptstyle{}$} (A1_1);
    \path (A1_2) edge [-] node [auto] {$\scriptstyle{}$} (A2_3);
    \path (A1_2) edge [-] node [auto] {$\scriptstyle{}$} (A0_3);
    \path (A1_1) edge [-] node [auto] {$\scriptstyle{}$} (A1_2);
    \path (A0_0) edge [-] node [auto] {$\scriptstyle{}$} (A1_1);
  \end{tikzpicture}
  \]
both the trivalent vertices have $c=2$.
We indicate with $(\dots,-2^{[a]},\dots)$ a portion of a string with a $-2$-chain of length $a>0$. 
\begin{te}\label{normalform}\cite{Neumann:1}
Let $\Gamma$ be a plumbing graph in normal form. Its dual graph $\Gamma^*$ can be obtained as 
follows. 
The weight $w(v)$ of every vertex which is not on a linear chain is replaced with
$-w(v)-c(v)$. Every maximal linear chain of the form
  \[
  \begin{tikzpicture}[xscale=1.5,yscale=-0.5]
    \node (A1_1) at (1, 1) {$a_1$};
    \node (A1_2) at (2, 1) {$a_2$};
    \node (A1_4) at (4, 1) {$a_n$};
    \node (A2_0) at (0, 2) {$\dots$};
    \node (A2_1) at (1, 2) {$\bullet$};
    \node (A2_2) at (2, 2) {$\bullet$};
    \node (A2_3) at (3, 2) {$\dots$};
    \node (A2_4) at (4, 2) {$\bullet$};
    \node (A2_5) at (5, 2) {$\dots$};
    \path (A2_2) edge [-] node [auto] {$\scriptstyle{}$} (A2_3);
    \path (A2_1) edge [-] node [auto] {$\scriptstyle{}$} (A2_2);
    \path (A2_3) edge [-] node [auto] {$\scriptstyle{}$} (A2_4);
    \path (A2_0) edge [-] node [auto] {$\scriptstyle{}$} (A2_1);
    \path (A2_4) edge [-] node [auto] {$\scriptstyle{}$} (A2_5);
  \end{tikzpicture}
  \]
is replaced with
  \[
  \begin{tikzpicture}[xscale=1.5,yscale=-0.5]
    \node (A1_1) at (1, 1) {$b_1$};
    \node (A1_2) at (2, 1) {$b_2$};
    \node (A1_4) at (4, 1) {$b_m$};
    \node (A2_0) at (0, 2) {$\dots$};
    \node (A2_1) at (1, 2) {$\bullet$};
    \node (A2_2) at (2, 2) {$\bullet$};
    \node (A2_3) at (3, 2) {$\dots$};
    \node (A2_4) at (4, 2) {$\bullet$};
    \node (A2_5) at (5, 2) {$\dots$};
    \path (A2_2) edge [-] node [auto] {$\scriptstyle{}$} (A2_3);
    \path (A2_1) edge [-] node [auto] {$\scriptstyle{}$} (A2_2);
    \path (A2_3) edge [-] node [auto] {$\scriptstyle{}$} (A2_4);
    \path (A2_0) edge [-] node [auto] {$\scriptstyle{}$} (A2_1);
    \path (A2_4) edge [-] node [auto] {$\scriptstyle{}$} (A2_5);
  \end{tikzpicture}
  \]
where the weights are determined as follows. If
$$
(a_1,\dots,a_n)=(-2^{[n_0]},-m_1-3,-2^{[n_1]},-m_2-3,\dots,-m_s-3,-2^{[n_s]})
$$
with $n_i\geq 0$ and $m_i\geq0$. Then
$$
(b_1,\dots,b_m)=(-n_0 -2,-2^{m_1},-n_1-3,\dots,-n_{s-1}-3,-2^{m_s},-n_s-2).
$$
\end{te}
The reason why we are interested in this construction of the dual graph of a plumbing 
graph in normal form will be clear in Section \ref{mainchapter}. Essentialy we are trying
to detect nullcobordant 3-manifolds using obstructions based on Donaldson's diagonalization theorem.
Since the property we want to detect does not depend on the orientation of a given 3-manifold
it is natural to examine both a plumbing graph $\Gamma$ and its dual $\Gamma^*$. Moreover,
the normal form is specifically defined to give a plumbing graph that minimizes the quantity
$b_+(\Gamma)$ among all plumbing graphs representing $\partial P\Gamma$ (see \cite{Neumann:2}
theorem 1.2).

We now introduce a quantity that will play an important role in the analysis developed from
Section \ref{3basic} till the end of the paper. 
\begin{de}
 Let $\Gamma$ be a plumbing graph in normal form, and let $v_1,\dots,v_n$ be its vertices.
 We define 
 $$
 I(\Gamma):=\sum_{i=1}^{n}-3-w(v_i).
 $$
\end{de}
The following Proposition is proved in \cite{Lisca:1}. It can also be proved directly using
Theorem \ref{normalform}.
\begin{Pro}\label{I}
 Let $\Gamma$ be a linear plumbing graph in normal form. We have
 $$
 I(\Gamma)+I(\Gamma^*)=-2.
 $$
\end{Pro}
\subsection{Lens spaces and Seifert manifolds}\label{lensseifert}
We briefly recall the plumbing description for lens spaces and Seifert manifolds.

In this context it is convenient to define a lens space as a closed 3-manifold whose
Heegaard genus is $\leq1$, the difference with the usual definition is that
we are including $S^3$ and $S^1\times S^2$. It is well known that every lens space
has a plumbing graph which is either empty ($S^3$) or a linear plumbing graph and that
every linear plumbing graph represents a lens space. It follows from Theorem \ref{exnormalform}
that the normal form of a plumbing graph representing a lens space other than $S^3$ or 
$S^1\times S^2$ is a linear plumbing graph 
   \[
  \begin{tikzpicture}[xscale=1.5,yscale=-0.7]
    \node (A0_0) at (0, 0) {$a_1$};
    \node (A0_1) at (1, 0) {$a_2$};
    \node (A0_3) at (3, 0) {$a_n$};
    \node (A1_0) at (0, 1) {$\bullet$};
    \node (A1_1) at (1, 1) {$\bullet$};
    \node (A1_2) at (2, 1) {$\dots$};
    \node (A1_3) at (3, 1) {$\bullet$};
    \path (A1_0) edge [-] node [auto] {$\scriptstyle{}$} (A1_1);
    \path (A1_2) edge [-] node [auto] {$\scriptstyle{}$} (A1_3);
    \path (A1_1) edge [-] node [auto] {$\scriptstyle{}$} (A1_2);
  \end{tikzpicture}
  \]

where $a_i\leq-2$ for each $i$.
It is easy to check that given a linear plumbing graph as above we have
$$
cf(\Gamma)=a_1-\frac{1}{a_2-\frac{1}{a_3+\dots}}=:[a_1,\dots,a_n]^-
$$
This fact justifies the name continued fraction. Note that $cf(\Gamma)<-1$.
The usual notation for a lens space $L(p,q)$, defined as $-\frac{p}{q}$-surgery on the unknot,
can recovered from the continued fraction as follows. Write $cf(\Gamma)=\frac{p}{-q}$, so that
$p>q\geq 1$ and $(p,q)=1$. Then $\partial P\Gamma=L(p,q)$.

A closed Seifert fibered manifold (see \cite{Neumann:3} and the references therein) 
can be described by its \emph{unnormalized Seifert invariants}
$$
(g;b;(\alpha_1,\beta_1,),\dots,(\alpha_k,\beta_k))
$$
where $g\geq0$ is the genus of the base surface, $b\in\mathbb{Z}$, $\alpha_i> 1$ and $(\alpha_i,\beta_i)=1$. This data, (which is not unique),
uniquely determines the manifold. When $g=0$ a surgery description for such a manifold is depicted in Figure \ref{seifertsurgery}.
\begin{figure}[H]  
\centering
  {\includegraphics[scale=1.0]{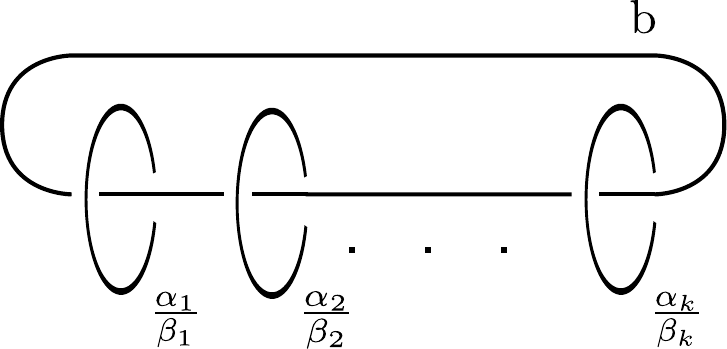}}
  \caption{A surgery description for the the Seifert fibered manifolds $(0;b;(\alpha_1,\beta_1,),\dots,(\alpha_k,\beta_k))$.}
  \label{seifertsurgery}
\end{figure}
The following theorem is proved in \cite{Neumann:3}.
\begin{te}\label{seiferttheorem}
 Let $\Gamma$ be the following starshaped plumbing graph in normal form.
   \[
  \begin{tikzpicture}[xscale=1.5,yscale=-0.5]
    \node (A0_1) at (1, 0) {$a^1_1$};
    \node (A0_3) at (3, 0) {$a^1_{n_1}$};
    \node (A1_1) at (1, 1) {$\bullet$};
    \node (A1_2) at (2, 1) {$\dots$};
    \node (A1_3) at (3, 1) {$\bullet$};
    \node (A2_0) at (0, 2) {$b$};
    \node (A2_1) at (1, 2) {$a^2_1$};
    \node (A2_3) at (3, 2) {$a^2_{n_2}$};
    \node (A3_0) at (0, 3) {$\bullet$};
    \node (A3_1) at (1, 3) {$\bullet$};
    \node (A3_2) at (2, 3) {$\dots$};
    \node (A3_3) at (3, 3) {$\bullet$};
    \node (A5_2) at (2, 5) {$\vdots$};
    \node (A6_1) at (1, 6) {$a^k_1$};
    \node (A6_3) at (3, 6) {$a^k_{n_k}$};
    \node (A7_1) at (1, 7) {$\bullet$};
    \node (A7_2) at (2, 7) {$\dots$};
    \node (A7_3) at (3, 7) {$\bullet$};
    \path (A7_2) edge [-] node [auto] {$\scriptstyle{}$} (A7_3);
    \path (A3_0) edge [-] node [auto] {$\scriptstyle{}$} (A3_1);
    \path (A1_1) edge [-] node [auto] {$\scriptstyle{}$} (A1_2);
    \path (A7_1) edge [-] node [auto] {$\scriptstyle{}$} (A7_2);
    \path (A3_2) edge [-] node [auto] {$\scriptstyle{}$} (A3_3);
    \path (A3_0) edge [-] node [auto] {$\scriptstyle{}$} (A1_1);
    \path (A3_0) edge [-] node [auto] {$\scriptstyle{}$} (A7_1);
    \path (A1_2) edge [-] node [auto] {$\scriptstyle{}$} (A1_3);
    \path (A3_1) edge [-] node [auto] {$\scriptstyle{}$} (A3_2);
  \end{tikzpicture}
  \]
Then $\partial P\Gamma$ is a Seifert manifold with unnormalized Seifert invariants
$$
(0,b;(\alpha_1,\beta_1),\dots,(\alpha_k,\beta_k))
$$
where
$$
\frac{\alpha_i}{\beta_i}=[a^i_1,\dots,a^i_{n_i}]^-.
$$
\end{te}
The quantity
$$
e(Y):=b-\sum_{i=1}^{k}\frac{\beta_i}{\alpha_i}
$$
is called the \emph{Euler number} of $Y$. It is easy to check that 
\begin{equation}\label{eulerfraction}
 e(Y)=cf(\Gamma)
\end{equation}
where $\Gamma$ is the plumbing graph in normal form associated to $Y$.
\begin{de}\label{complementary}
 Let $\Gamma_1$ and $\Gamma_2$ be two linear plumbing graphs in normal form.
   \[
  \begin{tikzpicture}[xscale=1.0,yscale=-0.5]
    \node (A0_1) at (1, 0) {$a_1$};
    \node (A0_2) at (2, 0) {$a_2$};
    \node (A0_4) at (4, 0) {$a_n$};
    \node (A1_0) at (0, 1) {$\Gamma_1:=$};
    \node (A1_1) at (1, 1) {$\bullet$};
    \node (A1_2) at (2, 1) {$\bullet$};
    \node (A1_3) at (3, 1) {$\dots$};
    \node (A1_4) at (4, 1) {$\bullet$};
    \path (A1_2) edge [-] node [auto] {$\scriptstyle{}$} (A1_3);
    \path (A1_3) edge [-] node [auto] {$\scriptstyle{}$} (A1_4);
    \path (A1_1) edge [-] node [auto] {$\scriptstyle{}$} (A1_2);
  \end{tikzpicture}
  \]
  \[
  \begin{tikzpicture}[xscale=1.0,yscale=-0.5]
    \node (A0_1) at (1, 0) {$b_1$};
    \node (A0_2) at (2, 0) {$b_2$};
    \node (A0_4) at (4, 0) {$b_m$};
    \node (A1_0) at (0, 1) {$\Gamma_2:=$};
    \node (A1_1) at (1, 1) {$\bullet$};
    \node (A1_2) at (2, 1) {$\bullet$};
    \node (A1_3) at (3, 1) {$\dots$};
    \node (A1_4) at (4, 1) {$\bullet$};
    \path (A1_2) edge [-] node [auto] {$\scriptstyle{}$} (A1_3);
    \path (A1_3) edge [-] node [auto] {$\scriptstyle{}$} (A1_4);
    \path (A1_1) edge [-] node [auto] {$\scriptstyle{}$} (A1_2);
  \end{tikzpicture}
  \]
$\Gamma_1$ and $\Gamma_2$ are said to be \emph{complementary} if $\Gamma_2=\Gamma_1^*$.
\end{de}
The following proposition has an elementary proof that we leave to the reader.
\begin{Pro}\label{compeq}
 With the notation of Definition \ref{complementary} the following conditions are equivalent
 \begin{enumerate}
  \item $\Gamma_1$ and $\Gamma_2$ are complementary
  \item $\partial P(  
  \begin{tikzpicture}[xscale=1.0,yscale=-0.5]
    \node (A0_0) at (0, 0) {$b_m$};
    \node (A0_2) at (2, 0) {$b_1$};
    \node (A0_3) at (3, 0) {$-1$};
    \node (A0_4) at (4, 0) {$a_1$};
    \node (A0_6) at (6, 0) {$a_n$};
    \node (A1_0) at (0, 1) {$\bullet$};
    \node (A1_1) at (1, 1) {$\dots$};
    \node (A1_2) at (2, 1) {$\bullet$};
    \node (A1_3) at (3, 1) {$\bullet$};
    \node (A1_4) at (4, 1) {$\bullet$};
    \node (A1_5) at (5, 1) {$\dots$};
    \node (A1_6) at (6, 1) {$\bullet$};
    \path (A1_4) edge [-] node [auto] {$\scriptstyle{}$} (A1_5);
    \path (A1_0) edge [-] node [auto] {$\scriptstyle{}$} (A1_1);
    \path (A1_5) edge [-] node [auto] {$\scriptstyle{}$} (A1_6);
    \path (A1_1) edge [-] node [auto] {$\scriptstyle{}$} (A1_2);
    \path (A1_2) edge [-] node [auto] {$\scriptstyle{}$} (A1_3);
    \path (A1_3) edge [-] node [auto] {$\scriptstyle{}$} (A1_4);
  \end{tikzpicture}  
)=S^1\times S^2$
\item $\frac{1}{cf(\Gamma_1)}+\frac{1}{cf(\Gamma_2)}=-1$
 \end{enumerate}
\end{Pro}
\begin{re}
 Note that, strictly speaking, the definition of complementary linear graphs should involve 
 an extra bit of data. In Definition \ref{complementary} we implicitly fixed
 an initial vertex and a final one on each graph (as suggested by the indexing of the weights). 
 Only in this way the condition $\Gamma_2=\Gamma_1^*$ makes sense.
\end{re}
It is useful to extend in the obvious way the notion of complementary linear graphs to that of 
\emph{complementary legs} in a starshaped plumbing graph. It follows by Theorem \ref{seifert}
and Proposition \ref{compeq} that pairs of complementary legs correspond to pairs
of Seifert invariants $(\alpha_i,\beta_i)$ and $(\alpha_j,\beta_j)$ that satisfy
$$
\frac{\beta_i}{\alpha_i}+\frac{\beta_j}{\alpha_j}=-1
$$
Note that, in general, this formula does not hold if we do not compute the Seifert invariants from
the weights of a star-shaped plumbing graph in normal form as in Theorem \ref{seiferttheorem}. 
We say that a pair of Seifert invariants
are complementary if they correspond to complementary legs in the associated star-shaped plumbing graph in normal form.
\subsection{The linear complexity of a tree}

Let $\Gamma$ be a plumbing graph in normal form. Let $lc(\Gamma)$ be the cardinality of the
smallest subset of vertices we need to remove from $\Gamma$ in order to obtain a linear graph.
We call $lc(\Gamma)$ the \emph{linear complexity} of $\Gamma$ and we set $lc(\varnothing)=-1$. We stress the fact that because of the
uniqueness of the normal form of a plumbing graph it makes sense to talk about the linear
complexity of a plumbed 3-manifold. Note that:
\begin{itemize}
 \item $lc(\Gamma)=0$ if and only if $\partial P\Gamma$ is a connected sum of lens spaces
 \item if $\partial(P\Gamma)$ is a Seifert manifold then $lc(\Gamma)=1$
 \item $lc(\Gamma_1\sqcup\Gamma_2)=lc(\Gamma_1)+lc(\Gamma_2)$.
\end{itemize}

\begin{Pro}\label{cfdet}
 Let $\Gamma$ be a plumbing graph in normal form such that $lc(\Gamma)=1$ 
 and for at least one choice
 of a vector $v\in\Gamma$ the graph $\Gamma_v$ is linear and negative definite.
 Then
 $$
 det\Gamma=0\Longleftrightarrow cf\Gamma=0.
 $$
\end{Pro}
\begin{proof}
 Assume that $\det\Gamma=0$. Let $v\in\Gamma$ be a vertex such that $\Gamma_v$ is 
 linear and negative definite. 
 By Proposition \ref{cf} we have
 \begin{equation}\label{zerocf}
  b\cdot det\Gamma_v-\sum_{i=1}^{k} \left(det\Gamma_{v_i}\prod_{j\neq i}det\Gamma_j\right)=0.
 \end{equation} 
 To obtain an expression for $cf(\Gamma)$ we divide both terms of the above equation by $det\Gamma_v$
 and we get
 $$
 cf(\Gamma)=b-\sum_{i=1}^{k}\frac{1}{cf(\Gamma_i)}=0
 $$
 this last equality holds because every $\Gamma_i$ is a linear negative definite graph
 and therefore its continued fraction is non vanishing.
 The converse is completely analogous.
\end{proof}
In Section \ref{mainchapter} we will deal mainly with plumbed 3-manifolds with $lc(\Gamma)=1$.
A generic plumbing graph $\Gamma$ with $lc(\Gamma)=1$ looks like the one shown below.
    \[
  \begin{tikzpicture}[xscale=1.5,yscale=-0.6]
    \node (A0_0) at (0, 0) {$\bullet$};
    \node (A0_1) at (1, 0) {$\dots$};
    \node (A0_2) at (2, 0) {$\bullet$};
    \node (A0_5) at (5, 0) {$\bullet$};
    \node (A0_6) at (6, 0) {$\dots$};
    \node (A0_7) at (7, 0) {$\bullet$};
    \node (A1_4) at (4, 1) {$\bullet$};
    \node (A2_0) at (0, 2) {$\bullet$};
    \node (A2_1) at (1, 2) {$\dots$};
    \node (A2_2) at (2, 2) {$\bullet$};
    \node (A2_5) at (5, 2) {$\bullet$};
    \node (A2_6) at (6, 2) {$\dots$};
    \node (A2_7) at (7, 2) {$\bullet$};
    \node (A3_2) at (2, 3) {$\cdot$};
    \node (A3_3) at (3, 3) {$\bullet$};
    \node (A3_4) at (4, 3) {$\cdot$};
    \node (A4_2) at (2, 4) {$\cdot$};
    \node (A4_4) at (4, 4) {$\cdot$};
    \node (A5_2) at (2, 5) {$\cdot$};
    \node (A5_4) at (4, 5) {$\cdot$};
    \node (A6_2) at (2, 6) {$\cdot$};
    \node (A6_5) at (5, 6) {$\bullet$};
    \node (A6_6) at (6, 6) {$\dots$};
    \node (A6_7) at (7, 6) {$\bullet$};
    \node (A7_0) at (0, 7) {$\bullet$};
    \node (A7_1) at (1, 7) {$\dots$};
    \node (A7_2) at (2, 7) {$\bullet$};
    \node (A7_4) at (4, 7) {$\bullet$};
    \node (A8_5) at (5, 8) {$\bullet$};
    \node (A8_6) at (6, 8) {$\dots$};
    \node (A8_7) at (7, 8) {$\bullet$};
    \path (A2_1) edge [-] node [auto] {$\scriptstyle{}$} (A2_2);
    \path (A7_2) edge [-] node [auto] {$\scriptstyle{}$} (A3_3);
    \path (A0_6) edge [-] node [auto] {$\scriptstyle{}$} (A0_7);
    \path (A7_0) edge [-] node [auto] {$\scriptstyle{}$} (A7_1);
    \path (A8_6) edge [-] node [auto] {$\scriptstyle{}$} (A8_7);
    \path (A7_1) edge [-] node [auto] {$\scriptstyle{}$} (A7_2);
    \path (A6_5) edge [-] node [auto] {$\scriptstyle{}$} (A6_6);
    \path (A0_5) edge [-] node [auto] {$\scriptstyle{}$} (A0_6);
    \path (A1_4) edge [-] node [auto] {$\scriptstyle{}$} (A3_3);
    \path (A6_6) edge [-] node [auto] {$\scriptstyle{}$} (A6_7);
    \path (A1_4) edge [-] node [auto] {$\scriptstyle{}$} (A2_5);
    \path (A1_4) edge [-] node [auto] {$\scriptstyle{}$} (A0_5);
    \path (A8_5) edge [-] node [auto] {$\scriptstyle{}$} (A8_6);
    \path (A2_2) edge [-] node [auto] {$\scriptstyle{}$} (A3_3);
    \path (A7_4) edge [-] node [auto] {$\scriptstyle{}$} (A3_3);
    \path (A0_0) edge [-] node [auto] {$\scriptstyle{}$} (A0_1);
    \path (A7_4) edge [-] node [auto] {$\scriptstyle{}$} (A6_5);
    \path (A2_5) edge [-] node [auto] {$\scriptstyle{}$} (A2_6);
    \path (A7_4) edge [-] node [auto] {$\scriptstyle{}$} (A8_5);
    \path (A2_6) edge [-] node [auto] {$\scriptstyle{}$} (A2_7);
    \path (A2_0) edge [-] node [auto] {$\scriptstyle{}$} (A2_1);
    \path (A0_1) edge [-] node [auto] {$\scriptstyle{}$} (A0_2);
    \path (A0_2) edge [-] node [auto] {$\scriptstyle{}$} (A3_3);
  \end{tikzpicture}
  \]
Such a graph is made of a distinguished vertex $v$ and several linear components.
These linear components are joined to $v$ via a final vertex (on the left-hand side of the picture above)
or via an internal vertex (right-hand side). 
\end{section}
\section{Motivations and obstructions}\label{motiv}
In this section we start dealing with rational homology cobordisms. As a motivation, 
we first explain in Proposition \ref{motivation} how rational homology cobordisms of 3-manifolds are relevant
for link concordance problems. Then, in Proposition \ref{mainobstruction} we state our lattice theoretical obstruction which will be used
in the proof of Theorem \ref{main}.

Two closed, oriented 3-manifolds $Y_1$,$Y_2$ are \emph{rational homology 
cobordant} ( or $\mathbb{Q}$H-cobordant ) if there exists a smooth compact 
4-manifold $W$ such that:
\begin{itemize}
\item $\partial W=Y_1\cup - Y_2$
\item both inclusions $Y_i\rightarrow W$ induce isomorphisms
$H_*(Y_i;\mathbb{Q})\longleftrightarrow H_*(W;\mathbb{Q})$.
\end{itemize}
The set of oriented rational homology spheres up to rational homology 
cobordism is an abelian group with the operation induced by connected sum. 
We denote this group by $\Theta_{\mathbb{Q}}^3$, the zero element is given by 
(the equivalence class of) $S^3$. Note that $Y$ is $\mathbb{Q}$H- cobordant to
$S^3$ if and only if it bounds a smooth rational homology ball.

It is well known that if a rational homology sphere is obtained as the branched double cover along a slice knot then it bounds a rational homology ball. 
In the next proposition we make an analogous observation concerning
branched double covers along slice links with more than one component.
\begin{Pro} \label{motivation}
Let $L\subset S^3$ be a link. Let $S\subset D^4$ be a properly embedded smooth surface without closed components
such that $\partial S=L$. Let $W$ be the double cover of $D^4$ branched along $S$. Assume that
$$
b_1(\partial W)\leq \chi(S)-1.
$$
Then $b_1(W)=\chi(S)-1$ and $b_2(W)=b_3(W)=0$. In particular, if $b_1(\partial W)>0$  we have an isomorphism
$$
H_*(W;\mathbb{Q})=H_*(\natural_{i=1}^{\chi(S)-1} S^1\times D^3;\mathbb{Q}).
$$
\end{Pro}
\begin{proof}
As shown in \cite{Lee:1} we have a long exact sequence
$$
\dots\rightarrow H_i(D^4,S\cup S^3)\rightarrow H_i(W,\partial W)\rightarrow H_i(D^4,S^3)
\rightarrow H_{i-1}(D^4,S\cup S^3)\rightarrow\dots
$$
from which we obtain an isomorphism $H_1(D^4,S\cup S^3)=H_1(W,\partial W)$. It follows
from the exact sequence of the pair that $H_1(D^4,S\cup S^3)=0$. We conclude that 
$0=H_1(W,\partial W)=H_3(W)$.
From the exact sequence of the pair $(W,\partial W)$ with rational 
coefficients we get
$$
\dots\rightarrow H_1(\partial W)\rightarrow H_1(W)\rightarrow 0.
$$
We obtain 
$$
b_1(W)\leq b_1(\partial W)\leq\chi(S)-1.$$
Since
$$
\chi (W)=2\chi(B^4)-\chi(S)=2-\chi(S) \Rightarrow 1-b_1(W)+b_2(W)=2-\chi(S)
$$
we see that $b_1(W)=\chi(S)-1$ and $b_2(W)=0$. 
\end{proof}
\begin{co}\label{slice}
Let $L$ be a slice link with $n$ components ($n>1$). 
Let $W$ be the branched double cover of the four-ball branched along a 
collection of slicing discs for $L$. We have an isomorphism
$$
H_*(W;\mathbb{Q})=H_*(\natural_{i=1}^{n-1} S^1\times D^3;\mathbb{Q})
$$
\end{co}
\begin{proof}
It is well known that $b_1(\partial W)=|L|-1$ (see for instance \cite{Kauffman:2}).
Then, we may apply Proposition \ref{motivation}.
\end{proof}

Motivated by Proposition \ref{motivation} we investigate $\mathbb{Q}H$-cobordisms of plumbed 
3-manifolds with $b_1\geq 1$. Note that if a 3-manifold $Y$ bounds a 
$\mathbb{Q}H-\natural_n S^1\times D^3$ then $b_1(Y)$ equals the number of $S^1\times D^3$ summands. 
\begin{Pro}\label{mainobstruction}
 Let $Y$ be a connected 3-manifold with $b_1(Y)=n$. Suppose that $Y$ bounds smooth 4-manifolds $X$ and $W$
 with the following properties:
 \begin{itemize}
  \item $X$ is simply connected, negative semidefinite and $rkQ_X=b_2(X)-n$
  \item $H_*(W,\mathbb{Q})=H_*(\natural_{i=1}^{n} S^1\times D^3;\mathbb{Q})$
 \end{itemize}
Then there exists a surjective morphism of integral lattices
$$
((H_2(X);\mathbb{Z}),Q_X)\longrightarrow (\mathbb{Z}^{b_2(X)-n},-Id).
$$
In particular for every definite sublattice $(G;Q_G)\subset (H_2(X);\mathbb{Z})$ whose rank is 
${b_2(X)-n}$ we obtain an embedding of integral lattices
$$
(G,Q_G)\longrightarrow (\mathbb{Z}^{b_2(X)-n},-Id)
$$
\end{Pro}
\begin{proof}
 Consider the smooth 4-manifold $X':=X\cup_Y -W$. 
 The Mayer-Vietoris exact sequence with integral coefficients reads
 $$
  \dots\rightarrow H_2(Y)\rightarrow H_2(X)\oplus H_2(W)\rightarrow H_2(X')\rightarrow H_1(Y)
  \rightarrow H_1(W)\rightarrow H_1(X')\rightarrow 0
  $$ 
Note that $b_1(Y)=b_1(W)$, moreover the map 
$H_1(Y;\mathbb{Q})\rightarrow H_1(W;\mathbb{Q})$ is an isomorphism. 
It follows that $b_1(X')=0$ and the kernel of the map $H_1(Y)\rightarrow H_1(W)$ 
is contained in the
the torsion subgroup of $H_1(Y)$. Call this kernel $T$. The group $H_2(W)$ is finite, 
let us call it $T'$. We obtain an exact sequence
$$
  \dots\rightarrow H_2(Y)\rightarrow H_2(X)\oplus T'\rightarrow H_2(X')\rightarrow T\rightarrow 0
$$
This yields another exact sequence
$$
  \dots\rightarrow H_2(Y)\rightarrow H_2(X)\rightarrow F(H_2(X'))\rightarrow 0
$$
where $F(H_2(X'))$ is the free summand of $H_2(X')$.
Since $b_3(X')=0$ we see that the free summand of $H_2(Y)$ injects into $H_2(X)$. Therefore we get the exact sequence
$$
  0\rightarrow F(H_2(Y))\rightarrow H_2(X)\rightarrow F(H_2(X'))\rightarrow 0
$$
therefore $b_2(X')=b_2(X)-b_2(Y)=b_2(X)-n$. Now note that $\sigma(X')=\sigma(X)$. This shows that
$X'$ is a smooth, closed negative definite 4-manifold, by Donaldson's diagonalization theorem
its intersection form is equivalent to the standard negative definite form on 
$\mathbb{Z}^{b_2(X')}$. The inclusion $X\rightarrow X'$ 
induces the desired morphism of integral lattices.
\end{proof}
\section{Constructing $\mathbb{Q}H$-cobordisms}\label{2-2}
In this section we introduce a procedure for constructing rational homology cobordisms between
plumbed 3-manifolds, our method is explained in Proposition \ref{construction}. We then introduce
some \emph{elementary building blocks} which are sufficient to produce all manifolds satisfying
the hypotheses of Theorem \ref{main} which bound rational homology $S^1\times D^3$'s.

Recall that a \emph{rooted} plumbing graph $(\Gamma,v)$ is a plumbing graph with
a distinguished vertex. In particular, a rooted plumbing graph is necessarly nonempty.  
\begin{de}
Let $(\Gamma_1,v_1)$ and $(\Gamma_2,v_2)$ be two rooted plumbing graphs. 
Let $\Gamma$ be the plumbing graph obtained from $\Gamma_1\sqcup\Gamma_2$
by identifing the two distinguished vertices and taking the sum of the corresponding weights.
We say that $\Gamma$ is obtained by \emph{joining} together $\Gamma_1$ and $\Gamma_2$ along
$v_1$ and $v_2$ and we write
$$
\Gamma:= \Gamma_1\bigvee_{v_1,v_2}\Gamma_2
$$
\end{de}
The following proposition follows immediately from Proposition \ref{cf}
\begin{Pro}\label{adcf}
 With the above notation we have
 $$
 cf(\Gamma_1\bigvee_{v_1,v_2}\Gamma_2)=cf(\Gamma_1)+cf(\Gamma_2)
 $$
 provided that the continued fractions on the right are computed with respect to the vertices 
 $v_1$ and $v_2$, and the continued fraction on the left is computed with respect to the vertex 
 resulting from joining $v_1$ and $v_2$.
\end{Pro}
\begin{lem}\label{canc}
 Let $W$ be a connected 4-dimensional handlebody without 3-handles. If 
 $H_*(\partial W;\mathbb{Q})=H_*(S^3;\mathbb{Q})$ then $H_1(W;\mathbb{Q})=0$.
 
 In particular, if $W$ is built using a single 1-handle $h^1$, and a single 2-handle $h^2$
 then the algebraic intersection of these handles does not vanish.
 \end{lem}
\begin{proof}
 The homology exact sequence of the pair $(W,\partial W)$ with rational coefficients reads
 $$
 \dots\rightarrow H_1(\partial W)\rightarrow H_1(W)\rightarrow H_1(W,\partial W)\rightarrow 0
 $$
 Since $H_1(\partial W)=0$ and by Lefschetz duality $H_1(W,\partial W)=H^3(W)=0$ 
 the conclusion follows. If there are only two handles $h^1$ and $h^2$, the attaching sphere
 of $h^2$ must have nonzero intersection number with the belt sphere of $h^1$, otherwise
 $h^1$ would represent a non trivial element in $H_1(W)$.
\end{proof}
The following lemma is an immediate consequence of the splitting move.
\begin{lem}\label{zerojoin}
 Let $(a_1,\dots,a_n)$ and $(b_1,\dots,b_m)$ be strings (where each coefficient
 is $\leq -2$). The 3-manifold described by the plumbing graph
     \[
  \begin{tikzpicture}[xscale=1.2,yscale=-0.5]
    \node (A0_0) at (0, 0) {$b_m$};
    \node (A0_2) at (2, 0) {$b_1$};
    \node (A1_0) at (0, 1) {$\bullet$};
    \node (A1_1) at (1, 1) {$\dots$};
    \node (A1_2) at (2, 1) {$\bullet$};
    \node (A2_3) at (3, 2) {$-1$};
    \node (A2_4) at (4, 2) {$0$};
    \node (A3_3) at (3, 3) {$\bullet$};
    \node (A3_4) at (4, 3) {$\bullet$};
    \node (A4_0) at (0, 4) {$a_1$};
    \node (A4_2) at (2, 4) {$a_n$};
    \node (A5_0) at (0, 5) {$\bullet$};
    \node (A5_1) at (1, 5) {$\dots$};
    \node (A5_2) at (2, 5) {$\bullet$};
    \path (A5_2) edge [-] node [auto] {$\scriptstyle{}$} (A3_3);
    \path (A5_0) edge [-] node [auto] {$\scriptstyle{}$} (A5_1);
    \path (A1_0) edge [-] node [auto] {$\scriptstyle{}$} (A1_1);
    \path (A1_1) edge [-] node [auto] {$\scriptstyle{}$} (A1_2);
    \path (A1_2) edge [-] node [auto] {$\scriptstyle{}$} (A3_3);
    \path (A3_3) edge [-] node [auto] {$\scriptstyle{}$} (A3_4);
    \path (A5_1) edge [-] node [auto] {$\scriptstyle{}$} (A5_2);
  \end{tikzpicture}
  \]
is a connected sum of two lens spaces.
\end{lem}
\begin{re}\label{zerojoinremark}
 If in the previous lemma we choose two complementary strings
 the plumbing graph depicted above, with the $0$-weighted vertex removed, represents $S^1\times S^2$. 
 However, not every linear plumbing graph that represents $S^1\times S^2$ has this form. 
 Apart from some obvious examples like
   \[
  \begin{tikzpicture}[xscale=1.1,yscale=-0.5]
    \node (A0_0) at (0, 0) {$0$};
    \node (A0_2) at (2, 0) {$-1$};
    \node (A0_3) at (3, 0) {$-1$};
    \node (A0_5) at (5, 0) {$-1$};
    \node (A0_6) at (6, 0) {$-2$};
    \node (A0_7) at (7, 0) {$-1$};
    \node (A1_0) at (0, 1) {$\bullet$};
    \node (A1_1) at (1, 1) {$;$};
    \node (A1_2) at (2, 1) {$\bullet$};
    \node (A1_3) at (3, 1) {$\bullet$};
    \node (A1_4) at (4, 1) {$;$};
    \node (A1_5) at (5, 1) {$\bullet$};
    \node (A1_6) at (6, 1) {$\bullet$};
    \node (A1_7) at (7, 1) {$\bullet$};
    \path (A1_2) edge [-] node [auto] {$\scriptstyle{}$} (A1_3);
    \path (A1_6) edge [-] node [auto] {$\scriptstyle{}$} (A1_7);
    \path (A1_5) edge [-] node [auto] {$\scriptstyle{}$} (A1_6);
  \end{tikzpicture}
  \]
there are also examples where there is a $-1$-weighted internal vertex. For instance
   \[
  \begin{tikzpicture}[xscale=1.1,yscale=-0.5]
    \node (A0_0) at (0, 0) {$0$};
    \node (A0_1) at (1, 0) {$-1$};
    \node (A0_2) at (2, 0) {$0$};
    \node (A0_4) at (4, 0) {$-1$};
    \node (A0_5) at (5, 0) {$-1$};
    \node (A0_6) at (6, 0) {$-1$};
    \node (A0_7) at (7, 0) {$-1$};
    \node (A0_8) at (8, 0) {$-1$};
    \node (A1_0) at (0, 1) {$\bullet$};
    \node (A1_1) at (1, 1) {$\bullet$};
    \node (A1_2) at (2, 1) {$\bullet$};
    \node (A1_3) at (3, 1) {$;$};
    \node (A1_4) at (4, 1) {$\bullet$};
    \node (A1_5) at (5, 1) {$\bullet$};
    \node (A1_6) at (6, 1) {$\bullet$};
    \node (A1_7) at (7, 1) {$\bullet$};
    \node (A1_8) at (8, 1) {$\bullet$};
    \path (A1_6) edge [-] node [auto] {$\scriptstyle{}$} (A1_7);
    \path (A1_0) edge [-] node [auto] {$\scriptstyle{}$} (A1_1);
    \path (A1_4) edge [-] node [auto] {$\scriptstyle{}$} (A1_5);
    \path (A1_1) edge [-] node [auto] {$\scriptstyle{}$} (A1_2);
    \path (A1_5) edge [-] node [auto] {$\scriptstyle{}$} (A1_6);
    \path (A1_7) edge [-] node [auto] {$\scriptstyle{}$} (A1_8);
  \end{tikzpicture}
  \]
These examples show that the assumption on the weights in Lemma \ref{zerojoin} is necessary. 
This follows from the fact that removing the central vertex in the two graphs above one obtains 
a plumbing graph that represents $S^1\times S^2\sharp S^1\times S^2$, 
instead of a rational homology sphere.

\end{re}

\begin{Pro}\label{construction}
 Let $(\Gamma,v)$ be a rooted plumbing graph such that $\partial P(\Gamma)=S^1\times S^2$ 
 and $\partial P(\Gamma\setminus\{v\})$ is a rational homology sphere. 
 Let $(\Gamma',v')$ be any rooted plumbing graph.  
 
 Then $b_1(\partial P(\Gamma'))=b_1(\partial P(\Gamma'\bigvee_{v',v}\Gamma))$ and these manifolds
 are $\mathbb{Q}H$-cobordant. 
 
\end{Pro}
\begin{proof}
 In Figure \ref{movie1}(a)
 we have a surgery description for $\partial P\Gamma'$. First we attach a 4-dimensional 1-handle to  
 $\partial P\Gamma'\times I$ as shown in Figure \ref{movie1}(b). In Figure \ref{movie1}(c) we draw the boundary of the four manifold
 obtained after the 1-handle attachmnet. This is just $\partial P\Gamma'\sharp S^1\times S^2$.
 In Figure \ref{movie1}(d) we draw the same manifold replacing the $0$-framed circle with the surgery diagram associated to the graph
 $\Gamma$. 
 Now we attach a 4-dimensional 2-handle as shown in Figure In Figure \ref{movie1}(e). 
 Via a zero-absorption move the result of this 2-handle attachment is a 4-manifold whose bottom boundary is  
 $\partial P(\Gamma'\bigvee_{v',v}\Gamma)$. This is shown in Figure \ref{movie1}(f).
  \begin{figure}[H]  
  {\includegraphics[scale=0.07]{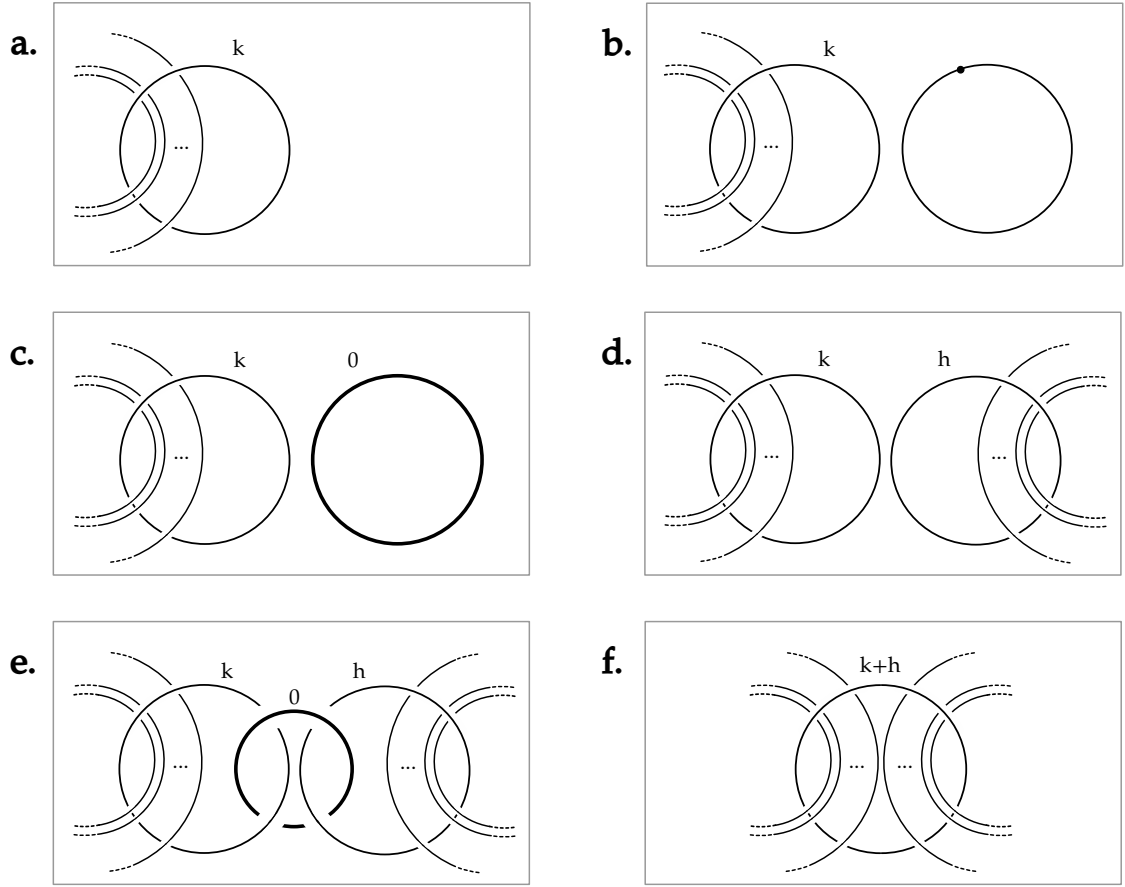}}
  \centering
  \caption{A rational homology cobordism between $\partial P\Gamma'$ and $\partial P(\Gamma'\bigvee_{v',v}\Gamma)$.}
  \label{movie1}
\end{figure}
We have constructed
a cobordism $W$ between $\partial P\Gamma'$ and $\partial P(\Gamma'\bigvee_{v',v}\Gamma)$ which
consists of one $1$-handle and one $2$-handle. 
In order to prove that $W$ is in fact a 
$\mathbb{Q}H$-cobordism it suffices to check that the algebraic intersection between the attaching
sphere of the 2-handle and the belt sphere of the 1-handle does not vanish.\\ 
Let us write $\alpha$ for the attaching sphere of the 2-handle.
The first homology group of $\partial P\Gamma'\sharp S^1\times S^2$ is 
$\mathbb{Q}^{b_1(\partial P\Gamma')}\oplus \mathbb{Q}$. 
Our algebraic intersection number is non zero if and only if
$\alpha$ represents a non trivial element when projected into
$H_1(S^1\times S^2)$. 
Note that in $H_1(\partial P\Gamma'\sharp S^1\times S^2)$
the curve $\alpha$ is homologous to the pair of curves $\alpha_1$ and $\alpha_2$ shown in Figure
\ref{cerchietti}.
\begin{figure}[H]  
  {\includegraphics[scale=0.06]{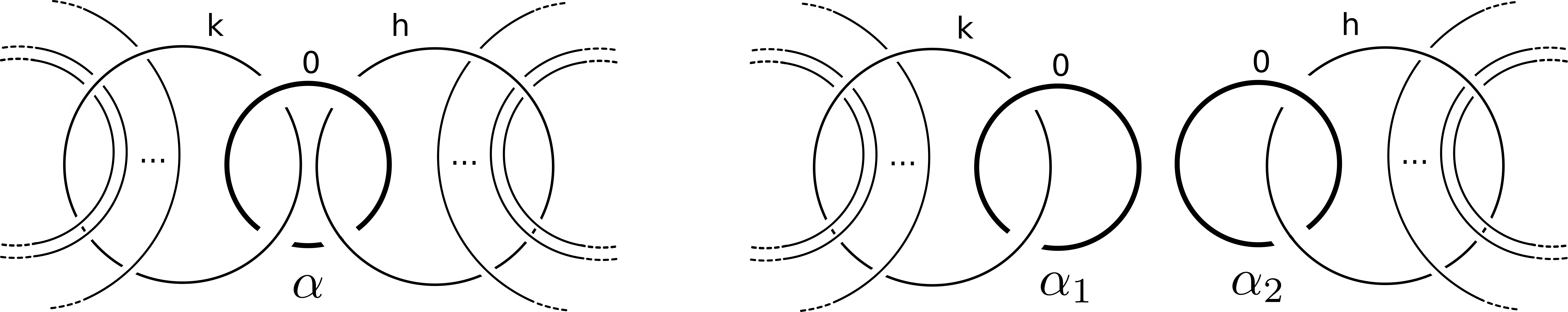}}
  \centering
  \caption{The thick curve on the leftmost diagram is homologous to the sum of the two thick curves on the rightmost diagram.}
  \label{cerchietti}
\end{figure}
This means that
the projection of $\alpha$ in $H_1(S^1\times S^2)$ is equivalent to $\alpha_2$.
The fact that $\alpha_2$ is a nontrivial element in $H_1(S^1\times S^2)$ follows immediately
from our hypotheses on $(\Gamma,v)$. 
To see this, let $\widetilde{L}$ be the link that gives a surgery description for $S^1\times S^2$ in 
Figure \ref{movie1}(d). 
Applying the splitting move on the link $\alpha_2\cup \widetilde{L}$ we see that the 3-manifold described by this link
is precisely $\partial P(\Gamma\setminus\{v\})$, which by our assumption is a rational homology sphere. 
This fact ensures that $\alpha_2$ represents a non trivial element
in $H_1(S^1\times S^2;\mathbb{Q})$. \\
It follows that 
$b_1(\partial P\Gamma')=b_1(\partial P(\Gamma'\bigvee_{v',v}\Gamma))$ and that $W$ is a 
$\mathbb{Q}H$-cobordism.
\end{proof}
\begin{re}
 The simplest way to use the above Proposition si to choose $(\Gamma,v)$ as any graph
 like the ones in Remark \ref{zerojoinremark}, the vertex $v$ being the one whose weight is $-1$.
\end{re}

\begin{re}
 The 2-handle
attachement used in Proposition \ref{construction} can also be described in terms of plumbing graphs as follows. 
We start with $\partial P(\Gamma'\sqcup\Gamma)$ which has the following description
  \[
  \begin{tikzpicture}[xscale=1.2,yscale=-0.5]
    \node (A0_6) at (6, 0) {$a_1$};
    \node (A0_8) at (8, 0) {$a_n$};
    \node (A1_0) at (0, 1) {};
    \node (A1_6) at (6, 1) {$\bullet$};
    \node (A1_7) at (7, 1) {$\dots$};
    \node (A1_8) at (8, 1) {$\bullet$};
    \node (A2_1) at (1, 2) {$w(v')$};
    \node (A2_5) at (5, 2) {$-1$};
    \node (A3_0) at (0, 3) {};
    \node (A3_1) at (1, 3) {$\bullet$};
    \node (A3_5) at (5, 3) {$\bullet$};
    \node (A4_6) at (6, 4) {$b_1$};
    \node (A4_8) at (8, 4) {$b_m$};
    \node (A5_0) at (0, 5) {};
    \node (A5_6) at (6, 5) {$\bullet$};
    \node (A5_7) at (7, 5) {$\dots$};
    \node (A5_8) at (8, 5) {$\bullet$};
    \path (A1_6) edge [-] node [auto] {$\scriptstyle{}$} (A1_7);
    \path (A3_5) edge [-] node [auto] {$\scriptstyle{}$} (A5_6);
    \path (A5_7) edge [-] node [auto] {$\scriptstyle{}$} (A5_8);
    \path (A3_0) edge [-] node [auto] {$\scriptstyle{}$} (A3_1);
    \path (A3_5) edge [-] node [auto] {$\scriptstyle{}$} (A1_6);
    \path (A1_0) edge [-] node [auto] {$\scriptstyle{}$} (A3_1);
    \path (A5_0) edge [-] node [auto] {$\scriptstyle{}$} (A3_1);
    \path (A5_6) edge [-] node [auto] {$\scriptstyle{}$} (A5_7);
    \path (A1_7) edge [-] node [auto] {$\scriptstyle{}$} (A1_8);
  \end{tikzpicture}
  \]
  Where, for simplicity we have choosen $\Gamma$ as in Lemma \ref{zerojoin}.
  The 2-handle then appears as an additional vertex as shown below.
   \[
  \begin{tikzpicture}[xscale=1.2,yscale=-0.5]
    \node (A0_4) at (4, 0) {$a_1$};
    \node (A0_6) at (6, 0) {$a_n$};
    \node (A1_0) at (0, 1) {};
    \node (A1_4) at (4, 1) {$\bullet$};
    \node (A1_5) at (5, 1) {$\dots$};
    \node (A1_6) at (6, 1) {$\bullet$};
    \node (A2_1) at (1, 2) {$w(v')$};
    \node (A2_2) at (2, 2) {$0$};
    \node (A2_3) at (3, 2) {$-1$};
    \node (A3_0) at (0, 3) {};
    \node (A3_1) at (1, 3) {$\bullet$};
    \node (A3_2) at (2, 3) {$\bullet$};
    \node (A3_3) at (3, 3) {$\bullet$};
    \node (A4_4) at (4, 4) {$b_1$};
    \node (A4_6) at (6, 4) {$b_m$};
    \node (A5_0) at (0, 5) {};
    \node (A5_4) at (4, 5) {$\bullet$};
    \node (A5_5) at (5, 5) {$\dots$};
    \node (A5_6) at (6, 5) {$\bullet$};
    \path (A1_4) edge [-] node [auto] {$\scriptstyle{}$} (A1_5);
    \path (A3_0) edge [-] node [auto] {$\scriptstyle{}$} (A3_1);
    \path (A1_5) edge [-] node [auto] {$\scriptstyle{}$} (A1_6);
    \path (A3_3) edge [-] node [auto] {$\scriptstyle{}$} (A1_4);
    \path (A1_0) edge [-] node [auto] {$\scriptstyle{}$} (A3_1);
    \path (A5_0) edge [-] node [auto] {$\scriptstyle{}$} (A3_1);
    \path (A5_4) edge [-] node [auto] {$\scriptstyle{}$} (A5_5);
    \path (A3_3) edge [-] node [auto] {$\scriptstyle{}$} (A5_4);
    \path (A3_2) edge [-] node [auto] {$\scriptstyle{}$} (A3_3);
    \path (A5_5) edge [-] node [auto] {$\scriptstyle{}$} (A5_6);
    \path (A3_1) edge [-] node [auto] {$\scriptstyle{}$} (A3_2);
  \end{tikzpicture}
  \]
  This last level of the cobordism can be described also by the following plumbing graph,
  using the 0-chain absorption move.
  \[
  \begin{tikzpicture}[xscale=2.2,yscale=-0.5]
    \node (A0_2) at (2, 0) {$a_1$};
    \node (A0_4) at (4, 0) {$a_n$};
    \node (A1_0) at (0, 1) {};
    \node (A1_2) at (2, 1) {$\bullet$};
    \node (A1_3) at (3, 1) {$\dots$};
    \node (A1_4) at (4, 1) {$\bullet$};
    \node (A2_1) at (1, 2) {$w(v')-1$};
    \node (A3_0) at (0, 3) {};
    \node (A3_1) at (1, 3) {$\bullet$};
    \node (A4_2) at (2, 4) {$b_1$};
    \node (A4_4) at (4, 4) {$b_m$};
    \node (A5_0) at (0, 5) {};
    \node (A5_2) at (2, 5) {$\bullet$};
    \node (A5_3) at (3, 5) {$\dots$};
    \node (A5_4) at (4, 5) {$\bullet$};
    \path (A5_2) edge [-] node [auto] {$\scriptstyle{}$} (A5_3);
    \path (A3_0) edge [-] node [auto] {$\scriptstyle{}$} (A3_1);
    \path (A3_1) edge [-] node [auto] {$\scriptstyle{}$} (A5_2);
    \path (A3_1) edge [-] node [auto] {$\scriptstyle{}$} (A1_2);
    \path (A1_0) edge [-] node [auto] {$\scriptstyle{}$} (A3_1);
    \path (A5_3) edge [-] node [auto] {$\scriptstyle{}$} (A5_4);
    \path (A5_0) edge [-] node [auto] {$\scriptstyle{}$} (A3_1);
    \path (A1_2) edge [-] node [auto] {$\scriptstyle{}$} (A1_3);
    \path (A1_3) edge [-] node [auto] {$\scriptstyle{}$} (A1_4);
  \end{tikzpicture}
  \]
\end{re}

\begin{ex}
 Let $(a_1,\dots,a_n)$ and $(b_1,\dots,b_m)$ be two complementary strings. The plumbing graph
 associated to the string $(a_1,\dots,a_n,-1,b_1,\dots,b_m)$ represents $S^1\times S^2$. 
 By the previous proposition all lens spaces associated to strings of the form
 $$
 (a_1,\dots,a_n,-2,b_1,\dots,b_m)
 $$
 are $\mathbb{Q}H$-cobordant to $S^3$. In fact, the correspnding plumbing graph is obtained by joining together
 a $-1$-weighted vertex and a graph as in Lemma \ref{zerojoin}.
\end{ex}
\begin{ex}\label{seifert}
 Choose strings $(a^i_{n_i},\dots,a^i_1,-1,b^i_1,\dots,b^i_{m_{i}})$, where $i=1,\dots,k$,
 as in the previous example. Consider the plumbed 3-manifold described by the following
 star-shaped plumbing graph
    \[
  \begin{tikzpicture}[xscale=2.5,yscale=-0.5]
    \node (A0_1) at (1, 0) {$a^1_1$};
    \node (A0_3) at (3, 0) {$a^1_{n_1}$};
    \node (A1_1) at (1, 1) {$\bullet$};
    \node (A1_2) at (2, 1) {$\dots$};
    \node (A1_3) at (3, 1) {$\bullet$};
    \node (A2_1) at (1, 2) {$b^1_1$};
    \node (A2_3) at (3, 2) {$b^1_{m_1}$};
    \node (A3_1) at (1, 3) {$\bullet$};
    \node (A3_2) at (2, 3) {$\dots$};
    \node (A3_3) at (3, 3) {$\bullet$};
    \node (A4_0) at (0, 4) {$-k$};
    \node (A4_2) at (2, 4) {$\vdots$};
    \node (A5_0) at (0, 5) {$\bullet$};
    \node (A5_2) at (2, 5) {$\vdots$};
    \node (A6_1) at (1, 6) {$a^k_1$};
    \node (A6_3) at (3, 6) {$a^k_{n_k}$};
    \node (A7_1) at (1, 7) {$\bullet$};
    \node (A7_2) at (2, 7) {$\dots$};
    \node (A7_3) at (3, 7) {$\bullet$};
    \node (A8_1) at (1, 8) {$b^k_1$};
    \node (A8_3) at (3, 8) {$b^k_{m_k}$};
    \node (A9_1) at (1, 9) {$\bullet$};
    \node (A9_2) at (2, 9) {$\dots$};
    \node (A9_3) at (3, 9) {$\bullet$};
    \path (A9_2) edge [-] node [auto] {$\scriptstyle{}$} (A9_3);
    \path (A7_2) edge [-] node [auto] {$\scriptstyle{}$} (A7_3);
    \path (A3_1) edge [-] node [auto] {$\scriptstyle{}$} (A3_2);
    \path (A1_1) edge [-] node [auto] {$\scriptstyle{}$} (A1_2);
    \path (A7_1) edge [-] node [auto] {$\scriptstyle{}$} (A7_2);
    \path (A5_0) edge [-] node [auto] {$\scriptstyle{}$} (A7_1);
    \path (A5_0) edge [-] node [auto] {$\scriptstyle{}$} (A9_1);
    \path (A5_0) edge [-] node [auto] {$\scriptstyle{}$} (A1_1);
    \path (A5_0) edge [-] node [auto] {$\scriptstyle{}$} (A3_1);
    \path (A1_2) edge [-] node [auto] {$\scriptstyle{}$} (A1_3);
    \path (A3_2) edge [-] node [auto] {$\scriptstyle{}$} (A3_3);
    \path (A9_1) edge [-] node [auto] {$\scriptstyle{}$} (A9_2);
  \end{tikzpicture}
  \]
By Proposition \ref{construction} such a manifold is $\mathbb{Q}H$-cobordant to $S^1\times S^2$
and thus it bounds a $\mathbb{Q}H-S^1\times D^3$. In Section \ref{mainchapter} we will see that these
are the only Seifert manifolds over the 2-sphere with this property.
\end{ex}
\subsection{Elementary building blocks}\label{lc1}
 In the previous example we have used the graph
  \[ \Gamma_1:=
  \begin{tikzpicture}[xscale=1.2,yscale=-0.5]
    \node (A0_0) at (0, 0) {$a_n$};
    \node (A0_2) at (2, 0) {$a_1$};
    \node (A0_3) at (3, 0) {$-1$};
    \node (A0_4) at (4, 0) {$b_1$};
    \node (A0_6) at (6, 0) {$b_m$};
    \node (A1_0) at (0, 1) {$\bullet$};
    \node (A1_1) at (1, 1) {$\dots$};
    \node (A1_2) at (2, 1) {$\bullet$};
    \node (A1_3) at (3, 1) {$\bullet$};
    \node (A1_4) at (4, 1) {$\bullet$};
    \node (A1_5) at (5, 1) {$\dots$};
    \node (A1_6) at (6, 1) {$\bullet$};
    \path (A1_4) edge [-] node [auto] {$\scriptstyle{}$} (A1_5);
    \path (A1_0) edge [-] node [auto] {$\scriptstyle{}$} (A1_1);
    \path (A1_1) edge [-] node [auto] {$\scriptstyle{}$} (A1_2);
    \path (A1_5) edge [-] node [auto] {$\scriptstyle{}$} (A1_6);
    \path (A1_2) edge [-] node [auto] {$\scriptstyle{}$} (A1_3);
    \path (A1_3) edge [-] node [auto] {$\scriptstyle{}$} (A1_4);
  \end{tikzpicture}
  \]
  as a building block for constructing rational homology cobordisms of  3-manifolds. 
  This is somehow the simplest
  way to use Proposition \ref{construction}. The process can be iterated by constructing 
  more complicated pieces to be used as building blocks. 
  
  Keeping in mind that we are interested in plumbed 3-manifolds with $lc=1$ we may introduce
  three more building blocks. The graph $\Gamma_1$ can be slightly modified obtaining
    \[
  \begin{tikzpicture}[xscale=1.1,yscale=-0.7]
    \node (A0_1) at (1, 0) {$a_n$};
    \node (A0_3) at (3, 0) {$a_1$};
    \node (A0_4) at (4, 0) {$-2$};
    \node (A0_5) at (5, 0) {$b_1$};
    \node (A0_7) at (7, 0) {$b_m$};
    \node (A1_1) at (1, 1) {$\bullet$};
    \node (A1_2) at (2, 1) {$\dots$};
    \node (A1_3) at (3, 1) {$\bullet$};
    \node (A1_4) at (4, 1) {$\bullet$};
    \node (A1_5) at (5, 1) {$\bullet$};
    \node (A1_6) at (6, 1) {$\dots$};
    \node (A1_7) at (7, 1) {$\bullet$};
    \node (A2_0) at (0, 2) {$\Gamma_2:=$};
    \node (A2_4) at (4, 2) {$\bullet$};
    \node (A3_4) at (4, 3) {$-1$};
    \path (A1_6) edge [-] node [auto] {$\scriptstyle{}$} (A1_7);
    \path (A1_4) edge [-] node [auto] {$\scriptstyle{}$} (A1_5);
    \path (A1_1) edge [-] node [auto] {$\scriptstyle{}$} (A1_2);
    \path (A1_5) edge [-] node [auto] {$\scriptstyle{}$} (A1_6);
    \path (A2_4) edge [-] node [auto] {$\scriptstyle{}$} (A1_4);
    \path (A1_2) edge [-] node [auto] {$\scriptstyle{}$} (A1_3);
    \path (A1_3) edge [-] node [auto] {$\scriptstyle{}$} (A1_4);
  \end{tikzpicture}
  \]
  Another building block can be obtained starting with 
    \[
  \begin{tikzpicture}[xscale=1.2,yscale=-0.5]
    \node (A0_0) at (0, 0) {$-2$};
    \node (A0_1) at (1, 0) {$-2$};
    \node (A0_3) at (3, 0) {$-2$};
    \node (A0_4) at (4, 0) {$-1$};
    \node (A0_5) at (5, 0) {$-n$};
    \node (A1_0) at (0, 1) {$\bullet$};
    \node (A1_1) at (1, 1) {$\bullet$};
    \node (A1_2) at (2, 1) {$\dots$};
    \node (A1_3) at (3, 1) {$\bullet$};
    \node (A1_4) at (4, 1) {$\bullet$};
    \node (A1_5) at (5, 1) {$\bullet$};
    \path (A1_3) edge [-] node [auto] {$\scriptstyle{}$} (A1_4);
    \path (A1_0) edge [-] node [auto] {$\scriptstyle{}$} (A1_1);
    \path (A1_4) edge [-] node [auto] {$\scriptstyle{}$} (A1_5);
    \path (A1_2) edge [-] node [auto] {$\scriptstyle{}$} (A1_3);
    \path (A1_1) edge [-] node [auto] {$\scriptstyle{}$} (A1_2);
  \end{tikzpicture}
  \]
where $n-1$ is the length of the $-2$-chain. This is just a special case of the previous 
building block. Now we join this graph with $\Gamma_1$ along the vertices of weight $-n$ and $-1$.
We obtain our third building block
   \[
  \begin{tikzpicture}[xscale=1.4,yscale=-0.5]
    \node (A0_6) at (6, 0) {$a_1$};
    \node (A0_8) at (8, 0) {$a_n$};
    \node (A1_6) at (6, 1) {$\bullet$};
    \node (A1_7) at (7, 1) {$\dots$};
    \node (A1_8) at (8, 1) {$\bullet$};
    \node (A2_1) at (1, 2) {$-2$};
    \node (A2_3) at (3, 2) {$-2$};
    \node (A2_4) at (4, 2) {$-1$};
    \node (A2_5) at (5, 2) {$-n-1$};
    \node (A3_0) at (0, 3) {$\Gamma_3:=$};
    \node (A3_1) at (1, 3) {$\bullet$};
    \node (A3_2) at (2, 3) {$\dots$};
    \node (A3_3) at (3, 3) {$\bullet$};
    \node (A3_4) at (4, 3) {$\bullet$};
    \node (A3_5) at (5, 3) {$\bullet$};
    \node (A4_6) at (6, 4) {$b_1$};
    \node (A4_8) at (8, 4) {$b_m$};
    \node (A5_6) at (6, 5) {$\bullet$};
    \node (A5_7) at (7, 5) {$\dots$};
    \node (A5_8) at (8, 5) {$\bullet$};
    \path (A1_6) edge [-] node [auto] {$\scriptstyle{}$} (A1_7);
    \path (A3_5) edge [-] node [auto] {$\scriptstyle{}$} (A5_6);
    \path (A5_7) edge [-] node [auto] {$\scriptstyle{}$} (A5_8);
    \path (A3_1) edge [-] node [auto] {$\scriptstyle{}$} (A3_2);
    \path (A3_4) edge [-] node [auto] {$\scriptstyle{}$} (A3_5);
    \path (A3_2) edge [-] node [auto] {$\scriptstyle{}$} (A3_3);
    \path (A3_5) edge [-] node [auto] {$\scriptstyle{}$} (A1_6);
    \path (A3_3) edge [-] node [auto] {$\scriptstyle{}$} (A3_4);
    \path (A5_6) edge [-] node [auto] {$\scriptstyle{}$} (A5_7);
    \path (A1_7) edge [-] node [auto] {$\scriptstyle{}$} (A1_8);
  \end{tikzpicture}
  \]
Note that $\partial P\Gamma_3=S^1\times S^2$. A fourth building block can be constructed as follows.
We start with 
  \[
  \begin{tikzpicture}[xscale=1.1,yscale=-0.5]
    \node (A0_0) at (0, 0) {$-2$};
    \node (A0_1) at (1, 0) {$-1$};
    \node (A0_2) at (2, 0) {$-2$};
    \node (A1_0) at (0, 1) {$\bullet$};
    \node (A1_1) at (1, 1) {$\bullet$};
    \node (A1_2) at (2, 1) {$\bullet$};
    \path (A1_0) edge [-] node [auto] {$\scriptstyle{}$} (A1_1);
    \path (A1_1) edge [-] node [auto] {$\scriptstyle{}$} (A1_2);
  \end{tikzpicture}
  \]
and then we attach to the final vertices of this graph two linear graphs like $\Gamma_1$.
We obtain
   \[
  \begin{tikzpicture}[xscale=1.2,yscale=-0.5]
    \node (A0_1) at (1, 0) {$a'_{n'}$};
    \node (A0_3) at (3, 0) {$a'_1$};
    \node (A0_7) at (7, 0) {$a_1$};
    \node (A0_9) at (9, 0) {$a_n$};
    \node (A1_1) at (1, 1) {$\bullet$};
    \node (A1_2) at (2, 1) {$\dots$};
    \node (A1_3) at (3, 1) {$\bullet$};
    \node (A1_7) at (7, 1) {$\bullet$};
    \node (A1_8) at (8, 1) {$\dots$};
    \node (A1_9) at (9, 1) {$\bullet$};
    \node (A2_4) at (4, 2) {$-3$};
    \node (A2_5) at (5, 2) {$-1$};
    \node (A2_6) at (6, 2) {$-3$};
    \node (A3_0) at (0, 3) {$\Gamma_4:=$};
    \node (A3_4) at (4, 3) {$\bullet$};
    \node (A3_5) at (5, 3) {$\bullet$};
    \node (A3_6) at (6, 3) {$\bullet$};
    \node (A4_1) at (1, 4) {$b'_{m'}$};
    \node (A4_3) at (3, 4) {$b'_1$};
    \node (A4_7) at (7, 4) {$b_1$};
    \node (A4_9) at (9, 4) {$b_m$};
    \node (A5_1) at (1, 5) {$\bullet$};
    \node (A5_2) at (2, 5) {$\dots$};
    \node (A5_3) at (3, 5) {$\bullet$};
    \node (A5_7) at (7, 5) {$\bullet$};
    \node (A5_8) at (8, 5) {$\dots$};
    \node (A5_9) at (9, 5) {$\bullet$};
    \path (A5_3) edge [-] node [auto] {$\scriptstyle{}$} (A3_4);
    \path (A5_8) edge [-] node [auto] {$\scriptstyle{}$} (A5_9);
    \path (A1_2) edge [-] node [auto] {$\scriptstyle{}$} (A1_3);
    \path (A3_6) edge [-] node [auto] {$\scriptstyle{}$} (A1_7);
    \path (A5_2) edge [-] node [auto] {$\scriptstyle{}$} (A5_3);
    \path (A5_7) edge [-] node [auto] {$\scriptstyle{}$} (A5_8);
    \path (A1_8) edge [-] node [auto] {$\scriptstyle{}$} (A1_9);
    \path (A3_4) edge [-] node [auto] {$\scriptstyle{}$} (A3_5);
    \path (A1_1) edge [-] node [auto] {$\scriptstyle{}$} (A1_2);
    \path (A3_5) edge [-] node [auto] {$\scriptstyle{}$} (A3_6);
    \path (A1_3) edge [-] node [auto] {$\scriptstyle{}$} (A3_4);
    \path (A5_1) edge [-] node [auto] {$\scriptstyle{}$} (A5_2);
    \path (A3_6) edge [-] node [auto] {$\scriptstyle{}$} (A5_7);
    \path (A1_7) edge [-] node [auto] {$\scriptstyle{}$} (A1_8);
  \end{tikzpicture}
  \]
Note that this last graph does not represent $S^1\times S^2$ since its normal form can be obtained
by blowing down the $-1$-vertex. Each of the four building blocks we have introduced have a 
distinguished $-1$-weighted vertex. From now on we will implicitly consider each of these graphs as 
a rooted plumbing graph where the prefered vertex is the one whose weight is $-1$.
\begin{de}\label{buildblock}	
The four families of rooted plumbing graphs introduced above will be called \emph{building blocks}
of the first, second, third and fourth type, respectively.
\end{de}
Of course, using Proposition \ref{construction}, one can construct many examples of plumbed
3-manifolds with arbitrarly high linear complexity that are $\mathbb{Q}H$-cobordant to $S^1\times S^2$.
A simple example is given by the following graph
    \[
  \begin{tikzpicture}[xscale=1.1,yscale=-0.5]
    \node (A0_0) at (0, 0) {$-2$};
    \node (A0_3) at (3, 0) {$-2$};
    \node (A1_0) at (0, 1) {$\bullet$};
    \node (A1_3) at (3, 1) {$\bullet$};
    \node (A2_1) at (1, 2) {$-2$};
    \node (A2_2) at (2, 2) {$-2$};
    \node (A2_5) at (5, 2) {$-2$};
    \node (A3_1) at (1, 3) {$\bullet$};
    \node (A3_2) at (2, 3) {$\bullet$};
    \node (A3_5) at (5, 3) {$\bullet$};
    \node (A4_0) at (0, 4) {$-2$};
    \node (A4_3) at (3, 4) {$-3$};
    \node (A4_4) at (4, 4) {$-2$};
    \node (A5_0) at (0, 5) {$\bullet$};
    \node (A5_3) at (3, 5) {$\bullet$};
    \node (A5_4) at (4, 5) {$\bullet$};
    \node (A6_5) at (5, 6) {$-2$};
    \node (A7_5) at (5, 7) {$\bullet$};
    \path (A5_4) edge [-] node [auto] {$\scriptstyle{}$} (A7_5);
    \path (A3_2) edge [-] node [auto] {$\scriptstyle{}$} (A5_3);
    \path (A3_1) edge [-] node [auto] {$\scriptstyle{}$} (A3_2);
    \path (A5_4) edge [-] node [auto] {$\scriptstyle{}$} (A3_5);
    \path (A3_2) edge [-] node [auto] {$\scriptstyle{}$} (A1_3);
    \path (A5_3) edge [-] node [auto] {$\scriptstyle{}$} (A5_4);
    \path (A5_0) edge [-] node [auto] {$\scriptstyle{}$} (A3_1);
    \path (A1_0) edge [-] node [auto] {$\scriptstyle{}$} (A3_1);
  \end{tikzpicture}
  \]
whose linear complexity is 2.

The following proposition is an immediate consequence of Proposition \ref{construction}.
\begin{Pro}\label{constructionbis}
 Let $\Gamma$ be a plumbing graph obtained by joining together two or more building blocks of any
 type along their $-1$-vertices. Then
 \begin{enumerate}
  \item $\Gamma$ is in normal form
  \item $lc(\Gamma)=1$
  \item $\partial P\Gamma$ bounds a $\mathbb{Q}H-S^1\times D^3$.
 \end{enumerate}
\end{Pro}
Our mani result, Theorem \ref{main}, should be thought of as a converse of this last proposition.
\section{Main results}\label{mainchapter}
In this section we state our main result, Theorem \ref{main}. We give a proof modulo a technical result,
Theorem \ref{technical} whose statement and proof are postponed to the next sections. We explain 
how to specialize our result to Seifert fibered spaces over the 2-sphere in Theorem \ref{mainbis}.

Before we state our main result we introduce some terminology. Let $\Gamma$
be a plumbing graph in normal form such that $lc(\Gamma)=1$. Choose $v\in\Gamma$
such that $\widetilde{\Gamma}:=\Gamma\setminus\{v\}$ is linear. The linear graph
$\widetilde{\Gamma}$ is a disjoint union of connected linear graphs 
$\Gamma_1,\dots,\Gamma_k$. We call $\Gamma_i$ a \emph{final leg} or an
\emph{internal leg} according to wether $v$ is linked to a final vector 
of $\Gamma_i$ or an internal one. 
We indicate with $i(\Gamma,v)$ and $f(\Gamma,v)$ the number of internal (resp. final)
legs of $\Gamma$. Finally each internal leg of $\Gamma$ has a distinguished vertex
which is 3-valent in $\Gamma$. We call these vertices the \emph{nodes} of 
$\Gamma$, and we indicate with $N(\Gamma)$ the set of all the nodes.
Note that, in some cases, these definitions 
depend on the choice of the vector $v$. This is the case for
three legged starshaped plumbing graphs (there are four choices for the
vector $v$) and plumbing graphs like
   \[
  \begin{tikzpicture}[xscale=1.0,yscale=-0.7]
    \node (A0_0) at (0, 0) {$\dots$};
    \node (A0_1) at (1, 0) {$\bullet$};
    \node (A0_4) at (4, 0) {$\bullet$};
    \node (A0_5) at (5, 0) {$\dots$};
    \node (A1_2) at (2, 1) {$\bullet$};
    \node (A1_3) at (3, 1) {$\bullet$};
    \node (A2_0) at (0, 2) {$\dots$};
    \node (A2_1) at (1, 2) {$\bullet$};
    \node (A2_4) at (4, 2) {$\bullet$};
    \node (A2_5) at (5, 2) {$\dots$};
    \path (A0_1) edge [-] node [auto] {$\scriptstyle{}$} (A1_2);
    \path (A2_4) edge [-] node [auto] {$\scriptstyle{}$} (A2_5);
    \path (A0_0) edge [-] node [auto] {$\scriptstyle{}$} (A0_1);
    \path (A1_3) edge [-] node [auto] {$\scriptstyle{}$} (A0_4);
    \path (A0_4) edge [-] node [auto] {$\scriptstyle{}$} (A0_5);
    \path (A2_1) edge [-] node [auto] {$\scriptstyle{}$} (A1_2);
    \path (A1_3) edge [-] node [auto] {$\scriptstyle{}$} (A2_4);
    \path (A1_2) edge [-] node [auto] {$\scriptstyle{}$} (A1_3);
    \path (A2_0) edge [-] node [auto] {$\scriptstyle{}$} (A2_1);
  \end{tikzpicture}
  \]
where there are two possible choices for the vector $v$.

\begin{te}\label{main}
 Let $\Gamma$ be a plumbing graph in normal form with $lc(\Gamma)=1$.
  Choose a vector $v\in\Gamma$ such that $\widetilde{\Gamma}:=\Gamma\setminus \{v\}$
  is linear. Suppose that each node of $\Gamma$ has weight less or equal to 
  $-2$ and that
  \begin{equation}\label{crucialhyp}
   I(\widetilde{\Gamma})\leq -f(\Gamma,v)-2i(\Gamma,v)-
   \sum_{u\in N(\Gamma)} max\{0,u\cdot u+3\}.
  \end{equation}
  The following conditions are equivalent:
  \begin{itemize}
   \item the 3-manifold $\partial P\Gamma$ bounds a $\mathbb{Q}H-S^1\times D^3$;
   \item equality holds in \eqref{crucialhyp} and
   $\Gamma$ is obtained by joining together building blocks along $-1$-vertices.
  \end{itemize}
\end{te}
\begin{proof}
If $\partial P\Gamma$ is obtained by joining together building blocks along
 $-1$-vectors then the conclusion follows from Proposition \ref{constructionbis}.
 
 Let $\Gamma$ be a plumbing graph in normal form satisfying the hypotheses of the theorem and let
 $W$ be a $\mathbb{Q}H-S^{1}\times D^3$ such that $\partial W=\partial P\Gamma$. Let $N$ be the number
 of vertices of $\Gamma$. Note that $b_0(\Gamma)=b_1(\partial P\Gamma)=1$, moreover 
 $H_2(P\Gamma;\mathbb{Z})$ contains a free subgroup of rank $N-1$ on which $Q_{\Gamma}$ 
 is negative definite (it is the subgroup $\mathbb{Z}\widetilde{\Gamma}$ spanned by all 
 vertices in $\widetilde{\Gamma}$). It
 follows that $Q_{\Gamma}$ is negative semidefinite, more precisely
 $$
 (b_0(\Gamma),b_-(\Gamma),b_+(\Gamma))=(1,N-1,0).
 $$
 Therefore we are in the situation described in Proposition \ref{mainobstruction}. There exists
 a morphism of integral lattices
 $$
 \Phi:(H_2(X(\Gamma);\mathbb{Z}),Q_{\Gamma})\longrightarrow (\mathbb{Z}^{N-1},-Id).
 $$
 Precomposing this map with the inclusion 
 $(\mathbb{Z}\widetilde{\Gamma},Q_{\widetilde{\Gamma}})
 \hookrightarrow(H_2(X(\Gamma);\mathbb{Z}),Q_{\Gamma})$ we obtain an embedding 
 of integral lattices
 $$
 \widetilde{\Phi}:(\mathbb{Z}\widetilde{\Gamma},Q_{\widetilde{\Gamma}})\longrightarrow (\mathbb{Z}^{N-1},-Id).
 $$
 Let us write $\{v_1,\dots,v_{N-1}\}$ for the set of vertices of $\widetilde{\Gamma}$. Now
 consider the subset $S:=\{\Phi(v_1),\dots,\Phi(v_{N-1})\}\subset\mathbb{Z}^{N-1}$. 
 The extra vector $\Phi(v)$
 is linked once to each connected component of $\widetilde{\Gamma}$ and is orthogonal to every other
 vector. The subset $S$ satisfies all the hypotheses of Theorem \ref{technical} and the
 conclusion follows. 
\end{proof}
Even though the class of plumbed 3-manifolds that satisfy the hypotheses of 
Theorem \ref{main} is quite large (it includes, up to orientation reversal all Seifert fibered spaces
over the 2-sphere with vanishing Euler invariant) some of the assumptions on the plumbing graph are rather technical and unnatural.
The need for these hypotheses can be explained as follows.

The fact that every vertex in $\widetilde{\Gamma}$ has weight less or equal to -2 allows us to avoid 
indefinite plumbing graphs. Consider, for instance, the following plumbing graph
     \[
  \begin{tikzpicture}[xscale=1.2,yscale=-0.5]
    \node (A0_1) at (1, 0) {$-2$};
    \node (A0_4) at (4, 0) {$-2$};
    \node (A1_1) at (1, 1) {$\bullet$};
    \node (A1_3) at (3, 1) {$-1$};
    \node (A1_4) at (4, 1) {$\bullet$};
    \node (A2_2) at (2, 2) {$-1$};
    \node (A2_3) at (3, 2) {$\bullet$};
    \node (A2_4) at (4, 2) {$-2$};
    \node (A3_0) at (0, 3) {$\Gamma:=$};
    \node (A3_2) at (2, 3) {$\bullet$};
    \node (A3_4) at (4, 3) {$\bullet$};
    \node (A4_1) at (1, 4) {$-2$};
    \node (A4_4) at (4, 4) {$-2$};
    \node (A5_1) at (1, 5) {$\bullet$};
    \node (A5_3) at (3, 5) {$-1$};
    \node (A5_4) at (4, 5) {$\bullet$};
    \node (A6_3) at (3, 6) {$\bullet$};
    \node (A6_4) at (4, 6) {$-2$};
    \node (A7_4) at (4, 7) {$\bullet$};
    \path (A3_2) edge [-] node [auto] {$\scriptstyle{}$} (A6_3);
    \path (A6_3) edge [-] node [auto] {$\scriptstyle{}$} (A7_4);
    \path (A3_2) edge [-] node [auto] {$\scriptstyle{}$} (A2_3);
    \path (A6_3) edge [-] node [auto] {$\scriptstyle{}$} (A5_4);
    \path (A2_3) edge [-] node [auto] {$\scriptstyle{}$} (A3_4);
    \path (A1_1) edge [-] node [auto] {$\scriptstyle{}$} (A3_2);
    \path (A5_1) edge [-] node [auto] {$\scriptstyle{}$} (A3_2);
    \path (A2_3) edge [-] node [auto] {$\scriptstyle{}$} (A1_4);
  \end{tikzpicture}
  \]
  Note that $\Gamma$ is in normal form. We have
  $$
  (b_0\Gamma,b_+\Gamma,b_-\Gamma)=(1,1,7)
  $$
  Moreover this plumbing graph is \emph{selfdual}, meaning that $\Gamma^*=\Gamma$, therefore
  reversing the orientation does not help. Theorem \ref{main} does not say if $\partial P\Gamma$
  bounds a $\mathbb{Q}H-S^1\times D^3$. However in this particular case $\partial P\Gamma$
  does bound a $\mathbb{Q}H-S^1\times D^3$. This can be checked easily using Proposition
  \ref{construction}. By splitting off three building blocks of the first type and then 
  applying the splitting move we obtain a $0$-weighted single vertex.
It follows that $\partial P\Gamma$ is $\mathbb{Q}H$-cobordant to $S^1\times S^2$.

The reason why we need the condition $I(\widetilde{\Gamma})<0$ can be explained as follows.
In the proof of Theorem \ref{main} we have shown that $\widetilde{\Gamma}$ gives rise to a 
subset $S=\{v_1,\dots,v_n\}\subset\mathbb{Z}^n$ with certain properties.
The starting point of our analysis is that these subsets are well understood provided that
$I(S)<0$. We use the known results on such subsets, as developed in \cite{Lisca:1} and
\cite{Lisca:2}, to show that the possible graphs of $S\cup\{v\}$ where $v$ is the vector that 
corresponds to the extra vertex in $\Gamma$ are obtained by joining together building blocks
along $-1$-vertices.
\subsection{More plumbing graphs}
The family of plumbing graphs described in Theorem \ref{main} is not the largest
family we can think of. As mentioned above we have intentionally avoided indefinite graphs.
Suppose, for istance, that we are given a plumbing graph in normal form $\Gamma$ with $lc(\Gamma)=1$ and that
we want to join this graph with a building block of the second type, denoted by $\Gamma_2$. 
The resulting plumbing graph can be depicted as
    \[
  \begin{tikzpicture}[xscale=1.5,yscale=-0.5]
    \node (A0_0) at (0, 0) {};
    \node (A0_3) at (3, 0) {$\bullet$};
    \node (A0_4) at (4, 0) {$\dots$};
    \node (A0_5) at (5, 0) {$\bullet$};
    \node (A1_1) at (1, 1) {$-a-1$};
    \node (A1_2) at (2, 1) {$-2$};
    \node (A2_0) at (0, 2) {};
    \node (A2_1) at (1, 2) {$\bullet$};
    \node (A2_2) at (2, 2) {$\bullet$};
    \node (A4_0) at (0, 4) {};
    \node (A4_3) at (3, 4) {$\bullet$};
    \node (A4_4) at (4, 4) {$\dots$};
    \node (A4_5) at (5, 4) {$\bullet$};
    \path (A0_3) edge [-] node [auto] {$\scriptstyle{}$} (A0_4);
    \path (A4_4) edge [-] node [auto] {$\scriptstyle{}$} (A4_5);
    \path (A2_1) edge [-] node [auto] {$\scriptstyle{}$} (A2_2);
    \path (A4_0) edge [-] node [auto] {$\scriptstyle{}$} (A2_1);
    \path (A0_4) edge [-] node [auto] {$\scriptstyle{}$} (A0_5);
    \path (A0_0) edge [-] node [auto] {$\scriptstyle{}$} (A2_1);
    \path (A2_2) edge [-] node [auto] {$\scriptstyle{}$} (A0_3);
    \path (A4_3) edge [-] node [auto] {$\scriptstyle{}$} (A4_4);
    \path (A2_0) edge [-] node [auto] {$\scriptstyle{}$} (A2_1);
    \path (A2_2) edge [-] node [auto] {$\scriptstyle{}$} (A4_3);
  \end{tikzpicture}
  \]
If we  reverse the orientation on this plumbed 3-manifold the relevant portion of the 
dual plumbing graph can be depicted as folllows.
\[
  \begin{tikzpicture}[xscale=1.7,yscale=-0.5]
    \node (A0_0) at (0, 0) {};
    \node (A0_3) at (3, 0) {$\bullet$};
    \node (A0_4) at (4, 0) {$\dots$};
    \node (A0_5) at (5, 0) {$\bullet$};
    \node (A1_1) at (1, 1) {$a'+1$};
    \node (A1_2) at (2, 1) {$0$};
    \node (A2_0) at (0, 2) {};
    \node (A2_1) at (1, 2) {$\bullet$};
    \node (A2_2) at (2, 2) {$\bullet$};
    \node (A4_0) at (0, 4) {};
    \node (A4_3) at (3, 4) {$\bullet$};
    \node (A4_4) at (4, 4) {$\dots$};
    \node (A4_5) at (5, 4) {$\bullet$};
    \path (A0_3) edge [-] node [auto] {$\scriptstyle{}$} (A0_4);
    \path (A4_4) edge [-] node [auto] {$\scriptstyle{}$} (A4_5);
    \path (A2_1) edge [-] node [auto] {$\scriptstyle{}$} (A2_2);
    \path (A4_0) edge [-] node [auto] {$\scriptstyle{}$} (A2_1);
    \path (A0_4) edge [-] node [auto] {$\scriptstyle{}$} (A0_5);
    \path (A0_0) edge [-] node [auto] {$\scriptstyle{}$} (A2_1);
    \path (A2_2) edge [-] node [auto] {$\scriptstyle{}$} (A0_3);
    \path (A4_3) edge [-] node [auto] {$\scriptstyle{}$} (A4_4);
    \path (A2_0) edge [-] node [auto] {$\scriptstyle{}$} (A2_1);
    \path (A2_2) edge [-] node [auto] {$\scriptstyle{}$} (A4_3);
  \end{tikzpicture}
  \]
  Here $a'$ is computed as in Theorem \ref{normalform}, ignoring the attached building block.
  Moreover the legs of the building block
  are not altered by this transformation (more precisely, they are turned into each other).
  To summarize the situation, consider the following diagram:
    \[
  \begin{tikzpicture}[xscale=2.5,yscale=-1.5]
    \node (A0_0) at (0, 0) {$\Gamma$};
    \node (A0_1) at (1, 0) {$\Gamma\bigvee\Gamma_2$};
    \node (A1_0) at (0, 1) {$\Gamma^*$};
    \node (A1_1) at (1, 1) {$(\Gamma\bigvee\Gamma_2)^*$};
    \path (A0_0) edge [->] node [auto] {$\scriptstyle{}$} (A0_1);
    \path (A0_0) edge [->] node [auto] {$\scriptstyle{}$} (A1_0);
    \path (A0_1) edge [->] node [auto] {$\scriptstyle{}$} (A1_1);
    \path (A1_0) edge [->] node [auto] {$\scriptstyle{}$} (A1_1);
  \end{tikzpicture}
  \]
This shows that there is a unique graph $\Gamma'$ such that
$(\Gamma\bigvee\Gamma_2)^*=\Gamma^*\bigvee\Gamma'$. We call
this graph a \emph{dual building block} of the second type and,
abusing notation, we indicate it with $\Gamma_2^*$. It can be written as
  \[
  \begin{tikzpicture}[xscale=1.5,yscale=-0.8]
    \node (A0_1) at (1, 0) {$+1$};
    \node (A0_2) at (2, 0) {$0$};
    \node (A0_3) at (3, 0) {$\bullet$};
    \node (A0_4) at (4, 0) {$\dots$};
    \node (A0_5) at (5, 0) {$\bullet$};
    \node (A1_0) at (0, 1) {$\Gamma_2^*:=$};
    \node (A1_1) at (1, 1) {$\bullet$};
    \node (A1_2) at (2, 1) {$\bullet$};
    \node (A2_3) at (3, 2) {$\bullet$};
    \node (A2_4) at (4, 2) {$\dots$};
    \node (A2_5) at (5, 2) {$\bullet$};
    \path (A0_4) edge [-] node [auto] {$\scriptstyle{}$} (A0_5);
    \path (A1_2) edge [-] node [auto] {$\scriptstyle{}$} (A0_3);
    \path (A1_2) edge [-] node [auto] {$\scriptstyle{}$} (A2_3);
    \path (A0_3) edge [-] node [auto] {$\scriptstyle{}$} (A0_4);
    \path (A1_1) edge [-] node [auto] {$\scriptstyle{}$} (A1_2);
    \path (A2_3) edge [-] node [auto] {$\scriptstyle{}$} (A2_4);
    \path (A2_4) edge [-] node [auto] {$\scriptstyle{}$} (A2_5);
  \end{tikzpicture}
  \]

   Note that $\Gamma_2^*$ is not the dual graph of $\Gamma_2$ in the usual sense,
   since $\Gamma_2$ is not in normal form this notion does not make sense.
   Arguing as above we can define dual building blocks of any type, and it is easy to see that 
   a building block of the first type is self dual. 
   A dual building block of the third type can be written as
      \[
  \begin{tikzpicture}[xscale=1.2,yscale=-0.8]
    \node (A0_1) at (1, 0) {$-n$};
    \node (A0_2) at (2, 0) {$0$};
    \node (A0_3) at (3, 0) {$n-1$};
    \node (A0_4) at (4, 0) {$\bullet$};
    \node (A0_5) at (5, 0) {$\dots$};
    \node (A0_6) at (6, 0) {$\bullet$};
    \node (A1_0) at (0, 1) {$\Gamma_3^*:=$};
    \node (A1_1) at (1, 1) {$\bullet$};
    \node (A1_2) at (2, 1) {$\bullet$};
    \node (A1_3) at (3, 1) {$\bullet$};
    \node (A2_4) at (4, 2) {$\bullet$};
    \node (A2_5) at (5, 2) {$\dots$};
    \node (A2_6) at (6, 2) {$\bullet$};
    \path (A2_4) edge [-] node [auto] {$\scriptstyle{}$} (A2_5);
    \path (A1_3) edge [-] node [auto] {$\scriptstyle{}$} (A0_4);
    \path (A2_5) edge [-] node [auto] {$\scriptstyle{}$} (A2_6);
    \path (A1_1) edge [-] node [auto] {$\scriptstyle{}$} (A1_2);
    \path (A0_4) edge [-] node [auto] {$\scriptstyle{}$} (A0_5);
    \path (A1_3) edge [-] node [auto] {$\scriptstyle{}$} (A2_4);
    \path (A1_2) edge [-] node [auto] {$\scriptstyle{}$} (A1_3);
    \path (A0_5) edge [-] node [auto] {$\scriptstyle{}$} (A0_6);
  \end{tikzpicture}
  \]

while a dual building block of the fourth type is a graph of the form
   \[
  \begin{tikzpicture}[xscale=1.2,yscale=-0.8]
    \node (A0_1) at (1, 0) {$\bullet$};
    \node (A0_2) at (2, 0) {$\dots$};
    \node (A0_3) at (3, 0) {$\bullet$};
    \node (A0_4) at (4, 0) {$1$};
    \node (A0_5) at (5, 0) {$1$};
    \node (A0_6) at (6, 0) {$1$};
    \node (A0_7) at (7, 0) {$\bullet$};
    \node (A0_8) at (8, 0) {$\dots$};
    \node (A0_9) at (9, 0) {$\bullet$};
    \node (A1_0) at (0, 1) {$\Gamma_4^*:=$};
    \node (A1_4) at (4, 1) {$\bullet$};
    \node (A1_5) at (5, 1) {$\bullet$};
    \node (A1_6) at (6, 1) {$\bullet$};
    \node (A2_1) at (1, 2) {$\bullet$};
    \node (A2_2) at (2, 2) {$\dots$};
    \node (A2_3) at (3, 2) {$\bullet$};
    \node (A2_7) at (7, 2) {$\bullet$};
    \node (A2_8) at (8, 2) {$\dots$};
    \node (A2_9) at (9, 2) {$\bullet$};
    \path (A2_8) edge [-] node [auto] {$\scriptstyle{}$} (A2_9);
    \path (A1_4) edge [-] node [auto] {$\scriptstyle{}$} (A1_5);
    \path (A0_3) edge [-] node [auto] {$\scriptstyle{}$} (A1_4);
    \path (A0_8) edge [-] node [auto] {$\scriptstyle{}$} (A0_9);
    \path (A0_1) edge [-] node [auto] {$\scriptstyle{}$} (A0_2);
    \path (A2_1) edge [-] node [auto] {$\scriptstyle{}$} (A2_2);
    \path (A1_5) edge [-] node [auto] {$\scriptstyle{}$} (A1_6);
    \path (A1_6) edge [-] node [auto] {$\scriptstyle{}$} (A2_7);
    \path (A2_2) edge [-] node [auto] {$\scriptstyle{}$} (A2_3);
    \path (A0_7) edge [-] node [auto] {$\scriptstyle{}$} (A0_8);
    \path (A2_7) edge [-] node [auto] {$\scriptstyle{}$} (A2_8);
    \path (A1_6) edge [-] node [auto] {$\scriptstyle{}$} (A0_7);
    \path (A0_2) edge [-] node [auto] {$\scriptstyle{}$} (A0_3);
    \path (A2_3) edge [-] node [auto] {$\scriptstyle{}$} (A1_4);
  \end{tikzpicture}
  \]

   Therefore we have seven elementary pieces
   that we can use to build plumbed 3-manifolds with $lc=1$ that bound $\mathbb{Q}H-S^1\times D^3$'s.
   We consider these dual building blocks as rooted plumbing graphs. The distinguished vertex of a dual is the same
   as the distinguished vertex of the old building block but it has a different weight.
   In the above pictures all adjacent legs are complementary and each vertex on a leg has weight $\leq -2$.
   We summarize the construction in the next proposition.
   \begin{Pro}\label{largestfamily}
    Let $\Gamma$ be a plumbing graph obtained by joining together two or more building block and/or their duals.
    Then, $\Gamma$ is a plumbing graph in normal form with $lc(\Gamma)=1$ and $\partial P\Gamma$ bounds a $\mathbb{Q}H-S^1\times D^3$.
   \end{Pro}
   \begin{proof}
    We may apply Proposition \ref{construction} to every building block or argue as follows.
    By Proposition \ref{construction}, up to $\mathbb{Q}H$-cobordism equivalence we can remove
    all the original building blocks. We are left with a graph obtained using
    only dual building blocks. Changing the orientation of the manifold and taking 
    the corresponding dual graph every dual building block turns into a regular one. We conclude aplying
    again Proposition \ref{construction}.
   \end{proof}

\begin{q}
 Does every plumbed 3-manifold with $lc(Y)=1$ that bounds a $\mathbb{Q}H-S^1\times D^3$ arise from
 the construction given in Proposition \ref{largestfamily}?
\end{q}
Trying to answer affermatively the above question would require invariants of $\mathbb{Q}H$-cobordisms
that do not rely on definite (or semidefinite) intersection forms.
\subsection{Seifert manifolds}
As we show in the next theorem, the assumption $I(\widetilde{\Gamma})<0$ in Theorem \ref{main} can be avoided when both
$\Gamma$ and $\Gamma^*$ are negative semidefinite. This fact is not true for every
graph with $lc(\Gamma)=1$ and $b_0(\Gamma)=1$. It is true, however, if we restrict ourselves to
starshaped plumbing graphs.

Before we state our main result on Seifert manifolds we give a necessary condition
for a starshaped plumbing graph to represent a manifold that bounds a $\mathbb{Q}H-S^1\times D^3$.
\begin{Pro}\label{necessary}
 Let $\Gamma$ be a starshaped plumbing graph in normal form, and let $\widetilde{\Gamma}$ 
 be the linear plumbing graph obtained by removing the central vertex from $\Gamma$. 
 If $\partial P\Gamma$ bounds a $\mathbb{Q}H-S^1\times D^3$, then the connected sum of Lens spaces
 $\partial P\widetilde{\Gamma}$ bounds a $\mathbb{Q}H-D^4$.
\end{Pro}
\begin{proof}
 Let $W$ be a $\mathbb{Q}H-S^1\times D^3$ such that $\partial W=\partial P\Gamma$. Any 2-handle
 attachment on $W$ that turns its boundary into a rational homology sphere will produce
 a rational homology ball. In particular, we may attach to $W$ a $0$-framed 2-handle linked
 once to the central vertex of $\Gamma$, obtaining a 4-manifold $\widetilde{W}$. Its boundary can 
 be depicted as
 \[
  \begin{tikzpicture}[xscale=1.2,yscale=-0.5]
    \node (A0_2) at (2, 0) {$\Gamma_1$};
    \node (A1_2) at (2, 1) {$\cdot$};
    \node (A2_0) at (0, 2) {$0$};
    \node (A2_1) at (1, 2) {$a$};
    \node (A2_2) at (2, 2) {$\cdot$};
    \node (A3_0) at (0, 3) {$\bullet$};
    \node (A3_1) at (1, 3) {$\bullet$};
    \node (A3_2) at (2, 3) {$\cdot$};
   
    \node (A4_2) at (2, 4) {$\cdot$};
    \node (A5_2) at (2, 5) {$\cdot$};
    \node (A6_2) at (2, 6) {$\Gamma_k$};
    \path (A3_0) edge [-] node [auto] {$\scriptstyle{}$} (A3_1);
    \path (A3_1) edge [-] node [auto] {$\scriptstyle{}$} (A6_2);
    \path (A3_1) edge [-] node [auto] {$\scriptstyle{}$} (A0_2);
  \end{tikzpicture}
  \]
  Using the splitting move we see that 
  $\partial \widetilde{W}=\partial P\Gamma_1\sharp\dots\sharp\partial P\Gamma_1$. Since $\Gamma$
  is in normal form we have $det(\Gamma_i)\neq 0$ for each $i$. Therefore $\partial \widetilde{W}$ 
  is a rational homology sphere.  
\end{proof}
The reason why the above proposition is relevant is that, by \cite{Lisca:2}, we know exactly
which connected sums of Lens spaces bound rational homology balls. Comparing the next theorem
with the results in \cite{Lisca:2} we see that Proposition \ref{necessary} does not give sufficient
conditions. For instance, no starshaped plumbing graph in normal form with an odd number of legs
bounds a $\mathbb{Q}H-S^1\times D^3$. 
\begin{te}\label{mainbis}
 A Seifert fibered manifold 
 $Y=(0;b;(\alpha_1,\beta_1),\dots,(\alpha_h,\beta_h))$ 
 bounds a $\mathbb{Q}H-S^1\times D^3$ if and only if the Seifert invariants occur in complementary pairs
 and $e(Y)=0$.
\end{te}
\begin{proof}
Assume that the Seifert invariants occur in complementary pairs and that $e(Y)=0$. By Theorem
\ref{seiferttheorem} we may write $Y=\partial P\Gamma$, where $\Gamma$ is the following plumbing
graph in normal form.
 \[
  \begin{tikzpicture}[xscale=1.5,yscale=-0.5]
    \node (A0_1) at (1, 0) {$a^1_1$};
    \node (A0_3) at (3, 0) {$a^1_{n_1}$};
    \node (A1_1) at (1, 1) {$\bullet$};
    \node (A1_2) at (2, 1) {$\dots$};
    \node (A1_3) at (3, 1) {$\bullet$};
    \node (A2_1) at (1, 2) {$b^1_1$};
    \node (A2_3) at (3, 2) {$b^1_{m_1}$};
    \node (A3_1) at (1, 3) {$\bullet$};
    \node (A3_2) at (2, 3) {$\dots$};
    \node (A3_3) at (3, 3) {$\bullet$};
    \node (A4_0) at (0, 4) {$-b$};
    \node (A4_2) at (2, 4) {$\vdots$};
    \node (A5_0) at (0, 5) {$\bullet$};
    \node (A5_2) at (2, 5) {$\vdots$};
    \node (A6_1) at (1, 6) {$a^k_1$};
    \node (A6_3) at (3, 6) {$a^k_{n_k}$};
    \node (A7_1) at (1, 7) {$\bullet$};
    \node (A7_2) at (2, 7) {$\dots$};
    \node (A7_3) at (3, 7) {$\bullet$};
    \node (A8_1) at (1, 8) {$b^k_1$};
    \node (A8_3) at (3, 8) {$b^k_{m_k}$};
    \node (A9_1) at (1, 9) {$\bullet$};
    \node (A9_2) at (2, 9) {$\dots$};
    \node (A9_3) at (3, 9) {$\bullet$};
    \path (A9_2) edge [-] node [auto] {$\scriptstyle{}$} (A9_3);
    \path (A7_2) edge [-] node [auto] {$\scriptstyle{}$} (A7_3);
    \path (A3_1) edge [-] node [auto] {$\scriptstyle{}$} (A3_2);
    \path (A1_1) edge [-] node [auto] {$\scriptstyle{}$} (A1_2);
    \path (A7_1) edge [-] node [auto] {$\scriptstyle{}$} (A7_2);
    \path (A5_0) edge [-] node [auto] {$\scriptstyle{}$} (A7_1);
    \path (A5_0) edge [-] node [auto] {$\scriptstyle{}$} (A9_1);
    \path (A5_0) edge [-] node [auto] {$\scriptstyle{}$} (A1_1);
    \path (A5_0) edge [-] node [auto] {$\scriptstyle{}$} (A3_1);
    \path (A1_2) edge [-] node [auto] {$\scriptstyle{}$} (A1_3);
    \path (A3_2) edge [-] node [auto] {$\scriptstyle{}$} (A3_3);
    \path (A9_1) edge [-] node [auto] {$\scriptstyle{}$} (A9_2);
  \end{tikzpicture}
  \]
Here the legs are pairwise complementary. 
Call $\Gamma^a_1,\Gamma^b_1,\dots,\Gamma^a_k,\Gamma^b_k$
the legs of $\Gamma$. The condition $e(Y)=0$ implies that $b=-k$.
Indeed
$$
0=e(Y)=cf(\Gamma)=b-\sum_{i=1}^{k}\left(\frac{1}{cf(\Gamma^a_i)}+\frac{1}{cf(\Gamma^b_i)}\right)=
b+\sum_{i=1}^{k}1.
$$
The conclusion follows from Proposition \ref{construction}, as explained in Example \ref{seifert}.

Now assume that $Y$ bounds a $\mathbb{Q}H-S^1\times D^3$. Then, so does $-Y$. Let $\Gamma$ and 
$\Gamma^*$ be their plumbing graphs in normal form, and let $\widetilde{\Gamma}$ and $\widetilde{\Gamma}^*$
be the graphs obtained from $\Gamma$ and $\Gamma^*$ by removing the central vertices. Note that
$\widetilde{\Gamma}^*$ is in fact the dual of $\widetilde{\Gamma}$, so there is no ambiguity with this 
notation. By Proposition \ref{I} we have
$$
I(\widetilde{\Gamma})+I(\widetilde{\Gamma}^*)=-2k
$$
where $k$ is the number of legs of $\Gamma$. In particular we may assume, for instance, that
$I(\widetilde{\Gamma})\leq k$ and apply Theorem \ref{main}. $\Gamma$ is obtained by joining 
together building blocks along their $-1$-vertices. Since $\Gamma$ is starshaped,
only building blocks of the first type may occur, which means that $Y$ belongs to the family
described in Example \ref{seifert}.
\end{proof}
\begin{re}{The case $b_1>1$}
The class of plumbed 3-manifolds admitting a plumbing graph in normal form with $lc(\Gamma)=1$
contains manifolds with arbitrary high first Betti number. For example, consider 
 the following plumbing graph in normal form
   \[
  \begin{tikzpicture}[xscale=1.2,yscale=-0.5]
    \node (A0_0) at (0, 0) {$-2$};
    \node (A0_4) at (4, 0) {$-2$};
    \node (A1_0) at (0, 1) {$\bullet$};
    \node (A1_1) at (1, 1) {$-1$};
    \node (A1_3) at (3, 1) {$-1$};
    \node (A1_4) at (4, 1) {$\bullet$};
    \node (A2_1) at (1, 2) {$\bullet$};
    \node (A2_3) at (3, 2) {$\bullet$};
    \node (A3_0) at (0, 3) {$-2$};
    \node (A3_2) at (2, 3) {$-1$};
    \node (A3_4) at (4, 3) {$-2$};
    \node (A4_0) at (0, 4) {$\bullet$};
    \node (A4_2) at (2, 4) {$\bullet$};
    \node (A4_4) at (4, 4) {$\bullet$};
    \node (A5_2) at (2, 5) {$\ \ \ \ \ \ -1$};
    \node (A6_2) at (2, 6) {$\bullet$};
    \node (A7_1) at (1, 7) {$-2$};
    \node (A7_3) at (3, 7) {$-2$};
    \node (A8_1) at (1, 8) {$\bullet$};
    \node (A8_3) at (3, 8) {$\bullet$};
    \path (A1_0) edge [-] node [auto] {$\scriptstyle{}$} (A2_1);
    \path (A8_3) edge [-] node [auto] {$\scriptstyle{}$} (A6_2);
    \path (A2_1) edge [-] node [auto] {$\scriptstyle{}$} (A4_2);
    \path (A8_1) edge [-] node [auto] {$\scriptstyle{}$} (A6_2);
    \path (A6_2) edge [-] node [auto] {$\scriptstyle{}$} (A4_2);
    \path (A4_4) edge [-] node [auto] {$\scriptstyle{}$} (A2_3);
    \path (A1_4) edge [-] node [auto] {$\scriptstyle{}$} (A2_3);
    \path (A2_3) edge [-] node [auto] {$\scriptstyle{}$} (A4_2);
    \path (A4_0) edge [-] node [auto] {$\scriptstyle{}$} (A2_1);
  \end{tikzpicture}
  \]
Its signature is $(b_+,b_-,b_0)=(1,7,2)$. This graph is obtained by joining
three blocks of the first type to the graph
  \[
  \begin{tikzpicture}[xscale=1.2,yscale=-0.5]
    \node (A0_1) at (1, 0) {$0$};
    \node (A1_1) at (1, 1) {$\bullet$};
    \node (A2_0) at (0, 2) {$-1$};
    \node (A2_1) at (1, 2) {$0$};
    \node (A3_0) at (0, 3) {$\bullet$};
    \node (A3_1) at (1, 3) {$\bullet$};
    \node (A4_1) at (1, 4) {$0$};
    \node (A5_1) at (1, 5) {$\bullet$};
    \path (A3_0) edge [-] node [auto] {$\scriptstyle{}$} (A5_1);
    \path (A3_0) edge [-] node [auto] {$\scriptstyle{}$} (A3_1);
    \path (A3_0) edge [-] node [auto] {$\scriptstyle{}$} (A1_1);
  \end{tikzpicture}
  \]
Since this last graph represents $S^1\times S^2\sharp S^1\times S^2$ we conclude,
by Proposition \ref{construction},
that $\partial P\Gamma$ bounds a $\mathbb{Q}H-(S^1\times D^3\natural S^1\times D^3)$. This example
can be easily generalized to produce infinitely many plumbed 3-manifolds $\partial P\Gamma$
where
\begin{itemize}
 \item $lc(\Gamma)=1$
 \item $b_0(\Gamma)$ is arbitrarily large
 \item $\partial P\Gamma$ bounds a $\mathbb{Q}H-\natural_{i=1}^{b_0(\Gamma)-1}(S^1\times D^3)$
\end{itemize} 
\end{re}

\begin{section}{The language of linear subsets}\label{3basic}
In this section we start our technical analysis needed to complete the proof of Theorem 
\ref{main}. We begin providing a brief introduction to the language of good subsets
and we prove Lemma \ref{dis}, which will be used extensively throughout later on. In Section
\ref{3main} we state the main technical results, Theorems \ref{technical} and \ref{irreducible}, and
explain the strategy of the proofs. In Section \ref{3irr} we carry out a detailed analysis of certain good subsets
and we conclude by proving Theorem \ref{irreducible}. 
In Section \ref{3ortho} we prove what we need to fill the gap between Theorem \ref{technical} and Theorem
\ref{irreducible}. Finally in Section \ref{3conclusion} we give the proof of Theorem \ref{technical}.

 An \emph{intersection lattice} is a pair $(G,Q_G)$ of a free abelian group $G$ together with a $\mathbb{Z}$-valued 
 symmetric bilinear form on it. We indicate with $(\mathbb{Z}^N,-Id)$ the intersection lattice 
 with the standard negative definite form defined by
 $$
 e_i\cdot e_j=-\delta_{ij}.
 $$
 We will always work with $\mathbb{Z}^N$ with the above form on it, so in most cases
 we will omit the form and indicate the intersection lattice simply by $\mathbb{Z}^N$. 
 Let  $S=\{v_1,\dots,v_N\}\subset\mathbb{Z}^N$ be such that
 \begin{itemize}
  \item $v_i\cdot v_i\leq -2$
  \item $v_i\cdot v_{j}\in\{0,1\}$ if $i\neq j$ 
 \end{itemize}
Define the \emph{intersection graph} of $S$ as the graph having a vertex for each element of $S$
and an edge for every pair $(v_i,v_j)$ such that $v_i\cdot v_j=1$. We indicate this graph with
$\Gamma_S$. The graph $\Gamma_S$ can be given integral weights on its vertices: 
the weight of the vertex corresponding to $v_i$ is $v_i\cdot v_i$.
\begin{de}
 A subset $S\subset\mathbb{Z}^N$ satisfying the above properties is said to be a \emph{linear subset}
 whenever $\Gamma_S$ is a linear graph. We will also say that $S$ is \emph{treelike} whenever 
 its graph is a tree. In this case we require that $v_i\cdot v_i\leq -2$ only when
 $v_i$ corresponds to a vertex on a linear chain.
\end{de}
Note that the graph of a treelike subset is a plumbing graph in normal form. We will use all 
the terminology we have introduced for plumbing graphs and intersection forms in this new context
without stating the obvious definitions. For example, given a linear subset $S$, a vector $v\in S$ 
can be isolated, internal or final just like the vertex of a plumbing graph. 

Given $v\in\mathbb{Z}^N$ and some basis vector $e_i$ we say that $e_i$ \emph{hits} $v$ (or that
$v$ hits $e_i$) if $v\cdot e_i\neq 0$. Two vectors $v,w\in\mathbb{Z}^N$ are \emph{linked} if there
exists a basis vector that hits both of them. A subset $S\subset\mathbb{Z}^N$ is \emph{irreducible}
if for every pair of vectors $v,w\in S$ there exists a sequence of vectors in $S$
$$
v_0=v,v_1,\dots,v_n=w
$$
such that $v_i$ and $v_{i+1}$ are linked for $i=0,\dots,n-1$. A subset which is not irreducible
is said to be \emph{reducible}. A linear irreducible subset is called a \emph{good subset}.
A good subset whose graph is connected is a \emph{standard subset}. We indicate with $c(S)$ the
number of connected components of $\Gamma_S$. This should not be confused with
the number of irreducible components, for which we do not introduce any simbol. In general
an irreducible component may have a graph consisting of several connected components.

There are some elementary operations that, under certain assumptions, can be performed
on a linear subset in order to obtain a smaller linear subset. Here we restrict
ourselves to \emph{$-2$-final expansions} and \emph{$-2$-final contractions} because
these are the only operations that we need. In \cite{Lisca:1} a more general notion 
of expansions and contraction is used.
We indicate with $\pi_h:\mathbb{Z}^N\rightarrow \mathbb{Z}^{N-1}$ the projection
onto the subgroup $<e_1,\dots,e_{h-1},e_{h+1},\dots,e_N>$.
\begin{de}\label{expansionscontractions}
 Let $S=\{v_1,\dots,v_n\}\subset\mathbb{Z}^n$ be a linear subset. 
 Suppose that there exists $e_i$ such that
 \begin{itemize}
  \item $e_i$ only hits two vectors $v_h$ and $v_k$
  \item one of these vectors, say $v_h$, is final 
  \item $v_h\cdot v_h=-2$ and $v_k\cdot v_k<-2$
 \end{itemize}
We say that the subset $S':=\pi_{h}(S\setminus\{v_h\})$ is obtained from $S$ by 
\emph{$-2$-final contraction} and we write $S\searrow S'$. We also say that $S$ is obtained from
$S'$ by \emph{$-2$-final expansion} and we write $S'\nearrow S$.
\end{de}
If we think of a subset $S\subset\mathbb{Z}^{n-1}$ as a square matrix whose columns are the vectors
$v_1,\dots,v_n$, then a $-2$-final contraction consists in removing one column and one row
provided that the above conditions are satisfied. Note that a $-2$-final contraction (or expansion)
of a linear subset $S$ is again a linear subset $S'$ whose graph $\Gamma_{S'}$ has the same number of components
as $\Gamma_{S}$.
\begin{de}\label{badcomponent}
 Let $S'=\{v_1,\dots,v_N\}\subset\mathbb{Z}^N$, $N\geq 3$ be a good subset. Let
 $C'=\{v_{s-1},v_s,v_{s+1}\}\subset S'$ be such that $\Gamma_{C'}$ is a connected component of
 $\Gamma_{S'}$ with $v_{s-1}\cdot v_{s-1}=v_{s+1}\cdot v_{s+1}=-2$ and $v_s\cdot v_s<-2$. 
 Suppose that there exists $e_j$ which hits all the vectors in $C'$ and no other vector
 of $S'$. Let $S$ be a subset obtained from $S'$ via a sequence of $-2$-final expansions
 performed on $C'$. The component $C\subset S$ corresponding to $C'\subset S'$
 is called a \emph{bad component} of the good subset $S$.
\end{de}
We indicate the number of bad components of a good subset with $b(S)$.
Given $v_1,\dots,v_j$ elements of a linear subset we also define
$$
E(v_1,\dots,v_j):=|\{ \ k \ |\ e_k\cdot v_1\neq 0,\dots,e_k\cdot v_j\neq 0\}|
$$
The situation we need to study is that of a linear subset together with an extra 
vector $v$ which is orthogonal to all but one vector, say $w_i$, of each connected component
$S_i$ of $S$ and $v\cdot w_i=1$. 
This last condition is expressed by saying that \emph{$v$ is linked once to $w$}.

The following lemmas will be used several times in the next sections. 
\begin{lem}\label{basic}
 Let $\Gamma$ be a linear plumbing graph in normal form with connected components
 $\Gamma_1,\dots,\Gamma_k$. Choose vertices $v_i\in\Gamma_i$ where $1\leq i\leq k$.
 Let $\Gamma'$ be the graph obtained from $\Gamma$ by adding a new vertex $v$ 
 with weight $w(v)\leq -1$ and new edges for the pairs $(v,v_i)$.
 If $det\Gamma'=0$, then one of the following holds:
 \begin{itemize}
  \item $w(v)>-k$
  \item $w(v_j)=-2$ \quad for some $j\in\{1,\dots,k\}$.
 \end{itemize}
\end{lem}
\begin{proof}
Since $det\Gamma'=0$, by Proposition \ref{cfdet} we must have $cf\Gamma=0$. Computing 
$cf\Gamma$ with respect to the vertex $v$, using Proposition \ref{cf}, we obtain
$$
w(v)-\sum_{i=1}^k\frac{1}{w(v_i)-\frac{1}{\alpha_i}-\frac{1}{\beta_i}}=0,
$$
where $\alpha_i$ and $\beta_i$ are the continued fractions of the two components
of $\Gamma_i\setminus\{v_i\}$, rooted at the vertices adjacent to $v_i$. Note that
if $v_i$ is final there is only one component. In this case we set $1/\beta_i=0$.
Suppose that for each $1\leq j\leq k$ we have $w(v_j)\leq -3$. We need to prove that $w(v)> -k$.
Each $\alpha_i$ (and $\beta_i$ if $v_i$ is internal) is the continued fraction of a linear
connected plumbing graph in normal form rooted at a final vertex. Therefore $\alpha_i,\beta_i<-1$ and,
since $w(v_j)\leq -3$ we have
$$
w(v_i)-\frac{1}{\alpha_i}-\frac{1}{\beta_i}<-1\Rightarrow \sum_{i=1}^k\frac{1}{w(v_i)-\frac{1}{\alpha_i}-\frac{1}{\beta_i}}>-k
$$
Combining this fact with the expression for $cf\Gamma$ we obtain $w(v)>-k$ and we are done.
\end{proof}
\begin{lem}\label{dis}
Let $S\subset\mathbb{Z}^N$ be a linear subset. Let $S_1,\dots,S_n$ be the connected components
of $S$. Suppose there is a vector $v\in\mathbb{Z}^N$ which is \emph{linked once} to a vector of each
$S_i$, say $v_i$,(i.e. $v\cdot v_i=1$) and is orthogonal to every other vector of $S_i\setminus{v_i}$. Then
$$
v\cdot v>\sum_{i=1}^{n}v_i\cdot v_i.
$$
\end{lem}
\begin{proof}
 Let $M$ be the $N\times N$ matrix whose columns are the elements of $S$. The conditions 
 on the extra vector $v$ can be expressed as a linear system of equations, namely
 \begin{equation}\label{matrixeq}
  ^tMv=\sum_{i=1}^n e_{k_i}
 \end{equation}
where the $k_i$-th column of $M$ is $v_i$. Multiplying both sides of \eqref{matrixeq} by $M$
we get
\begin{equation}
 M ^tMv=M\sum_{i=1}^n e_{k_i}=\sum_{i=1}^n v_i.
\end{equation}
The matrix $M ^tM$ is conjugated to $^tMM$, in particular they have the same eigenvalues.
The matrix $-^tMM$ represents the intersection form of $P\Gamma_S$. It consists of $n$ blocks,
one for each connected component of $S$.
Each block can be diagonalized as shown in chapter V of \cite{Eisenbud} , the $k$-th eigenvalue is given by the
negative continued fractions corresponding to the first $k$ diagonal entries. In particular, 
it is easy to prove by induction that, for each eigenvalue $\lambda$, we have $\lambda<-1$. 
It follows that
$$
||v||^2<||M ^tMv||^2=||\sum_{i=1}^n v_i||^2=\sum_{i=1}^n ||v_i||^2.
$$
Where $||\cdot||$ denotes the usual Euclidean norm. Rewriting the above inequality using
the standard negative definite product in $\mathbb{Z}^N$ we obtain 
$$
v\cdot v>(M ^tMv)\cdot (M ^tMv)=(\sum_{i=1}^n v_i)\cdot(\sum_{i=1}^n v_i)=\sum_{i=1}^n v_i\cdot v_i.
$$
\end{proof}
\end{section}
\begin{section}{Main results and strategy of the proof}\label{3main}
The key technical result that will complete the proof of Theorem \ref{main} is the
following.

\begin{te}\label{technical}
 Let $S\subset\mathbb{Z}^N$ be a linear subset. 
 Suppose that there exists $v\in\mathbb{Z}^N$ which is linked once
 to a vector of each connected component of $S$ and is orthogonal to any other vector 
 of $S$. Assume also that, with the notation introduced in Section \ref{mainchapter} we have
 \begin{equation}\label{crucialhyptec}
 I(\Gamma_S)\leq -f(\Gamma_{S\cup\{v\}},v)-2i(\Gamma_{S\cup\{v\}},v)-
   \sum_{u\in N(\Gamma_{S\cup\{v\}})} max\{0,u\cdot u+3\}.  
 \end{equation}
 Then, $\Gamma_{S\cup\{v\}}$ can be obtained by joining together two or more building
 blocks along their $-1$-vertices.  
 \end{te}

The main ingredient for the proof of Theorem \ref{technical} is the following result
 which explains that the irreducible components of the 
 given subset together with the corresponding extra vector give rise to building blocks.
 
 \begin{te}\label{irreducible}
 Let $S\subset\mathbb{Z}^N$ be a good subset such that $I(S)+c(S)\leq0$ and $I(S)+b(S)<0$. 
 Suppose there exists $v\in\mathbb{Z}^N$ which is linked once
 to a vector of each connected component of $S$ and is orthogonal to all the other vectors 
 of $S$. Then, $v\cdot v=-1$ and $\Gamma_{S\cup\{v\}}$ is a building block.
\end{te}
 
 The idea of the proof of Theorem \ref{irreducible} is the following. The assumtpions on $S$
 are chosen so that, by the results of \cite{Lisca:2} the subset $S$ falls in one of the following
 classes:
 \begin{enumerate}
  \item $c(S)=1$, so that the graph of $S$ is a single linear component 
     \[
  \begin{tikzpicture}[xscale=1.5,yscale=-0.5]
    \node (A0_0) at (0, 0) {$\bullet$};
    \node (A0_1) at (1, 0) {$\bullet$};
    \node (A0_2) at (2, 0) {$\dots$};
    \node (A0_3) at (3, 0) {$\bullet$};
    \path (A0_0) edge [-] node [auto] {$\scriptstyle{}$} (A0_1);
    \path (A0_2) edge [-] node [auto] {$\scriptstyle{}$} (A0_3);
    \path (A0_1) edge [-] node [auto] {$\scriptstyle{}$} (A0_2);
  \end{tikzpicture}
  \]
In this case we will prove that the extra vector $v$ is linked to a internal vector
of $S$ and that the graph of $S\cup\{v\}$, which is of the form
   \[
  \begin{tikzpicture}[xscale=1.5,yscale=-1.0]
    \node (A0_0) at (0, 0) {$\bullet$};
    \node (A0_1) at (1, 0) {$\dots$};
    \node (A0_2) at (2, 0) {$\bullet$};
    \node (A0_3) at (3, 0) {$\dots$};
    \node (A0_4) at (4, 0) {$\bullet$};
    \node (A1_2) at (2, 1) {$\circ$};
    \path (A0_0) edge [-] node [auto] {$\scriptstyle{}$} (A0_1);
    \path (A0_2) edge [-] node [auto] {$\scriptstyle{}$} (A0_3);
    \path (A0_3) edge [-] node [auto] {$\scriptstyle{}$} (A0_4);
    \path (A1_2) edge [-,dashed] node [auto] {$\scriptstyle{}$} (A0_2);
    \path (A0_1) edge [-] node [auto] {$\scriptstyle{}$} (A0_2);
  \end{tikzpicture}
  \]

is a building block of the second type.
Here the extra vector $v$ has been depicted with a white dot and the edges coming out of it 
are dashed.
\item $c(S)=2$. In this case the graph of $S$ consists of two linear components. There are
three possible graphs for $S\cup\{v\}$ according to wether $v$ is linked to a pair of final vectors,
to a final vector and an internal one or to two internal vectors. We will prove that:
\begin{itemize}
 \item in the first case $b(S)=0$ and $\Gamma_{S\cup\{v\}}$ is a building block of the first type
 \item in the second case $b(S)=1$ and $\Gamma_{S\cup\{v\}}$ is a building block of the second type
 \item in the third case $b(S)=2$ and $\Gamma_{S\cup\{v\}}$ is a building block of the fourth type
\end{itemize}
the graphs corresponding to these three possibilities are the following
  \[
  \begin{tikzpicture}[xscale=1.5,yscale=-1.0]
    \node (A0_0) at (0, 0) {$\bullet$};
    \node (A0_1) at (1, 0) {$\dots$};
    \node (A0_2) at (2, 0) {$\bullet$};
    \node (A0_3) at (3, 0) {$\circ$};
    \node (A0_4) at (4, 0) {$\bullet$};
    \node (A0_5) at (5, 0) {$\dots$};
    \node (A0_6) at (6, 0) {$\bullet$};
    \path (A0_5) edge [-] node [auto] {$\scriptstyle{}$} (A0_6);
    \path (A0_0) edge [-] node [auto] {$\scriptstyle{}$} (A0_1);
    \path (A0_1) edge [-] node [auto] {$\scriptstyle{}$} (A0_2);
    \path (A0_4) edge [-] node [auto] {$\scriptstyle{}$} (A0_5);
    \path (A0_3) edge [-,dashed] node [auto] {$\scriptstyle{}$} (A0_4);
    \path (A0_2) edge [-,dashed] node [auto] {$\scriptstyle{}$} (A0_3);
  \end{tikzpicture}
  \]
    \[
  \begin{tikzpicture}[xscale=1.5,yscale=-1.0]
    \node (A0_0) at (0, 0) {$\bullet$};
    \node (A0_1) at (1, 0) {$\dots$};
    \node (A0_2) at (2, 0) {$\bullet$};
    \node (A1_3) at (3, 1) {$\bullet$};
    \node (A1_4) at (4, 1) {$\circ$};
    \node (A1_5) at (5, 1) {$\bullet$};
    \node (A1_6) at (6, 1) {$\dots$};
    \node (A1_7) at (7, 1) {$\bullet$};
    \node (A2_0) at (0, 2) {$\bullet$};
    \node (A2_1) at (1, 2) {$\dots$};
    \node (A2_2) at (2, 2) {$\bullet$};
    \path (A1_6) edge [-] node [auto] {$\scriptstyle{}$} (A1_7);
    \path (A0_0) edge [-] node [auto] {$\scriptstyle{}$} (A0_1);
    \path (A0_1) edge [-] node [auto] {$\scriptstyle{}$} (A0_2);
    \path (A2_1) edge [-] node [auto] {$\scriptstyle{}$} (A2_2);
    \path (A1_4) edge [-,dashed] node [auto] {$\scriptstyle{}$} (A1_5);
    \path (A2_2) edge [-] node [auto] {$\scriptstyle{}$} (A1_3);
    \path (A1_5) edge [-] node [auto] {$\scriptstyle{}$} (A1_6);
    \path (A2_0) edge [-] node [auto] {$\scriptstyle{}$} (A2_1);
    \path (A0_2) edge [-] node [auto] {$\scriptstyle{}$} (A1_3);
    \path (A1_3) edge [-,dashed] node [auto] {$\scriptstyle{}$} (A1_4);
  \end{tikzpicture}
  \]
 
    \[
  \begin{tikzpicture}[xscale=1.3,yscale=-1.0]
    \node (A0_0) at (0, 0) {$\bullet$};
    \node (A0_1) at (1, 0) {$\dots$};
    \node (A0_2) at (2, 0) {$\bullet$};
    \node (A0_6) at (6, 0) {$\bullet$};
    \node (A0_7) at (7, 0) {$\dots$};
    \node (A0_8) at (8, 0) {$\bullet$};
    \node (A1_3) at (3, 1) {$\bullet$};
    \node (A1_4) at (4, 1) {$\circ$};
    \node (A1_5) at (5, 1) {$\bullet$};
    \node (A2_0) at (0, 2) {$\bullet$};
    \node (A2_1) at (1, 2) {$\dots$};
    \node (A2_2) at (2, 2) {$\bullet$};
    \node (A2_6) at (6, 2) {$\bullet$};
    \node (A2_7) at (7, 2) {$\dots$};
    \node (A2_8) at (8, 2) {$\bullet$};
    \path (A2_0) edge [-] node [auto] {$\scriptstyle{}$} (A2_1);
    \path (A0_2) edge [-] node [auto] {$\scriptstyle{}$} (A1_3);
    \path (A2_1) edge [-] node [auto] {$\scriptstyle{}$} (A2_2);
    \path (A1_4) edge [-,dashed] node [auto] {$\scriptstyle{}$} (A1_5);
    \path (A0_6) edge [-] node [auto] {$\scriptstyle{}$} (A0_7);
    \path (A2_2) edge [-] node [auto] {$\scriptstyle{}$} (A1_3);
    \path (A2_7) edge [-] node [auto] {$\scriptstyle{}$} (A2_8);
    \path (A1_3) edge [-,dashed] node [auto] {$\scriptstyle{}$} (A1_4);
    \path (A0_0) edge [-] node [auto] {$\scriptstyle{}$} (A0_1);
    \path (A0_7) edge [-] node [auto] {$\scriptstyle{}$} (A0_8);
    \path (A2_6) edge [-] node [auto] {$\scriptstyle{}$} (A2_7);
    \path (A1_5) edge [-] node [auto] {$\scriptstyle{}$} (A2_6);
    \path (A0_1) edge [-] node [auto] {$\scriptstyle{}$} (A0_2);
    \path (A1_5) edge [-] node [auto] {$\scriptstyle{}$} (A0_6);
  \end{tikzpicture}
  \]
 The analysis required by the above four cases may be sketched as follows. We may think of
$S$ as a square matrix whose columns are its elements. The condition on the extra vector $v$ may be
translated into a matrix equation, namely
$$
^tSv=e_i \ \ \textrm{for some} \ \ i\leq N
$$
for the first case, and
$$
^tSv=e_i+e_j \ \ \textrm{for some} \ \ i,j\leq N
$$
for the other cases. In each case there is an obvious solution to the above equations,
which gives rise to a subset whose graph is a buillding block. Using this language,
the content of Theorem \ref{irreducible} amounts to saying that the only integral solutions to the above
systems of equations are the obvious ones. This fact will be proved by assuming that there is
a nonobvious solution and then finding a contradiction with the constraints provided by Lemma \ref{dis}.
\end{enumerate}

\end{section}
 \section{Irreducible subsets}\label{3irr}
 In this section collect all the results we need to prove Theorem \ref{irreducible}. As explained at the end of the previous
 section, we will need to examine several cases.
 
 \begin{Pro}\label{c1}
 Let $S=\{v_1,\dots,v_n\}\subset\mathbb{Z}^n$ be a standard subset.
 Suppose there exists $v\in\mathbb{Z}^n$ which is linked once to a
 vector, say $v_k$, of $S$ and is orthogonal to every other vector of $S$. Then, 
 \begin{itemize}
 \item $v_k$ is internal and $v_k\cdot v_k=-2$
  \item $v\cdot v=-1$
  \item $\Gamma_{S\cup\{v\}}$ is a building block of second type
  \item $I(S)=-3$
 \end{itemize} 
\end{Pro}
\begin{proof}
Assume by contradiction that $v_k$ is final. Then, $\Gamma_{S\cup\{v\}}$ is a linear plumbing graph consisting of $n+1$ linearly dependent vectors
and, by Proposition \ref{cfdet} we have $cf(\Gamma_{S\cup\{v\}})=0$, which means that
$$
cf(\Gamma_{S\cup\{v\}})=v\cdot v-\frac{1}{cf(\Gamma_S)}=0.
$$
This is impossible because $cf(\Gamma_S)<-1$. It follows that $v_k$ is internal.
 By Proposition \ref{cfdet} the continued fraction associated to
 $S\cup\{v\}$ must vanish and it can be written as
 $$
 v_k\cdot v_k - \frac{q_1}{p_1} - \frac{q_2}{p_2} - \frac{1}{v\cdot v}=0,
 $$
 where the $\frac{p_i}{q_i}$'s are the continued fractions associated to the linear graphs obtained from
 $S$ by deleting $v_k$. Since $0<-\frac{q_i}{p_i}<1$ it follows that $v_k\cdot v_k\in\{-1,-2\}$.
 The case $v_k\cdot v_k=-1$ cannot occur because $S$ is standard, therefore $v_k\cdot v_k=-2$.
 
 By Lemma \ref{dis} we have $v\cdot v> v_k\cdot v_k=-2$, therefore $v\cdot v=-1$. We may write 
 $v=e_s$ for some $s\in\{1,\dots,n\}$. Since $v$ is orthogonal to every vector of $S\setminus\{v_k\}$, 
 we can perform the transformation
 $$
 S\cup\{v\}\mapsto S':=\pi_s(S).
 $$
At the level of graphs this is just a blowdown move. Since $n=|S'|\subset\mathbb{Z}^{n-1}$ 
we see that $det\Gamma_{S'}=0$. By Proposition \ref{cfdet} we have $cf(\Gamma_{S'})=0$, which means
that the condition 3 of Proposition \ref{compeq} holds, where $\Gamma_1$ and $\Gamma_2$ are the connected components of 
$S'\setminus\{\pi_s(v)\}$. This shows that $\Gamma_{S'}$ is a building
block of the first type and $\Gamma_{S\cup\{v\}}$ is a building block of the second type.

Since $S\setminus\{v_k\}$ consists of two complementary legs, we have $I(S\setminus\{v_k\})=-2$
and so $I(S)=-3$.

\end{proof}

 In the next proposition we make explicit a characterization of certain good subsets which is
contained in \cite{Lisca:2} (see the proof of the main theorem). 

 \begin{Pro}\label{char}
 Let $S$ be a good subset such that $I(S)+c(S)\leq0$, $I(S)+b(S)<0$.
 Then $c(S)\leq2$. Assume $c(S)=2$.
 \begin{enumerate}
  \item if $b(S)=0$ then $\Gamma_S$ consists of two complementary legs
  \item if $b(S)=1$ then one of the following holds
  \begin{itemize}
   \item $\Gamma_S$ is obtained from the following graph 
     \[
  \begin{tikzpicture}[xscale=1.5,yscale=-0.5]
    \node (A0_0) at (0, 0) {$-2$};
    \node (A0_1) at (1, 0) {$ -(n+1)$};
    \node (A0_2) at (2, 0) {$-2$};
    \node (A0_3) at (3, 0) {$-2$};
    \node (A0_4) at (4, 0) {$-2$};
    \node (A0_6) at (6, 0) {$-2$};
    \node (A1_0) at (0, 1) {$\bullet$};
    \node (A1_1) at (1, 1) {$\bullet$};
    \node (A1_2) at (2, 1) {$\bullet$};
    \node (A1_3) at (3, 1) {$\bullet$};
    \node (A1_4) at (4, 1) {$\bullet$};
    \node (A1_5) at (5, 1) {$\dots$};
    \node (A1_6) at (6, 1) {$\bullet$};
    \path (A1_0) edge [-] node [auto] {$\scriptstyle{}$} (A1_1);
    \path (A1_4) edge [-] node [auto] {$\scriptstyle{}$} (A1_5);
    \path (A1_3) edge [-] node [auto] {$\scriptstyle{}$} (A1_4);
    \path (A1_1) edge [-] node [auto] {$\scriptstyle{}$} (A1_2);
    \path (A1_5) edge [-] node [auto] {$\scriptstyle{}$} (A1_6);
  \end{tikzpicture}
  \]
  
(the $-2$-chain has length $n-1$ and $n\geq 2$)
    via a finite number of $-2$-final expansions performed on the leftmost component.
    
   \item $\Gamma_S=\Gamma_1\sqcup\Gamma_2$, where $\Gamma_1$ is obtained
   from the graph
      \[
  \begin{tikzpicture}[xscale=1.5,yscale=-0.5]
    \node (A0_0) at (0, 0) {$-2$};
    \node (A0_1) at (1, 0) {$ -a$};
    \node (A0_2) at (2, 0) {$-2$};
    \node (A1_0) at (0, 1) {$\bullet$};
    \node (A1_1) at (1, 1) {$\bullet$};
    \node (A1_2) at (2, 1) {$\bullet$};
    \node (A1_3) at (3, 1) {$;$};
    \node (A1_4) at (4, 1) {$a\geq 3$};
    \path (A1_0) edge [-] node [auto] {$\scriptstyle{}$} (A1_1);
    \path (A1_1) edge [-] node [auto] {$\scriptstyle{}$} (A1_2);
  \end{tikzpicture}
  \]

   via a finite number of $-2$-final expansions and $\Gamma_2$ is dual to
   a graph obtained from the one above via a finite number of $-2$-final expansions.  
  \end{itemize}

  \item If $b(S)=2$ then $\Gamma_S=\Gamma_1\sqcup\Gamma_2$ where 
  each $\Gamma_i$ is obtained from 
    \[
  \begin{tikzpicture}[xscale=1.5,yscale=-0.5]
    \node (A0_0) at (0, 0) {$-2$};
    \node (A0_1) at (1, 0) {$ -3$};
    \node (A0_2) at (2, 0) {$-2$};
    \node (A1_0) at (0, 1) {$\bullet$};
    \node (A1_1) at (1, 1) {$\bullet$};
    \node (A1_2) at (2, 1) {$\bullet$};
    \path (A1_0) edge [-] node [auto] {$\scriptstyle{}$} (A1_1);
    \path (A1_1) edge [-] node [auto] {$\scriptstyle{}$} (A1_2);
  \end{tikzpicture}
  \]

  via a finite sequence of $-2$-final expansions.
 \end{enumerate}

\end{Pro}
 \begin{re}
  It maybe useful to explain how the graph of a linear subset changes via $-2$-final expansions.
  Suppose that $S$ is a linear subset and that, for some index $i$, $e_i$ hits only two final vectors $v_1$ and $v_2$.
  If $v_1$ and $v_2$ belong to the same connected component of $\Gamma_S$ then, a $-2$-final expansion changes the graph as follows
    \[
  \begin{tikzpicture}[xscale=1.5,yscale=-0.5]
    \node (A0_0) at (0, 0) {$v_1$};
    \node (A0_3) at (3, 0) {$v_2$};
    \node (A0_5) at (5, 0) {$e_i+e_j$};
    \node (A0_6) at (6, 0) {$v_1$};
    \node (A0_9) at (9, 0) {$v_2-e_j$};
    \node (A1_0) at (0, 1) {$\bullet$};
    \node (A1_1) at (1, 1) {$\bullet$};
    \node (A1_2) at (2, 1) {$\dots$};
    \node (A1_3) at (3, 1) {$\bullet$};
    \node (A1_4) at (4, 1) {$\longrightarrow$};
    \node (A1_5) at (5, 1) {$\bullet$};
    \node (A1_6) at (6, 1) {$\bullet$};
    \node (A1_7) at (7, 1) {$\bullet$};
    \node (A1_8) at (8, 1) {$\dots$};
    \node (A1_9) at (9, 1) {$\bullet$};
    \path (A1_6) edge [-] node [auto] {$\scriptstyle{}$} (A1_7);
    \path (A1_0) edge [-] node [auto] {$\scriptstyle{}$} (A1_1);
    \path (A1_1) edge [-] node [auto] {$\scriptstyle{}$} (A1_2);
    \path (A1_5) edge [-] node [auto] {$\scriptstyle{}$} (A1_6);
    \path (A1_2) edge [-] node [auto] {$\scriptstyle{}$} (A1_3);
    \path (A1_8) edge [-] node [auto] {$\scriptstyle{}$} (A1_9);
    \path (A1_7) edge [-] node [auto] {$\scriptstyle{}$} (A1_8);
  \end{tikzpicture}
  \]
where we are assuming that $v_1=-e_j+...$ and $v_2=e_j+...$. 
An analogous operation can be performed when $v_1$ and $v_2$ belong to different connected components.
 \end{re}

 \begin{Pro}\label{nobad}
 Let $S=S_1\cup S_2$ be a good subset with no bad components such that $I(S)<0$ and $c(S)=2$. Let $v$ be an element of, say, $S_1$.\\  
 \begin{enumerate}
  \item if $v$ is internal and $v\cdot v\geq -3$ there exists a vector $v'\in S_2$ such that $E(v,v')=2$;
  \item if $v$ is internal and $v\cdot v=-k<-3$ there exists a $-2$-chain in $S_2$ of the form
  $$
  (\dots,e_1-e_2,e_2-e_3,\dots,e_{k-3}-e_{k-2},\dots).
  $$
  and $|e_i\cdot v|=1$ for each $i\leq k-2$;
  \item if $v$ is final and $v\cdot v=-k<-2$ there exists a $-2$-chain in $S_2$ of the form
  $$
  (e_1-e_2,e_2-e_3,\dots,e_{k-2}-e_{k-1},\dots).
  $$
  and $|e_i\cdot v|=1$ for each $1\leq i\leq k-2$;
 \end{enumerate}
 \end{Pro}
 \begin{proof}
  It is shown in \cite{Lisca:2} (in the proof of theorem 1.1) that a subset $S$ satisfying our
  hypothesis is obtained via a sequence of $-2$-final expansions as described in Lemma 4.7 in \cite{Lisca:2} from
  a subset of the form $\{e_1-e_2,e_1+e_2\}$.
  In particular, $|e_i\cdot v|\in\{0,1\}$ for each $i$ every $v\in S$. This means that we can always write
  $$
  v=\sum_{i=1}^{|v\cdot v|}\varepsilon_ie_i \ \ \ \textrm{where} \ \ \ \varepsilon_i\in\{\pm 1\}.
  $$
  If $v\cdot v=-2$ write $v=e_1+e_2$. Again by Lemma 4.7 in \cite{Lisca:2} every basis vector
  that hits an internal vector hits exactly three vectors of $S$. It follows that $e_1$ hits
  two more vectors, say $v'$ and $v''$. Suppose that $e_2$ does not hit any of these vectors. 
  Then we must have $v'\cdot v=v''\cdot v=1$. Now $e_2$ must hit some vector, say $v'''$. Since
  $e_1$ does not hit $v'''$, we would have $v'''\cdot v=1$. But then $v$ would be adiacent to three
  vectors, which is impossible. The same argument works if $v\cdot v=-3$, we omit the details.
  
  If $v\cdot v=k\leq-4$ write $v=\sum_{i=1}^{k}e_i$. It is clear from the proof of the main theorem
  in \cite{Lisca:2} that the subset $S$ is obtained by $-2$-final expansions  from a subset $S'$
  whose associated string is
  $$
  (2,2)\cup (3).
  $$
  Then the assertion is easily proved by induction on the number of expansions needed to obtain 
  $S$ from $S'$, we omit the details.
  
  The third assertion is proved similarly. If $v\cdot v=-k<-2$ then $S$ originates from a subset $S'$ via $k-2$ $-2$-final expansions.
  Similarly $v$ originates from a final vector $v'\in S'$, with $v'\cdot v'=-2$.
  Each $-2$-final expansion creates a new $-2$-final vector in $S_2$ linked to the one resulting from the previous expansion.
 \end{proof}

 \subsection{First case: $b(S)=0$} \label{b=0}
 In this subsection we examine the subset in Proposition \ref{char} with no bad components.
  We will need the following lemma.

 \begin{lem}\label{ijbis}
  Let $S$ be a good subset such that $c(S)=2$ and $b(S)=0$. 
  Let $v_i,v_j$ be two vectors in $S$. 
  We have
  \begin{itemize}
   \item if $v_i\cdot v_j=1$ then $E(v_i,v_j)=1$
   \item if $v_i\cdot v_j=0$ then $E(v_i,v_j)\in\{0,2\}$
  \end{itemize}
 \end{lem}
 \begin{proof}
 The Lemma clearly holds for the subset $S_3$.
 Let $S_n$ be a subset obtained from $S_3$ via a sequence of $-2$-final expansions
 $$
 S_3\nearrow S_4\nearrow\dots\nearrow S_n
 $$
 Suppose the lemma holds for $S_{n-1}$. The conclusion follows easily from the fact that
 the new vector which has been introduced has square $-2$. We omit the details.
 \end{proof}
 
\begin{Pro}\label{b0}
 Let $S\subset\mathbb{Z}^{N}$ be a good subset such that $c(S)=2$ and $b(S)=0$. Suppose that there exists
 a vector $v\in\mathbb{Z}^{N}$ that is linked once to a vector of each connected component of $S$ and is
 orthogonal to all the remaining vectors of $S$. Then
 \begin{itemize}
  \item $v$ is linked to a pair of final vectors
  \item $v\cdot v=-1$
  \item the graph of $S\cup\{v\}$ is a building block of the first type
  \item $I(S)=-2$
 \end{itemize}
\end{Pro}
\begin{proof}
 Write $S=S_1\cup S_2$ and $w_1,w_2$ for the two vectors linked once with $v$. First note that if 
 both $w_1$ and $w_2$ are final vectors then the graph associated to $S\cup\{v\}$ is linear and since
 $det\Gamma_{S\cup\{v\}}=0$ the corresponding plumbed 3-manifold is diffeomorphic to $S^1\times S^2$. This means
 that $\Gamma_{S\cup\{v\}}$ cannot be in normal form which is only possible if $v\cdot v=-1$. By Proposition
 \ref{compeq} the graph $\Gamma_{S\cup\{v\}}$ is building block of the first type. Also by Proposition
 \ref{compeq} the two components of $S$ are complementary and so $I(S)=-2$.
 Therefore it is enough to show that both $w_1$ and $w_2$ are final.
 
 Assume by contradiction that $w_1$ is an internal vector. Then we have $v\cdot v<-1$. To see this 
 note that if $v\cdot v=-1$ then, by lemma 4.7 in \cite{Lisca:2}, the vector $v$ 
 can only hit final vectors. 
 By Lemma \ref{basic} at least one vector among $w_1$ and $w_2$ has $-2$ square.
 
 We have two possibilities.
 
 \textbf{First case:} The vector $w_2$ is final.\\
 The graph $\Gamma_{S\cup\{v\}}$ has the following form
   \[
  \begin{tikzpicture}[xscale=1.5,yscale=-0.5]
    \node (A0_5) at (5, 0) {$\bullet$};
    \node (A0_6) at (6, 0) {$\dots$};
    \node (A0_7) at (7, 0) {$\bullet$};
    \node (A1_0) at (0, 1) {$\bullet$};
    \node (A1_1) at (1, 1) {$\dots$};
    \node (A1_2) at (2, 1) {$\bullet$};
    \node (A1_3) at (3, 1) {$\bullet$};
    \node (A1_4) at (4, 1) {$\bullet$};
    \node (A2_5) at (5, 2) {$\bullet$};
    \node (A2_6) at (6, 2) {$\dots$};
    \node (A2_7) at (7, 2) {$\bullet$};
    \path (A2_6) edge [-] node [auto] {$\scriptstyle{}$} (A2_7);
    \path (A0_6) edge [-] node [auto] {$\scriptstyle{}$} (A0_7);
    \path (A1_0) edge [-] node [auto] {$\scriptstyle{}$} (A1_1);
    \path (A2_5) edge [-] node [auto] {$\scriptstyle{}$} (A2_6);
    \path (A1_1) edge [-] node [auto] {$\scriptstyle{}$} (A1_2);
    \path (A1_4) edge [-] node [auto] {$\scriptstyle{}$} (A2_5);
    \path (A1_4) edge [-] node [auto] {$\scriptstyle{}$} (A0_5);
    \path (A1_2) edge [-] node [auto] {$\scriptstyle{}$} (A1_3);
    \path (A0_5) edge [-] node [auto] {$\scriptstyle{}$} (A0_6);
    \path (A1_3) edge [-] node [auto] {$\scriptstyle{}$} (A1_4);
  \end{tikzpicture}
  \]

 It is a star-shaped
 plumbing graph in normal form with three legs. Since $det\Gamma_{S\cup\{v\}}=0$ 
 the weight of the central vertex, which is $w_1$, can only be $-1$ or $-2$, since $S$ is a good
 subset we have $w_1\cdot w_1=-2$. We may write $w_1=e_1+e_2$. 
 Recall that by Lemma \ref{dis} we must have
 \begin{equation}\label{disnobad}
  \|v\|^2< 2+\|w_2\|^2.
 \end{equation}
 Moreover we claim that
 \begin{equation}\label{Ew}
  E(w_1,w_2)=0.
 \end{equation}
To see this note that since $w_1\cdot w_2=0$ and $w_1\cdot w_1=-2$ we have $E(w_1,w_2)\in\{0,2\}$. 
If both $e_1$ and $e_2$ hit $w_2$ then, by Lemma 4.7(3) in \cite{Lisca:2}, at least one of them
hits exactly two vectors in $S$. But then, again by Lemma 4.7(2) in \cite{Lisca:2}, these two vectors
are not internal. This contradicts the fact that $w_1$ is internal.

 Now we proceed by distinguishing several cases according to the weight of $w_2$.\\
 \emph{First subcase}: $w_2\cdot w_2=-2$.\\
 By \eqref{Ew} we may write
 $$
 w_1=e_1+e_2 \ \ \ ; \ \ \ w_2=e_3+e_4.
 $$
 Note that \eqref{disnobad} tells us that $\|v\|^2<4$, in particular $|v\cdot e_i|\leq 1$ for each $e_i$.
 Therefore, since $1=v\cdot w_1=v\cdot e_1+v\cdot e_2$, either $v\cdot e_1=0$ or $v\cdot e_2=0$. Similarly
 either $v\cdot e_3=0$ or $v\cdot e_4=0$. Without loss of generality we may write $v=-e_1-e_3+v'$, where $v'\cdot e_i=0$ for $i\leq 4$.
 By \eqref{disnobad}, we have $\|v'\|^2\leq 1$. Since $w_1$ is internal, by Lemma 4.7 in \cite{Lisca:2} we know that $e_1$ hits exactly 
 three vectors in $S$, say $w_1$, $u_1$ and $u_2$. The condition $v\cdot u_1=v\cdot u_2=0$ shows that $v'\neq 0$, say $v'=e_5$. 
 We obtain the expression $v=-e_1-e_3+e_5$.
 We have $v\cdot u_i=-e_1\cdot u_i+e_5\cdot u_i=0$ for $i=1,2$. Therefore we may write
 $u_i=\varepsilon_i(e_1+e_5)+u_i'$ with $u_i'\cdot e_1=u_i'\cdot e_5=0$ and $\varepsilon_i=\pm 1$ for $i=1,2$. 
 This fact together with $|u_1\cdot u_2|\leq 1$ implies that $E(u_1,u_2)>2$ which contradits Lemma \ref{ijbis}.\\
 \emph{Second subcase}: $w_2\cdot w_2=-3$.\\
 By \eqref{Ew} we may write
 $$
 w_1=e_1+e_2 \ \ \ ; \ \ \ w_2=e_3+e_4+e_5.
 $$
 By lemma 4.7 in \cite{Lisca:2}, there exists a final vector $w_3$ which, without loss of generality, we can write
 as $w_3=e_3-e_4$. Now let us write
 $$
 v=v'+\alpha_1e_1+\alpha_2e_2+\alpha_3e_3+\alpha_4e_4+\alpha_5e_5
 $$
 where $v'\cdot e_i=0$ for each $i\leq 5$. Since at least two $\alpha_i$'s are non zero it follows by \eqref{disnobad} that $|\alpha_i|\leq 1$
 for each $i\leq 5$ and that $\sum_{i=1}^5 |\alpha_i|<5$. In particular at least one coefficient is zero.
 The conditions $v\cdot w_1=v\cdot w_2=1$ and $v\cdot w_3=0$ quickly imply the following 
 \begin{itemize}
  \item $(\alpha_1,\alpha_2)\in\{(-1,0),(0,-1)\}$;
  \item $(\alpha_3,\alpha_4)\in\{(0,0),(1,1)\}$;
  \item $\alpha_5=-1$;
 \end{itemize}
 If $(\alpha_3,\alpha_4)= (1,1)$ then $\|v\|^2=4+\|v'\|^2$ and therefore $v'=0$. We can write
 $$
 v=\alpha_2e_2+\alpha_3e_3+\alpha_4e_4+\alpha_5e_5
 $$
 
 Let $w_4$ be the vector of $S_1$ such that $w_3\cdot w_4=1$. We may write this vector as 
 $w_4=w_4'+e_4$, and since $w_4\cdot w_2=0$ we may write $w_4=w_4''+e_4-e_5$. Clearly $e_i\cdot w_4''=0$ for $i\leq 5$.
 But then, since $v\cdot w_4=0$, we would have $\alpha_3=\alpha_4=\alpha_5$ which does not match with 
 the previous conditions we obtained for these coefficients.
 
 Therefore we may assume that $(\alpha_3,\alpha_4)=(0,0)$.
 In this situation we may perform a $-2$-final contraction on $S$ that has the effect of deleting the vector $w_3$ and decreasing
 the norm of $w_2$ by one. The extra vector $v$ is not affected by this operation and all the hypothesis that we need remain valid.
 In this situation $v$ is linked to a final vector whose weight is $-2$ and therefore we may repeat the argument given in the first subcase.\\
 \emph{Third subcase}: $w_2\cdot w_2< -3$. \\
 We may write
$w_2=\sum_{i=1}^{k}e_i$, with $k\geq 3$. By Proposition \ref{nobad} there is a $-2$-chain of the form
$$
  (e_1-e_2,e_2-e_3,\dots,e_{k-2}-e_{k-1},\dots).
  $$
  By \eqref{Ew} we know that $w_1$ does not belong to this chain.
  Therefore $v$ must be orthogonal to every vector in this chain.
  It follows that either $v$ hits all of the vectors in the set $\{e_1,\dots,e_{k-1}\}$ or it does not hit any 
  of them. 
  
  If $v$ hits all of the vectors in the set $\{e_1,\dots,e_{k-1}\}$ we can write, without loss of generality,
  $$
  v=v'+\alpha\sum_{i=1}^{k-1}e_i
  $$
  where $v'\cdot e_i=0$ for $i\leq k-1$ and $\alpha\in\mathbb{Z}\setminus\{0\}$. But then the condition $v\cdot w_2=1$ implies
  $v\cdot e_k=\alpha(k-1)+1$ and therefore
  $$
  \|v\|^2\geq \alpha^2(k-1)+(\alpha(k-1)+1)^2\geq k-1+k^2\geq k+2
  $$
  and this contradicts \eqref{disnobad}.
  
  If $v$ does not hit any of the vectors in the set $\{e_1,\dots,e_{k-2}\}$ we can perform a series
  of $-2$-final contractions that will eliminate these vectors. These contractions do not alter
  the vector $v$. Let $w'_2$ be the image of $w_2$ after these contractions are performed. 
  Since $w'_2\cdot w'_2=-2$ we can apply the argument given in the first subcase.

\textbf{Second case:} 
The vector $w_2$ is internal. The graph $\Gamma_{S\cup\{v\}}$ has the following form
  \[
  \begin{tikzpicture}[xscale=1.5,yscale=-0.5]
    \node (A0_0) at (0, 0) {$\bullet$};
    \node (A0_1) at (1, 0) {$\dots$};
    \node (A0_2) at (2, 0) {$\bullet$};
    \node (A0_6) at (6, 0) {$\bullet$};
    \node (A0_7) at (7, 0) {$\dots$};
    \node (A0_8) at (8, 0) {$\bullet$};
    \node (A1_3) at (3, 1) {$\bullet$};
    \node (A1_4) at (4, 1) {$\bullet$};
    \node (A1_5) at (5, 1) {$\bullet$};
    \node (A2_0) at (0, 2) {$\bullet$};
    \node (A2_1) at (1, 2) {$\dots$};
    \node (A2_2) at (2, 2) {$\bullet$};
    \node (A2_6) at (6, 2) {$\bullet$};
    \node (A2_7) at (7, 2) {$\dots$};
    \node (A2_8) at (8, 2) {$\bullet$};
    \path (A1_4) edge [-] node [auto] {$\scriptstyle{}$} (A1_5);
    \path (A2_2) edge [-] node [auto] {$\scriptstyle{}$} (A1_3);
    \path (A2_6) edge [-] node [auto] {$\scriptstyle{}$} (A2_7);
    \path (A0_1) edge [-] node [auto] {$\scriptstyle{}$} (A0_2);
    \path (A2_1) edge [-] node [auto] {$\scriptstyle{}$} (A2_2);
    \path (A0_7) edge [-] node [auto] {$\scriptstyle{}$} (A0_8);
    \path (A0_0) edge [-] node [auto] {$\scriptstyle{}$} (A0_1);
    \path (A1_5) edge [-] node [auto] {$\scriptstyle{}$} (A2_6);
    \path (A2_0) edge [-] node [auto] {$\scriptstyle{}$} (A2_1);
    \path (A0_2) edge [-] node [auto] {$\scriptstyle{}$} (A1_3);
    \path (A2_7) edge [-] node [auto] {$\scriptstyle{}$} (A2_8);
    \path (A0_6) edge [-] node [auto] {$\scriptstyle{}$} (A0_7);
    \path (A1_5) edge [-] node [auto] {$\scriptstyle{}$} (A0_6);
    \path (A1_3) edge [-] node [auto] {$\scriptstyle{}$} (A1_4);
  \end{tikzpicture}
  \]
Recall that we have shown that $v\cdot v<-1$. By Lemma \ref{basic}, we may assume, as in the first case, that 
one of the vectors $w_1$ and $w_2$, say $w_1$, has $-2$-square. As a consequence Equation \eqref{disnobad} holds.
Note that if $w_2\cdot w_2=-2$ the argument given in the first case works as well in this situation. Therefore 
we may assume that $w_2\cdot w_2\leq -3$. 

Let $e_s$ be a base vector that hits two final vectors of $S$. It is easy to see that if $e_s\cdot v=0$
then the $-2$-final contraction $S\searrow S'$ associated to $e_s$ does not affect the vector $v$. In this situation
the subset $S'$ satisfies all the hypotheses in the statement and the conlcusions hold for $S'$ if and only if they hold for $S$.
This process may be iterated, via a sequence of $-2$-final contractions $S\searrow\dots\searrow \overline{S}$,
until one of the following hold:
\begin{enumerate}
 \item the image in $\overline{S}$ of one vector among $w_1$ and $w_2$ is a final vector;
 \item no more contractions can be performed on $S$ without affecting the vector $v$.
\end{enumerate}
If the first condition holds we may apply the argument given in the first case. Assume the second condition holds.
The subset $\overline{S}$ has two $-2$-final vectors of the form $e_{j_1}-e_{j_2}$ and $e_{j_3}-e_{j_4}$. By our assumption
\begin{equation}\label{colpodiscena}
v\cdot e_{j_i}\neq 0 \ \ \ \textrm{for each} \ \ \ 1\leq i\leq 4. 
\end{equation}
 Now we distinguish two cases.
 
 \emph{First subcase}:$w_2\cdot w_2=-3$.\\	
 In this case Equation \eqref{colpodiscena} contradicts Equation \eqref{disnobad}.\\
 \emph{Second subcase}: $w_2\cdot w_2<-3$.\\
By Proposition \ref{nobad} there is a $-2$-chain of the form
$$
  (\dots,-e_1+\dots,e_1-e_2,e_2-e_3,\dots,e_{k-3}-e_{k-2},e_{k-2}+\dots,\dots).
  $$
and $w_2=\sum_{i=1}^ke_i$. Since $v$ is orthogonal to every vector in the $-2$-chain, either $v\cdot e_i\neq 0$
for each $i\leq k-2$ or $v\cdot e_i= 0$ for each $i\leq k-2$. In the first case we quickly obtain a contradiction with 
Equation \eqref{disnobad} (by taking into account \eqref{colpodiscena}). In the second case we may remove the whole 
$-2$-chain performing the transformation
$$
(\dots,-e_1+\dots,e_1-e_2,\dots,e_{k-3}-e_{k-2},e_{k-2}+\dots,\dots) \longrightarrow 
(\dots,-e_1+\dots,e_1+\dots,\dots).
$$
The image of the vector $w_2$ under this transformation is $\overline{w}_2=e_1+e_{k-1}+e_k$. Since 
$\overline{w}_2\cdot \overline{w}_2=-3$ we may argue as in the first subcase, and we are done.
\end{proof}
 \subsection{Second case: $b(S)=1$}
 In this section we deal with the subsets of Theorem \ref{irreducible} having a single
 bad component. As stated in Proposition \ref{char}, there are two different classes
 of such subsets. First we show that for one of these classes it is not possible to find an
 extra vector $v$ satisfying the hypothesis of Theorem \ref{irreducible}. Then we deal
 with the other class of subsets which will give rise to building block of the third type.

 \begin{Pro}\label{b1a}
  Let $S'$ be a good subset such that $b(S)=1$ and its graph $\Gamma_S$ is of the form
     \[
  \begin{tikzpicture}[xscale=1.5,yscale=-0.4]
    \node (A0_0) at (0, 0) {$-2$};
    \node (A0_1) at (1, 0) {$-a$};
    \node (A0_2) at (2, 0) {$-2$};
    \node (A0_4) at (4, 0) {$-3$};
    \node (A0_5) at (5, 0) {$-2$};
    \node (A0_7) at (7, 0) {$-2$};
    \node (A0_8) at (8, 0) {$-3$};
    \node (A1_0) at (0, 1) {$\bullet$};
    \node (A1_1) at (1, 1) {$\bullet$};
    \node (A1_2) at (2, 1) {$\bullet$};
    \node (A1_4) at (4, 1) {$\bullet$};
    \node (A1_5) at (5, 1) {$\bullet$};
    \node (A1_6) at (6, 1) {$\cdots$};
    \node (A1_7) at (7, 1) {$\bullet$};
    \node (A1_8) at (8, 1) {$\bullet$};
    \path (A1_6) edge [-] node [auto] {$\scriptstyle{}$} (A1_7);
    \path (A1_0) edge [-] node [auto] {$\scriptstyle{}$} (A1_1);
    \path (A1_4) edge [-] node [auto] {$\scriptstyle{}$} (A1_5);
    \path (A1_1) edge [-] node [auto] {$\scriptstyle{}$} (A1_2);
    \path (A1_5) edge [-] node [auto] {$\scriptstyle{}$} (A1_6);
    \path (A1_7) edge [-] node [auto] {$\scriptstyle{}$} (A1_8);
  \end{tikzpicture}
  \]
where $a\geq 3$ and the $-2$-chain has length $a-3$.
  Let $S$ be a good subset which is obtained via $-2$-final expansions from $S'$ as explained in Proposition \ref{char}. 
  Then, there exists no vector $v\in\mathbb{Z}^N$ 
  linked once to a vector of each connected component of $S$ and orthogonal to all the other vectors of $S$.
 \end{Pro}
 \begin{proof}
  Assume by contradiction that there exists
  $v\in\mathbb{Z}^N$ linked once to a vector of each connected component of $S$ and orthogonal to all the other vectors of $S$.
  We write $S=S_1\cup S_2$ where $S_1$ is obtained from the bad component
  of $S'$ via $-2$-final expansions and $S_2$ is obtained from the non bad component of $S'$ in a similar way.
  Note that the only vector of $S_1$ which is linked to a vector of $S_2$ is the central one. Call this vector $w$. More precisely,
  we may choose base vectors of $\mathbb{Z}^{N}$ $\{e_1,\dots,e_k,e_{k+1},\dots,e_N\}$ so that
  \begin{itemize}
   \item if $i\leq k+1$ we have $e_i\cdot u=0$ for each $u\in S_2$
   \item if $i\geq k+2$ we have $e_i\cdot u=0$ for each $u\in S_1\setminus\{w\}$
   \item $e_{k+1}\cdot w\neq 0$ and for some $j\geq k+2$ we have $e_{j}\cdot w\neq 0$.
  \end{itemize}
Note that $|S_1|=k+2$ and $|S_2|=N-k-2$. Now we proceed by distinguishing several cases.\\
\emph{First case}: $w\cdot v=0$. We can write $v=v_1+v_2$ so that $v_1$ is spanned
by $\{e_1,\dots,e_{k+1}\}$ and $v_2$ by $\{e_{k+2},\dots,e_N\}$. In particular, $v_1$ (resp. $v_2$)
is orthogonal to every element of $S_2$ (resp. $S_1$), and moreover both $v_1$ and $v_2$ are nonzero. The subset
$\widetilde{S}_1:=(S_1\setminus\{w\})\subset\mathbb{Z}^{k+1}$ consists of two complementary linear components, $T_1$ and $T_2$.
Since $w\cdot v=0$, the vector $v_1$ is linked once to a vector of, say, $T_1$ and is orthogonal to the other vectors of $\widetilde{S}_1$. The graph
$\Gamma_{\widetilde{S}_1\cup\{v_1\}}$ is given by the disjoint union $\Gamma_{T_1\cup\{v\}}\sqcup\Gamma_{T_2}$ where 
$\Gamma_{T_1\cup\{v\}}$ is starshaped with three legs and $\Gamma_{T_2}$ is linear. Since $det\Gamma_{\widetilde{S}_1\cup\{v_1\}}=0$, we have
$$
0=det\Gamma_{\widetilde{S}_1\cup\{v_1\}}=det(\Gamma_{T_1\cup\{v\}}\sqcup\Gamma_{T_2})=det\Gamma_{T_1\cup\{v\}}det\Gamma_{T_2}. 
$$
Since $det\Gamma_{T_2}\neq 0$, we must have $det\Gamma_{T_1\cup\{v\}}=0$. It follows that, as in the proof of Proposition
\ref{c1}, $v$ is linked once to a vector of $T_1$ with $-2$ square. This quickly leads to a contradiction with Lemma \ref{dis}.\\
\emph{Second case}: $w\cdot v=1$. We may write $v=v_1+v_2$ as in the first case. Since $v_1$ is orthogonal to the vectors
of $S_1\setminus \{w\}$ we must have $v_1=0$ (because $v_1$ is orthogonal $n$ linearly indipendent vectors in $\mathbb{Z}^n$). 
Consider the good subset 
$$
\widetilde{S}:=(S\setminus S_1)\cup\{\pi_{k+1}(w)\}.
$$ 
The vector $v=v_2$
is linked once to a vector of each connected component of $\widetilde{S}$ and is orthogonal to the other vectors of 
$\widetilde{S}$. The graph $\Gamma_{\widetilde{S}\cup\{v\}}$ is either starshaped with three legs (if $v$ is linked once to an internal
vector of $S_2$) or linear (if $v$ is linked once to a final vector of $S_2$). The latter possibility cannot occur. To see this
suppose that $\Gamma_{\widetilde{S}\cup\{v\}}$ is linear. Since $det\Gamma_{\widetilde{S}\cup\{v\}}=0$ we must have $v\cdot v=-1$. Moreover,
by Proposition \ref{compeq} the two component of $\widetilde{S}$ are complementary. Since one of these components consists of
a single vertex, the other one must be a $-2$-chain which is not the case. 
Therefore we may assume that the graph $\Gamma_{\widetilde{S}\cup\{v\}}$ is starshaped with three legs. The subset 
$\widetilde{S}$ is obtained via $-2$-final expansions (performed on the rightmost component) from a subset whose graph is
    \[
  \begin{tikzpicture}[xscale=1.5,yscale=-0.5]
    \node (A0_0) at (0, 0) {$-a+1$};
    \node (A0_2) at (2, 0) {$-3$};
    \node (A0_3) at (3, 0) {$-2$};
    \node (A0_5) at (5, 0) {$-2$};
    \node (A0_6) at (6, 0) {$-3$};
    \node (A1_0) at (0, 1) {$\bullet$};
    \node (A1_2) at (2, 1) {$\bullet$};
    \node (A1_3) at (3, 1) {$\bullet$};
    \node (A1_4) at (4, 1) {$\dots$};
    \node (A1_5) at (5, 1) {$\bullet$};
    \node (A1_6) at (6, 1) {$\bullet$};
    \path (A1_2) edge [-] node [auto] {$\scriptstyle{}$} (A1_3);
    \path (A1_4) edge [-] node [auto] {$\scriptstyle{}$} (A1_5);
    \path (A1_3) edge [-] node [auto] {$\scriptstyle{}$} (A1_4);
    \path (A1_5) edge [-] node [auto] {$\scriptstyle{}$} (A1_6);
  \end{tikzpicture}
  \]
where $a\geq 3$ and the $-2$-chain has legth $a-3$. Up to automorphisms of the integral lattice
$\mathbb{Z}^a$ this subset may be written as
\begin{equation}\label{tildetilde}
\widetilde{\widetilde{S}}:=\left\{\sum_{i=1}^{a-1} e_i\right\}\cup\{e_1-e_2+e_a,e_2-e_3,\dots,e_{a-2}-e_{a-1},e_{a-1}+e_a-e_1\} 
\end{equation}
Note that, as in the proof of Proposition \ref{c1}, the vector $v$ must be linked to a $-2$-vector, say $u$, of $\widetilde{S}\setminus\{\pi_{k+1}(w)\}$. 
We have two possibilities which we examine separately.\\
\emph{First subcase}: The vector $u$ is not affected by the series of $-2$-final contractions from $\widetilde{S}$ to 
$\widetilde{\widetilde{S}}$.
In this case the vector $u$ belongs to the $-2$-chain that appears in \eqref{tildetilde}. By Lemma \ref{dis} we must have
$v\cdot v<-a-2$. Write $u=e_k-e_{k+1}$ with $2\leq k\leq a-1$. It is easy to see that $v$ can be written as follows
$$
v=v'+\alpha\sum_{i=2}^{k}e_i+(1+\alpha)\sum_{i=k+1}^{a-1}e_i
$$
where $\alpha\in\mathbb{Z}\setminus\{0,-1\}$. This expression quickly leads to a contradiction with the inequality
$v\cdot v<-a-2$.\\
\emph{Second subcase}: The vector $u$ is the result of one of the $-2$-final expansions from $\widetilde{\widetilde{S}}$
to $\widetilde{S}$. Write $u=e_s+e_t$. We have either $e_s\cdot v\neq 0$ or $e_t\cdot v\neq 0$, and it is easy to see that $v$ mast hit
at least another base vector which is not in $\{e_1,\dots,e_{a-1}\}$. Moreover
since $w\cdot v=1$ the vector $v$ hits at least one vector among $\{e_1,\dots,e_{a-1}\}$. Since $v$ is orthogonal to all the vectors
in the $-2$-chain in \eqref{tildetilde} we see that $e_2\cdot v=\dots=e_{a-1}\cdot v$. If $e_2\cdot v\neq 0$ then we quickly obtain
a contradiction with Lemma \ref{dis} by computing $e_1\cdot v$. If $e_2\cdot v=0$ we may write
$v=v'-e_1+e_a$ where $e_j	\cdot v'=0$ for each $j\leq a$. In this situation we can change the subset $\widetilde{S}$ by removing
the coordinate vectors appearing in the $-2$-chain of $\widetilde{\widetilde{S}}$ and the vector $w$. 
We call this new subset $T$, it is obtained from the subset
$$
\{e_1-e_a,e_1+e_a\}
$$
via $-2$-final expansions.
The vector $v$ is not affected by this transformation. Note that $T$ is a good subset with two complementary
connected components and that $v$ is linked once to a vector of one conncted component and is
orthogonal to any other vector. The graph $\Gamma_{T\cup\{v\}}$ is the disjoint union of a three legged starshaped graph and
a linear one. Now we can argue as in the first case. Since $det\Gamma_{T\cup\{v\}}=0$ the vector $v$ must be linked to a $-2$-weighted
vertex which quickly leads to a contradiction with Lemma \ref{dis}.
\end{proof}

\begin{Pro}\label{b1b}
 Let $S=S_1\cup S_2$ be a good subset such that $c(S)=2$, $b(S)=1$ and $I(S)<0$. 
 Suppose that $\Gamma_S$ is obtained from 
     \[
  \begin{tikzpicture}[xscale=1.5,yscale=-0.5]
    \node (A0_0) at (0, 0) {$-2$};
    \node (A0_1) at (1, 0) {$ -(n+1)$};
    \node (A0_2) at (2, 0) {$-2$};
    \node (A0_3) at (3, 0) {$-2$};
    \node (A0_4) at (4, 0) {$-2$};
    \node (A0_6) at (6, 0) {$-2$};
    \node (A1_0) at (0, 1) {$\bullet$};
    \node (A1_1) at (1, 1) {$\bullet$};
    \node (A1_2) at (2, 1) {$\bullet$};
    \node (A1_3) at (3, 1) {$\bullet$};
    \node (A1_4) at (4, 1) {$\bullet$};
    \node (A1_5) at (5, 1) {$\dots$};
    \node (A1_6) at (6, 1) {$\bullet$};
    \path (A1_0) edge [-] node [auto] {$\scriptstyle{}$} (A1_1);
    \path (A1_4) edge [-] node [auto] {$\scriptstyle{}$} (A1_5);
    \path (A1_3) edge [-] node [auto] {$\scriptstyle{}$} (A1_4);
    \path (A1_1) edge [-] node [auto] {$\scriptstyle{}$} (A1_2);
    \path (A1_5) edge [-] node [auto] {$\scriptstyle{}$} (A1_6);
  \end{tikzpicture}
  \]
  
( where the $-2$-chain has length $n-1$ and $n\geq 2$)
    via a finite number of $-2$-final expansions performed on the leftmost component.
 Assume that there exists $v\in\mathbb{Z}^N$ linked once to a vector of each connected component of
 $S$ and orthogonal to any other vector of $S$. Then
 \begin{itemize}
  \item $v$ is linked to the central vector of the bad component of $S$ and to a final vector
  of the $-2$-chain
  \item $v\cdot v=-1$
  \item the graph $\Gamma_{S\cup\{v\}}$ is a building block of the third type
  \item $I(S)=-3$
 \end{itemize}

\end{Pro}

 \begin{proof}
 The 
   vectors corresponding to the $-2$-chain can be written as
   $$
   (e_1-e_2,e_2-e_3,\dots,e_{n-1}-e_n).
   $$
   The vectors corresponding to the bad component
   (before the $-2$-final expansions are performed) can be written as
   $$
   S_3=\{-e_{n+1}-e_{n+2},\sum_{j=1}^{n+1}e_j,-e_{n+1}+e_{n+2}\}.
   $$
   Note that the central vector is not altered by $-2$-final expansions and the same holds 
   for one of the two final vectors.
   
   \textbf{Claim}: the extra vector $v$ is linked to a final vector of the $-2$-chain.\\
   To see this 
   suppose $v$ is linked to an internal vector, say $e_i-e_{i+1}$, where $1<i<n-1$. Then we can write
   \begin{equation}\label{alfa}
    v=v'+\alpha\sum_{j=1}^{i}e_j+(1+\alpha)\sum_{j=i+1}^{n}e_j \ \ \ \textrm{where} \ \ 
   \alpha\in\mathbb{Z}\setminus\{0,-1\}.
   \end{equation}
   and $v'\cdot e_i=0$ for $1\leq i\leq n$. Now $v$ must be linked to some vector of the bad 
   component, first assume $v$ is linked to the central vector whose weight is $n+1$. In this case
   Lemma \ref{dis} implies that $\|v\|^2<n+3$. Using the expression for $v$ in \eqref{alfa} we obtain
   $$
   \|v'\|^2+i\alpha^2+(n-i)(1+\alpha)^2< n+3
   $$
   which is impossible when $\alpha\notin \{0,-1\}$. Now assume $v$ is linked to some vector, say $w$, 
   of the bad component other than the central one. If  $n\leq3$ the claim is trivial so we may assume 
   that $n>3$. It follows by Lemma \ref{dis} that
   $$
   \|v'\|^2+i\alpha^2+(n-i)(1+\alpha)^2<2+\|w\|^2.
   $$
   In particular 
   $$
   \|w\|^2>3 \ \ \ \textrm{and} \ \ \ \|w\|^2-\|v'\|^2> 2.
   $$
   We can write $w=\sum_{h=1}^{k}e_{j_h}$, where $k\geq 4$. The relevant portion of the bad component
   can be written as
   $$
   (\dots,u+e_{j_1}-e_{j_2},e_{j_2}-e_{j_3},\dots,e_{j_{k-2}}-e_{j_{k-1}}+u',
   \dots,\sum_{h=1}^{k}e_{j_h},\dots).
   $$
   In particular there is a $-2$-chain of lenght $k-3$. If $v'$ hits one of the basis vectors in this
   chain then it hits them all, and this would contradict the inequality $\|w\|^2-\|v'\|^2> 2$. 
   Therefore we may assume that $e_{j_2}\cdot v=\dots=e_{j_{k-2}}\cdot v=0$. In this situation we can
   change the bad component by removing the vectors $e_{j_2},\dots,e_{j_{k-2}}$. The relevant portion
   of this new component can be written as
   $$
   (\dots,u+e_{j_1}-e_{j_2},e_{j_2}-e_{j_{k-1}}+u',
   \dots,e_{j_1}+e_{j_2}+e_{j_{k-1}}+e_{j_k},\dots)   
   $$
   Everything we said so far holds for this new component, in particular the inequality
   $\|w\|^2-\|v'\|^2> 2$ now implies $\|v'\|^2=1$ which is easily seen to be impossible and the claim
   is proved.
   
   We can write $v=-e_1+v'$, where $v'$ does not hit any vector in the $-2$-chain. Note that if
   $v$ is linked to the central vector of the bad component then we must have $v'=0$. This is because
   $-e_1$ is linked once to a final vector of the $-2$-chain and once to the central vector 
   of the bad component and there is at most one vector in $\mathbb{Z}^N$ with this property (the conditions on $v$ can be expressed as a nonsingular
   $n\times n$ system of equations).\\
   In this case the plumbing graph corresponding to 
   $S\cup\{v\}$ is a building block of the third type. 

   Therefore in order to conclude we need to show that $v'=0$.\
   Assume $v'\neq 0$, then $v$ must be linked to some vector of the bad component, say $w$, other than
   the central one. By Lemma \ref{dis} we have 
   $\|v\|^2=1+\|v'\|^2<2+\|w\|^2$, therefore
   \begin{equation}\label{contr}
\|v'\|^2\leq\|w\|^2    
   \end{equation}
   We can write $w=\sum_{h=1}^{k}e_{j_h}$, again the relevant portion of the bad component
   can be written as
   $$
   (\dots,u+e_{j_1}-e_{j_2},e_{j_2}-e_{j_3},\dots,e_{j_{k-2}}-e_{j_{k-1}}+u',\dots,
   \sum_{h=1}^{k}e_{j_h}).
   $$
   If $k=2$ then $w=e_1-e_2$ and $v'$ can be written as 
   $$
   v'=\alpha e_{j_1}+(1+\alpha)e_{j_2} \ \ \ \ \textrm{with} 
   \ \ \ \alpha\in\mathbb{Z}\setminus\{0,-1\}
   $$
   but then $\|v'\|^2\geq 5$, which contradicts \eqref{contr}. If $k=3$, write 
   $w=e_{j_1}+e_{j_2}+e_{j_3}$. It is easy to show that again the possible expressions for $v'$
   contradicts \eqref{contr} (one needs to distinguish the three possibilities where
   $v'$ hits one, two or all of the vectors among $\{e_{j_1},e_{j_2},e_{j_3}\}$). 
   If $k\geq4$ there is
   a $-2$-chain associated to $w$ whose length is $k-3$ and either $v'$ hits every vector in this chain or it does not
   hit any of them. If $v'$ hits every vector in the $-2$-chain it is easy to see that
   this would contradict again \eqref{contr}. If $v'$ does not hit any vector in the $-2$-chain, 
   the chain can be contracted as we did before, and we are back to the case $k=3$. 
   
   \end{proof}

   \subsection{Third case: $b(S)=2$}
    In this subsection we examine the good subsets with two bad components
    satisfying the hypothesis of Theorem \ref{irreducible} and we show that they give rise to 
    building block of the fourth type.
    \begin{Pro}\label{b2}
     Let $S$ be a good subset such that $c(S)=b(S)=2$ and $I(S)<0$. Suppose that there exists
     $v\in\mathbb{Z}^N$ which is linked once to a vector of each connected component of $S$
     and is orthogonal to the other vectors of $S$. Then,
     \begin{itemize}
      \item $v$ is linked to the central vectors of each bad component of $S$
      \item $v\cdot v=-1$
      \item the graph $\Gamma_{S\cup\{v\}}$ is a building block of the fourth type
      \item $I(S)=-4$
     \end{itemize}

    \end{Pro}
\begin{proof}
   Write $S=S_1\cup S_2$. The string associated to $S$ is of the form $s_1\cup s_2$ where each 
   $s_i$ is obtained from $(2,3,2)$ via $-2$-final expansions. 
   
   First let us assume that
   the extra vector $v$ is linked
   to both the central vectors of the two bad components. Then 
   note that $det(\Gamma_{S\cup\{v\}})=0$. By Proposition \ref{cfdet} this is equivalent to
   $cf(\Gamma_{S\cup\{v\}})=0$, therefore
   $$
   0=cf(\Gamma_{S\cup\{v\}})=v\cdot v-\frac{1}{cf(\Gamma_{S_1})}-\frac{1}{cf(\Gamma_{S_2})}
   $$
   where each $\Gamma_{S_i}$ is rooted at its central vector. The graph obtained from 
   $\Gamma_{S_i}$ by removing the central vector consists of two components which are dual of 
   each other. Therefore $cf(\Gamma_{S_i})=-2$ which implies $v\cdot v=-1$.\\
   It is clear that the graph $\Gamma_{S\cup\{v\}}$ is a building block of hte fourth type. 
   To see this first blow down the extra vector 
   and split the graph along one of its trivalent vertices. In other words $\Gamma_{S\cup\{v\}}$
   is a building block of the fourth type. The fact that $I(S)=-4$ is a straightforward computation.
   
   In order to conclude we need to rule out the possibility of $v$ being linked to a noncentral 
   vector. Let $w_1\in S_1$ and $w_2\in S_2$ be the two vectors of $S$ which are linked to $v$. Suppose $w_1$
   is non central.\\ 
   \textbf{Claim:} Possibly after a sequence of contractions which do not alter the extra vector
   $v$ we may assume that $\|v\|^2\geq\|w_1\|^2$.\\
   We prove the claim in three steps which correspond to the three cases $w_1\cdot w_1=-2$, 
   $w_1\cdot w_1=-3$ and
   $w_1\cdot w_1\leq-4$.
   If $w_1\cdot w_1=-2$ we can write $w_1=e_1+e_2$ and assume $e_1\cdot v\neq 0$. 
   If $e_2\cdot v\neq0$ we are done. If $e_2\cdot v=0$ note that $e_1$ must hit some other vector 
   $u\in S_1$. Since $u\cdot v=0$ we see that $v$ must hit some basis vector other than $e_1$
   and therefore $\|v\|^2\geq 2=\|w_1\|^2$. If $w_1\cdot w_1=-3$ we may write $w_1=e_1+e_2+e_3$
   and assume $e_1\cdot v\neq 0$. If $e_2\cdot v\neq 0$ and $e_3\cdot v\neq 0$ we are done. If
   this is not the case it is easy to find two more basis vectors that hit $v$ arguing just like 
   above. If $w_1\cdot w_1\leq -4$ we may write $w_1=\sum_{i=1}^{k}e_i$, in this case there is a 
   $-2$-chain associated to $w_1$. The relevant portion of $S$ can be written as follows
   $$
   (\dots,u+e_1-e_2,e_2-e_3,e_3-e_4,\dots,e_{k-2}-e_{k-1},e_{k-1}-e_{k}+u',\dots,\sum_{i=1}^{k}e_i,\dots).
   $$
   Note that either $v$ hits every vector in the $-2$-chain or it does not hit any of them. 
   If $v$ hits every vector in the $-2$-chain the inequality $\|v\|^2\geq\|w_1\|^2$ follows easily.
   If $v$ does not hit any vector in the $-2$-chain we remove from $S$ the $-2$-chain. 
   We obtain a new subset $\widetilde{S}\subset\mathbb{Z}^{N-k+3}$. The relevant portion of $\widetilde{S}$ 
   can be written as
   $$
   (\dots,u+e_1-e_{k-1},e_{k-1}-e_{k}+u',\dots,e_1+e_{k-1}+e_{k},\dots).
   $$
   Now we can repeat the argument we used for the case $w_1\cdot w_1=-3$ and the claim is proved.
   There are two possibilities according to whether $w_2$ is central or not. If $w_2$ is not 
   central we may repeat the argument used in the claim, we obtain the inequality
   $\|v\|^2\geq \|w_1\|^2+\|w_2\|^2$ which contradicts Lemma \ref{dis}. If $w_2$ is central it is
   easy to contradict again Lemma \ref{dis}.
   \end{proof}
   \subsection{Conclusion}
   Now we are ready to prove Theorem \ref{irreducible}
   \begin{proof}(\emph{Theorem \ref{irreducible}})
    By proposition 4.10 in \cite{Lisca:2} we have $c(S)\leq2$. If $c(S)=1$ then $S$ is standard
    and the conclusion follows from Proposition \ref{c1}. If $c(S)=2$ there are four possibilities
    as explained in Proposition \ref{char}. If $b(S)=0$ the conclusion follows from Proposition
    \ref{b0}. If $b(S)=1$ the two different cases are settled by Proposition \ref{b1a} and Proposition
    \ref{b1b}. When $b(S)=2$ we can apply Proposition \ref{b2}.
   \end{proof}

\section{Orthogonal subsets}\label{3ortho}
In this section we basically fill the gap between Theorem \ref{irreducible} and Theorem \ref{technical}.
Roughly speaking, we need to remove the technical assuption $I(S)+b(S)<0$, since this is not
a property of the plumbing graph. The main result of this section is Proposition \ref{orthotris},
which shows that the subsets that are of interest for us have at most two components.
Given a linear subset $S=\{v_1,\dots,v_n\}\subset\mathbb{Z}^n$ we define, following
\cite{Lisca:2}, $p_k(S)$ as the number of $e_i$'s which hits exactly $k$ vectors in $S$.
Thinking of $S$ as a matrix $p_k(S)$ is the number of rows with $k$ nonzero entries.
Note that
\begin{eqnarray}\label{pdis}
 \sum_{i=1}^{n}p_i(S)&=&n\\
 \sum_{i=1}^{n}ip_i(S)&\leq&-\sum_{i=1}^n v_i\cdot v_i
\end{eqnarray}

 A linear subset $S=\{v_1,\dots,v_N\}\subset\mathbb{Z}^N$ is said to be \emph{orthogonal} if 
 $v_i\cdot v_j=0$ whenever $i\neq j$. 
 \begin{lem}\label{ortho}
  Let $S=\{v_1,\dots,v_n\}$ be a good orthogonal subset such that $n\geq 3$ and $I(S)=0$. 
  The following conditions are satisfied:
  \begin{enumerate}
   \item $p_3(S)=n$ and $p_i(S)=0$ for each $i\neq 3$
   \item there exists $v\in S$ such that $v\cdot v=-2$
  \end{enumerate}
 \end{lem}
 \begin{proof}
  First we prove that $p_1(S)=0$. Assume by contradiction that 
  $v_j=\alpha e_1+\pi_1(v_j)$ for some $v_j\in S$ and that no other vector in $S$
hits $e_1$. Since $S$ is irreducible we have $\pi_1(v_j)\neq 0$. 
Moreover $\pi_1(v_j)\cdot v_i=0$ for each $i\neq j$ and since the vectors
$v_1,\dots,v_{j-1},v_{j+1},\dots,v_n$ are indipendent in $\mathbb{Z}^{n-1}$ 
we must have $\pi_1(v_j)=0$ which is a contradiction, therefore $p_1(S)=0$.

Now we show that there exists $v\in S$ such that $v\cdot v=-2$.
Assume, by contradiction, that $v_i\cdot v_i\leq -3$ for each $1\leq i\leq n$.
Since $\sum_{i=1}^{n}v_i\cdot v_i=-I(S)-3n$, we see that $v_i\cdot v_i=-3$ for each $1\leq i\leq n$.
We claim that $p_2(S)=0$.  Suppose
$p_2(S)\neq 0$. We may write
$$
v_j=e_1+e_2+e_3 \ \ \ \ ; \ \ \ \ v_h=e_1-e_2+e_4
$$
where $e_1\cdot v_i=e_2\cdot v_i=0$ for each $i\notin\{j,h\}$. This is impossible because, since 
$p_1(S)=0$ both $e_3$ and $e_4$ must hit some other element of $S$. Therefore $p_2(S)=0$.
Using Equations \eqref{pdis} we obtain
$$
\sum_{i=1}^{n}ip_i(S)-3\sum_{i=1}^{n}p_i(S)\leq 0 \Rightarrow \sum_{i=3}^{k}(i-3)p_i(S)\leq 0
$$
We conclude that $p_j(S)=0$ for each $j\neq 3$ and $p_3(S)=n$. So far we have shown 
that the matrix whose columns are the $v_i$'s has exactly three non zero entries in each 
row and in each column. Note that for each $v_i$ there exists $v_j$ and $v_h$ such that 
\begin{equation}\label{Ew1w2}
 E(v_i,v_j)=E(v_i,v_h)=2 \ \ \textrm{and} \ \ E(v_i,v_j,v_h)=1
\end{equation}
Consider the following reordering on the elements of $S$ defined inductively
\begin{itemize}
 \item choose any element in $S$ and call it $v_1$
 \item choose $v_2$ so that $E(v_1,v_2)=2$
 \item choose $v_3$ so that $E(v_2,v_3)=2$ and $E(v_1,v_2,v_3)=1$
 \item choose $v_4$ so that $E(v_3,v_4)=2$ and $E(v_2,v_3,v_4)=1$
 \item $\dots$
\end{itemize}
By \eqref{Ew1w2} we may order the whole $S$ following the above procedure. It is easy to check that
for each $v_h$ there exists $e_j$ such that $e_j\cdot v_1=\dots=e_j\cdot v_{h-1}=0$ and 
$e_j\cdot v_h\neq 0$.          
In other words at each step we introduce a new basis vector. Moreover, at the first step
we introduce three basis vectors. Therefore, we would need $k+2$ basis vectors, which is impossible.

Now we show that $p_2(S)=0$. Assume by contradiction that $p_2(S)\neq 0$. Let $e_i$,$v_j,v_h$ 
be such that $e_i$ only hits $v_j$ and $v_h$ among the elements of $S$. We may assume that,
say $v_h$, is such that $v_h\cdot v_h\leq -3$ (otherwise, the set $\{v_h,v_j\}$ would be an 
irreducible component of $S$ which is impossible because $S$ is irreducible and $|S|\geq 3$).
Either $v_j\cdot v_j\leq -3$ or $v_j\cdot v_j=-2$. 
If $v_j\cdot v_j=-2$ then we may write $v_j=e_i+e_s$ and, since $e_i$ only hits $v_h$ and $v_j$,
the same conlcusion holds for $e_s$. Write
$v_h=e_i-e_s+v_h'$. Since $v_h'$ is orthogonal to any vector in $S\setminus\{v_j,v_h\}$ 
it must vanish. Therefore the subset $\{v_j,v_h\}$ is an irreducible component of $S$.
But this is impossible because $S$ is irreducible and $|S|\geq3$.
Therefore we may assume that
$v_j\cdot v_j\leq -3$.
Consider the subset
$$
S'=S\setminus \{v_h,v_j\}\cup \{\pi_i(v_j)\}
$$
It is easy to check that $S'$ is a good orthogonal subset, moreover
\begin{eqnarray*}
I(S')&=&I(S)+v_h\cdot v_h+3+v_j\cdot v_j+3-\pi_i(v_j)\cdot \pi_i(v_j)-3=\\
&=&I(S)+v_h\cdot v_h+3+v_j\cdot v_j-\pi_i(v_j)\cdot \pi_i(v_j)\leq\\
&\leq &I(S)+v_j\cdot v_j-\pi_i(v_j)\cdot \pi_i(v_j)\\
&<&I(S).
\end{eqnarray*}
In particular $I(S')<0$. By lemma 4.9 in \cite{Lisca:2} we must have $c(S')\leq 2$. Since
$|S|=c(S)\geq 3$ we have $c(S')=2$. It is easy to check that $S'$ must be of the form
$$
S'=\{e_1+e_2,e_1-e_2\}.
$$
Now it is easy to see that $S'$ cannot be expanded to a good orthogonal subset $S$ such that
$I(S)=0$. In fact there are no good orthogonal subset such that 
$(c(S),I(S))=(3,0)$. This is a contradiction and we conclude that $p_2(S)=0$. 

Finally, note that by \eqref{pdis} we have
$$
\sum_{i=1}^{k}(i-3)p_i(S)\leq 0
$$
which means that $p_i(S)=0$ for each $i\geq 4$. 
 \end{proof}
 \begin{Pro}\label{orthobis}
  Let $S$ be a good orthogonal subset such that $I(S)=0$. Then $c(S)=4$ and, up to automorphisms
  of the integral lattice $\mathbb{Z}^4$, $S$ has the following Gram matrix:
$$\left(
  \begin{array}{rrrr}
   1&1&1&0\\
   1&-1&-1&0\\
   0&1&-1&1\\
   0&-1&1&1
  \end{array}
  \right)
$$
 \end{Pro}
 \begin{proof}
It is easy to check that $|S|>2$. By Lemma \ref{ortho} we may choose $v\in S\subset\mathbb{Z}^N$ and write $v=e_1+e_2$. Moreover, since $p_3(S)=n$,
$e_1$ hits two more vectors, say $v'$ and $v''$. Since $v'\cdot v=v''\cdot v=0$ we see that
$e_2$ hits $v'$ and $v''$ as well. Writing $S$ as a matrix whose first three columns are
$v,v'v''$ we have
$$
\left(
\begin{array}{rrrrrr}
 1&1&1&0&\cdots&0\\
 1&-1&-1&0&\cdots&0\\
 0&*&*&&&\\
 \vdots&\vdots&\vdots&&&\\
 0&*&*&&&
\end{array}
\right)
$$
Where the fact that $|v'\cdot e_i|=|v''\cdot e_i|=1$ for $i=1,2$ follows from the fact that each row of the matrix above has exactly three
non zero entries and therefore $0=I(S)=\sum_{i,j}a_{i,j}^2-3n\geq 0$ and equality holds if and only if $|a_{i,j}|\leq 1$. 
Consider the subset
$$
S'=S\setminus\{v,v',v''\}\cup \{\pi_1(v'),\pi_1(v'')\}\subset\mathbb{Z}^{N-1}
$$
Note that $\pi_1(v')\cdot\pi_1(v'')=1$.
It is easy to see that $S'$ is a good subset. Moreover $(c(S'),I(S'))=(N-2,-1)$ and $b(S')=0$.
By Proposition 4.10 in \cite{Lisca:2} we have $c(S')\leq 2$, which implies $N\leq 4$. It is easy
to verify that $N\geq 4$. We conclude that $N=4$. The matrix description for $S$ follows
easily by filling the remaining entries in the above matrix.
\end{proof}
\begin{lem}\label{lemma4}
 Let $S=\{v_1,v_2,v_3,v_4\}\subset\mathbb{Z}^4$ be the subset of Proposition \ref{orthobis}.
 Let $v\in\mathbb{Z}^4\setminus\{0\}$ be such that for each $i=1,\dots,4$, we have $v\cdot v_i\in\{0,1\}$.
 Then the graph of $S\cup\{v\}$ is the following
    \[
  \begin{tikzpicture}[xscale=1.5,yscale=-0.7]
    \node (A0_0) at (0, 0) {$-2$};
    \node (A0_2) at (2, 0) {$-4$};
    \node (A1_0) at (0, 1) {$\bullet$};
    \node (A1_1) at (1, 1) {$-1$};
    \node (A1_2) at (2, 1) {$\bullet$};
    \node (A2_0) at (0, 2) {$-2$};
    \node (A2_1) at (1, 2) {$\bullet$};
    \node (A2_2) at (2, 2) {$-4$};
    \node (A3_0) at (0, 3) {$\bullet$};
    \node (A3_2) at (2, 3) {$\bullet$};
    \path (A1_0) edge [-] node [auto] {$\scriptstyle{}$} (A2_1);
    \path (A2_1) edge [-] node [auto] {$\scriptstyle{}$} (A3_2);
    \path (A2_1) edge [-] node [auto] {$\scriptstyle{}$} (A1_2);
  \end{tikzpicture}
  \]
\end{lem}
\begin{proof}
 Let $M$ be the matrix of $S$. For each $J\subset\{1,2,3,4\}$ consider the following linear system
 of equations
 $$
 ^tMv=-\sum_{j\in J}e_j.
 $$
 The lemma is equivalent to the fact that among these 16 linear systems the only ones which 
 are solvable in $\mathbb{Z}^4$ correspond to the above graph. We omit the details.  
\end{proof}
\begin{lem}\label{extrabad}
 Let $S\subset\mathbb{Z}^N$ be a good subset such that $-I(S)=b(S)=c(S)=4$. 
 There exists no vector $v\in\mathbb{Z}^N$ linked once to a vector of each 
 connected component of $S$ and orthogonal to the vectors of $S$. 
\end{lem}
\begin{proof}
Let us write $S=B_1\cup\dots\cup B_4$ where each $B_i$ is a bad component.
By definition of bad component there is a sequence of $-2$-final contractions
$$
S\searrow\dots\searrow\widetilde{S}
$$ 
such that 
$\widetilde{S}=\widetilde{B}_1\cup\dots\cup\widetilde{B}_4$ and each $\widetilde{B}_i$ is a 
bad component whose graph is of the form
  \[
  \begin{tikzpicture}[xscale=1.5,yscale=-0.5]
    \node (A0_0) at (0, 0) {$-2$};
    \node (A0_1) at (1, 0) {$a_i$};
    \node (A0_2) at (2, 0) {$-2$};
    \node (A1_0) at (0, 1) {$\bullet$};
    \node (A1_1) at (1, 1) {$\bullet$};
    \node (A1_2) at (2, 1) {$\bullet$};
    \path (A1_0) edge [-] node [auto] {$\scriptstyle{}$} (A1_1);
    \path (A1_1) edge [-] node [auto] {$\scriptstyle{}$} (A1_2);
  \end{tikzpicture}
  \]
for some $a_i\leq -3$. For each $i=1,\dots,4$, let $v_i\in B_i$ 
be the only vector of $B_i$ that is linked
once to $v$, and let $u_i$ be the central vector of $B_i$. \\
\textbf{Claim:} $v_i= u_i$ for each $i\leq 4$.
To see this we may argue exactly as in the proof of Proposition \ref{b1b}. Indeed, assume by contradiction
that $v_i\neq u_i$. Let $v'$ be the projection of $v$ onto the subspace generated by the basis vectors
that span the subset $S_i':=S_i\setminus{u_i}$. Note that $S_i'$ is a good subset consisting of two complementary components.
The vector $v'$ is linked once to a vector of a connected component and is orthogonal to all the other vectors of $S'$. We have already observed
in the proof of Proposition \ref{b1b} that such a vector does not exist. This proves the claim.

It is easy to see that $E(v,w)=0$ for each $w\in S\setminus\{u_1,\dots,u_4\}$. It follows
that $v\cdot u_i=v\cdot \overline{u}_i$ for each $i=1,\dots,4$. Let 
$\{\overline{u}_1,\dots,\overline{u}_4\}$ be the obtained by projecting each $u_i$ 
onto the subspace orthogonal to the one generated by the basis vectors
that span the subset $S_i':=S_i\setminus{u_i}$.
 Clearly $\{\overline{u}_1,\dots,\overline{u}_4\}$ is of the form described in Proposition \ref{orthobis}.
The fact that $v\cdot\overline{u}_i=1$ for each $i\leq 4$ contradicts Lemma \ref{lemma4}.
\end{proof}

\begin{Pro}\label{orthotris}
  Let $S\subset\mathbb{Z}^N$ be a good subset such that $I(S)+c(S)\leq 0$. Suppose that there exists
  $v\in\mathbb{Z}^N$ which is linked once to a vector of each connected component of $S$ and is
  orthogonal to all the vectors. Then $c(S)\leq2$.
 \end{Pro}
\begin{proof}
 By Proposition $4.10$ in \cite{Lisca:2} if $I(S)<-b(S)$ then $c(S)\leq2$. Assume by contradiction 
 that $c(S)\geq 3$.
 Then, $I(S)\geq -b(S)$ and we have
 $$
 -b(S)\leq I(S)\leq -c(S)\leq -b(S)
 $$
 therefore $I(S)=-c(S)=-b(S)$. Write 
 $S=B_1\cup\dots\cup B_k$ where each $B_i$ is a bad component. Let $S'$
 be the subset obtained from $S$ via a sequence of $-2$-final contractions so that
 each bad component has been reduced to its minimal configuration consisting of three vectors
 as in Definition \ref{badcomponent}. The graph of $S'$ has the following form
   \[
  \begin{tikzpicture}[xscale=1.2,yscale=-0.5]
    \node (A0_0) at (0, 0) {$-2$};
    \node (A0_1) at (1, 0) {$a_1$};
    \node (A0_2) at (2, 0) {$-2$};
    \node (A0_3) at (3, 0) {$-2$};
    \node (A0_4) at (4, 0) {$a_2$};
    \node (A0_5) at (5, 0) {$-2$};
    \node (A0_7) at (7, 0) {$-2$};
    \node (A0_8) at (8, 0) {$a_k$};
    \node (A0_9) at (9, 0) {$-2$};
    \node (A1_0) at (0, 1) {$\bullet$};
    \node (A1_1) at (1, 1) {$\bullet$};
    \node (A1_2) at (2, 1) {$\bullet$};
    \node (A1_3) at (3, 1) {$\bullet$};
    \node (A1_4) at (4, 1) {$\bullet$};
    \node (A1_5) at (5, 1) {$\bullet$};
    \node (A1_6) at (6, 1) {$\dots$};
    \node (A1_7) at (7, 1) {$\bullet$};
    \node (A1_8) at (8, 1) {$\bullet$};
    \node (A1_9) at (9, 1) {$\bullet$};
    \path (A1_4) edge [-] node [auto] {$\scriptstyle{}$} (A1_5);
    \path (A1_0) edge [-] node [auto] {$\scriptstyle{}$} (A1_1);
    \path (A1_1) edge [-] node [auto] {$\scriptstyle{}$} (A1_2);
    \path (A1_3) edge [-] node [auto] {$\scriptstyle{}$} (A1_4);
    \path (A1_8) edge [-] node [auto] {$\scriptstyle{}$} (A1_9);
    \path (A1_7) edge [-] node [auto] {$\scriptstyle{}$} (A1_8);
  \end{tikzpicture}
  \]
  where $a_i\leq -3$ for each $1\leq i\leq k$.
Note that $S'$ is a good subset and $(c(S'),I(S'))=(c(S),I(S))$. Since $I(S')=-k$ 
we have
\begin{equation}\label{Ibad}
\sum_{i=1}^{k}a_i=-4k. 
\end{equation}
Each bad component can be written as
  \[
  \begin{tikzpicture}[xscale=1.5,yscale=-0.5]
    \node (A0_0) at (0, 0) {$e_1+e_2$};
    \node (A0_1) at (1, 0) {$-e_2+w_i$};
    \node (A0_2) at (2, 0) {$e_2-e_1$};
    \node (A1_0) at (0, 1) {$\bullet$};
    \node (A1_1) at (1, 1) {$\bullet$};
    \node (A1_2) at (2, 1) {$\bullet$};
    \path (A1_0) edge [-] node [auto] {$\scriptstyle{}$} (A1_1);
    \path (A1_1) edge [-] node [auto] {$\scriptstyle{}$} (A1_2);
  \end{tikzpicture}
  \]
where $w\cdot e_1=w_i\cdot e_2=0$ and $w_i\cdot w_i\leq -2$.
Consider the subset $S''=\{w_1,\dots,w_k\}$. Its graph is
  \[
  \begin{tikzpicture}[xscale=1.5,yscale=-0.5]
    \node (A0_0) at (0, 0) {$a_1+1$};
    \node (A0_1) at (1, 0) {$a_2+1$};
    \node (A0_3) at (3, 0) {$a_k+1$};
    \node (A1_0) at (0, 1) {$\bullet$};
    \node (A1_1) at (1, 1) {$\bullet$};
    \node (A1_2) at (2, 1) {$\dots$};
    \node (A1_3) at (3, 1) {$\bullet$};
  \end{tikzpicture}
  \]

Note that this is a good orthogonal subset
and by Equation \eqref{Ibad} we have
$$
\sum_{i=1}^{k}w_i\cdot w_i =\sum_{i=1}^{k}(a_i+1)=-3k.
$$
Therefore the subset $S''$ satisfies the hypothesis of Lemma \ref{ortho} and Proposition 
\ref{orthobis}. In particular $k=4$. 

The proof is concluded by using Lemma \ref{extrabad}, which shows that there exist
no subset $S$ and a vector $v$ with the above properties.
\end{proof}
\section{Conclusion of the proof}\label{3conclusion}
Putting together Theorem \ref{irreducible} and Proposition \ref{orthotris} we can finally
prove Theorem \ref{technical}.
\begin{proof}(\emph{Theorem \ref{technical}})\\
 Let $S=S_1\cup\dots\cup S_k$
 be the decomposition of $S$ into its irreducible components. We may write
 $v=v_1+\dots+v_k$ so that each $v_i$ is the projection of $v$ onto the subspace 
 that corresponds to $S_i$. From \eqref{crucialhyptec} we obtain 
 $$
 I(S)+c(S)=\sum_{i=1}^{k}I(S_i)+c(S_i)\leq 0.
 $$
 We may choose an irreducible component $S_j$ such that $I(S_j)+c(S_j)\leq 0$.
 By Proposition \ref{orthotris} we have $c(S_j)\leq 2$. Moreover $I(S_j)+b(S_j)\leq I(S_j)+c(S_j)\leq 0$
 
 We claim that $I(S_j)+b(S_j)<0$. Assume by contradiction that
 $I(S_j)=-b(S_j)=-c(S)=-2$ and write $S_j=B_1\cup B_2$. Since it is 
 easy to check that for every bad component $B$ we have $I(B)\geq -2$, we may assume that
 one of the following holds
 \begin{itemize}
  \item $I(B_1)=I(B_2)=-1$
  \item $I(B_1)=-2$ and $I(B_2)=0$
 \end{itemize}
Arguing as in the proof of Proposition \ref{orthotris} we would get orthogonal subsets
whose associated graph is either 
  \[
  \begin{tikzpicture}[xscale=1.0,yscale=-0.5]
    \node (A0_0) at (0, 0) {$-3$};
    \node (A0_1) at (1, 0) {$-3$};
    \node (A1_0) at (0, 1) {$\bullet$};
    \node (A1_1) at (1, 1) {$\bullet$};
  \end{tikzpicture}
  \]
or
  \[
  \begin{tikzpicture}[xscale=1.0,yscale=-0.5]
    \node (A0_0) at (0, 0) {$-2$};
    \node (A0_1) at (1, 0) {$-4$};
    \node (A1_0) at (0, 1) {$\bullet$};
    \node (A1_1) at (1, 1) {$\bullet$};
  \end{tikzpicture}
  \]
It is easy to check that none of these configurations are realizable, and the claim is proved.

 We can now apply Theorem \ref{irreducible}. The graph $\Gamma_{S_j\cup\{v_j\}}$ is a
 building block. Moreover it is easy to check that \eqref{crucialhyptec} holds for the subset
 $S\setminus S_j$ so that we may iterate the argument above with all the irreducible components of
 $S$, and we are done.
 \end{proof}


\begin{thebibliography}{}
  \addcontentsline{toc}{section}{Bibliography}
    \providecommand\bibmarginpar{\leavevmode\marginpar}
\def\urlstyle#1{{\tt #1}}

\bibitem{Aceto}
\textbf{P\, Aceto}, 
\emph{Arborescent link concordance}
	In preparation  
	
\bibitem{Bonahon}
\textbf{F\, Bonahon , L\, C\, Siebenmann}, 
\emph{New geometric splittings of classical knots and the classification and symmetries of arborescent knots}
	Version June 12, 2010  

	\bibitem{Wu}
\textbf{M\, Brittenham, Y\, Wu}, \emph{The classification of exceptional Dehn surgeries on 2-bridge knots}
Comm. Anal. Geom. 9, (1995) 97-113

\bibitem{Casson}
\textbf{A\, J\, Casson ,J\, L\, Harer}, \emph{Some homology lens spaces which bound rational homology balls}
Pac. J. Math. 96, 1, (1981)

	\bibitem{Donald:1}
\textbf{A\, Donald}, \emph{Embedding Seifert manifolds in $S^4$}, Trans. Amer. Math. Soc. 367 (2015), 559-595

\bibitem{Eisenbud}
\textbf{D\, Eisenbud, W\, D\, Neumann}, \emph{Three-dimensional Link Theory and Invariants of Plane Curve Singularities}
Annals of mathematics studies 110, Princeton University Press, 1985

\bibitem{Stern}
\textbf{R\, Fintushel, R\, Stern}, \emph{Pseudofree orbifolds}
Ann. Math. 122 (1985) 335-364

\bibitem{GompfStip}
\textbf{R\, E\, Gompf, A\, Stipsicz}, \emph{4-manifolds and Kirby calculus}
Graduate studies in mathematics, American Mathematical Soc.

\bibitem{Kauffman}
\textbf{L\, H\, Kauffman}, \emph{On knots}
	Annals of Mathematics studies, Princeton University Press (1987) 
	
	\bibitem{Kauffman:2}
\textbf{L\, H\, Kauffman, L\, R\, Taylor}, \emph{Signature of links}
	Trans. Amer. Math. Soc. 216  (1976) 351-366 
	
\bibitem{Kim}
\textbf{P\, Kim, J\, Tollefson}, \emph{Splitting the PL involutions of nonprime 3-manifolds}
	Michigan Math. J. 27 (1980), 259-274 

\bibitem{Lecuona:1}
\textbf{A\, Lecuona}, \emph{On the slice-ribbon conjecture for Montesinos knots}
	Trans. Amer. Math. Soc. 364 (2012), 233-285 
	
\bibitem{Lee:1}
\textbf{R\, Lee\, S H Weintraub}, \emph{On the homology of double branched covers}
	Proc. Amer. Math. Soc. 123 (1995), 1263-1266 
	
\bibitem{Lisca:1}
\textbf{P\, Lisca}, \emph{Lens spaces, rational balls and the ribbon conjecture}, 
	Geom. Topol. 11 (2007) 429-472
	
\bibitem{Lisca:2}
\textbf{P\, Lisca}, \emph{Sums of lens spaces bounding rational balls}, 
	Algebr. Geom. Topol. 7 (2007) 2141-2164
	
\bibitem{Moser}
\textbf{L\, Moser}, \emph{Elementary surgery along a torus knot}
Pac. J. Math. 38, 3, (1971)
	
\bibitem{Neumann:1}
\textbf{W\, Neumann}, \emph{A calculus for plumbing applied to the topology of complex 
surface singularities and degenerating complex curves}
	Trans. Amer. Math. Soc. 268 (1981), 299-342 
	
\bibitem{Neumann:2}
\textbf{W\, Neumann}, \emph{On bilinear forms represented by trees}
	Bull. Austral. Math. Soc. 40 (1989), 303-321 
	
\bibitem{Neumann:3}
\textbf{W\, Neumann, F Raymond}, \emph{Seifert manifolds, plumbing, $\mu$-invariants
and orientation reversing map}
	 Algebr. Geom. Topol. Lecture Notes in Mathematics 664 (1978) 163-196

\bibitem{Saveliev}
\textbf{N\, Saveliev}, \emph{Dehn surgery along torus knots}
Topology and its applications 83 (1998), 193-202

\end{thebibliography}
\end{document}